\documentclass[11pt]{amsart}

\usepackage{amsmath}
\usepackage{amssymb}

\usepackage{enumerate,enumitem}
\setenumerate{nolistsep}
\setitemize{nolistsep}
\usepackage[square, numbers, comma, sort&compress]{natbib}
\usepackage[final]{pdfpages}
\usepackage{graphicx}
\usepackage{esint}
\usepackage{latexsym}
\usepackage{hyperref}
\usepackage{cleveref}
\usepackage{autonum}
\usepackage{mathtools}

\usepackage{lipsum}

\newcommand{\R}{{\mathbb R}}       
\newcommand{\DD}{{\mathcal D}}

\newcommand{\EE}{{\mathcal E}}

\newcommand{\UU}{{\mathcal U}}
\newcommand{\cA}{{\mathcal A}}
\newcommand{\HH}{{\mathcal H}}
\newcommand{\LL}{{\mathcal L}}

\newcommand{\NB}{{\mathcal {NB}}}
\newcommand{\FL}{{\mathcal {FL}}}
\newcommand{\RR}{{\mathcal R}}
\newcommand{\ve}{{\varepsilon}}
\newcommand{\tve}{{\tilde\varepsilon}}
\newcommand{\tmu}{{\tilde\mu}}
\newcommand{\tsigma}{{\tilde\sigma}}
\newcommand{\tchi}{{\tilde\chi}}
\newcommand{\wt}[1]{{\widetilde{#1}}}
\newcommand{\wh}[1]{{\widehat{#1}}}
\newcommand{\rf}[1]{{\eqref{#1}}}

\usepackage{tikz}
\usetikzlibrary{graphs, angles, quotes,patterns,hobby}
\usepackage{pgfplots}
\pgfplotsset{compat=1.6}

\DeclareMathOperator{\D}{\mathcal{D}}
\DeclareMathOperator{\Div}{div}
\DeclareMathOperator{\diam}{diam}
\DeclareMathOperator{\n1}{\nabla_1}
\DeclareMathOperator{\supp}{supp}
\DeclareMathOperator{\Lip}{Lip}
\DeclareMathOperator{\dist}{dist}
\DeclareMathOperator{\Rn1}{\mathbb{R}^{n+1}}
\DeclareMathOperator{\E}{\mathcal{E}}

\DeclareMathOperator{\Stop}{\mathsf{Stop}}
\DeclareMathOperator{\Tree}{\mathsf{Tree}}
\DeclareMathOperator{\Nice}{\mathsf{Nice}}
\DeclareMathOperator{\Top}{\mathsf{Top}}

\DeclareMathOperator{\Ch}{\mathsf{Ch}}
\DeclarePairedDelimiter\norm{\lVert}{\rVert}

\theoremstyle{plain}
\newtheorem{theor}{Theorem}[section]

\theoremstyle{plain}
\newtheorem*{theor*}{Theorem}
\theoremstyle{plain}
\newtheorem*{theora*}{Theorem A}
\theoremstyle{plain}
\newtheorem*{theorb*}{Theorem B}

\theoremstyle{plain}
\newtheorem{prop}{Proposition}[section]

\theoremstyle{remark}
\newtheorem{rem}{\bf Remark}[section]
\newtheorem*{rem*}{\bf Remark}

\theoremstyle{definition}
\newtheorem{defin}{Definition}[section]

\theoremstyle{plain}
\newtheorem{coroll}{Corollary}[section]

\theoremstyle{remark}

\theoremstyle{plain}
\newtheorem{lemm}{Lemma}[section]

\numberwithin{equation}{section}

\usepackage{hyperref}
\hypersetup{
    colorlinks=true,
    linkcolor=blue,
    filecolor=blue,      
    urlcolor=blue,
    citecolor=red,
}

\def\Xint#1{\mathchoice
{\XXint\displaystyle\textstyle{#1}}%
{\XXint\textstyle\scriptstyle{#1}}%
{\XXint\scriptstyle\scriptscriptstyle{#1}}%
{\XXint\scriptscriptstyle\scriptscriptstyle{#1}}%
\!\int}
\def\XXint#1#2#3{{\setbox0=\hbox{$#1{#2#3}{\int}$ }
\vcenter{\hbox{$#2#3$ }}\kern-.58\wd0}}

\def\avint{\Xint-}

\begin{document}


\title[Gradients of Single Layer
Potentials and Uniform Rectifiability]{$L^2$-boundedness of Gradients of Single Layer
Potentials and Uniform Rectifiability}


\author{Laura Prat}

\address{Laura Prat
\\
Departament de Matem\`atiques
\\
Universitat Aut\`onoma de Barcelona
\\
08193 Bellaterra (Barcelona), Catalonia.
}
\email{laurapb@mat.uab.cat}

\author{Carmelo Puliatti}
\address{Carmelo Puliatti
\\
Departament de Matem\`atiques and BGSMath
\\
Universitat Aut\`onoma de Barcelona
\\
08193 Bellaterra (Barcelona), Catalonia.
}
\email{puliatti@mat.uab.cat}

\author{Xavier Tolsa}

\address{Xavier Tolsa
\\
ICREA, Passeig Llu\'{\i}s Companys 23 08010 Barcelona, Catalonia\\
 Departament de Matem\`atiques, and BGSMath
\\
Universitat Aut\`onoma de Barcelona
\\
08193 Bellaterra (Barcelona), Catalonia.
}
\email{xtolsa@mat.uab.cat} 

\thanks{All the authors were partially supported by 2017-SGR-0395 (Catalonia) and MTM-2016-77635-P (MICINN, Spain).
 X.T. was also partially supported by the ERC grant 320501 of the European Research Council (FP7/2007-2013) and MDM-2014-044 (MINECO, Spain). C.P. is also partially supported by the grant MDM-2014-0445  (MINECO through the Mar\'ia de Maeztu Programme for Units of Excellence in R\&D, Spain).
}

\begin{abstract}
Let $A(\cdot)$ be an $(n+1)\times (n+1)$ uniformly elliptic matrix with H\"older continuous real coefficients and let $\EE_A(x,y)$
be the fundamental solution of the PDE $\mathrm{div} A(\cdot) \nabla u =0$ in $\R^{n+1}$. Let $\mu$ be a compactly supported $n$-AD-regular measure in $\R^{n+1}$ and consider the associated operator
$$T_\mu f(x) = \int \nabla_x\EE_A(x,y)\,f(y)\,d\mu(y).$$
We show that if $T_\mu$ is bounded in $L^2(\mu)$, then $\mu$ is uniformly $n$-rectifiable. This extends the solution of the codimension $1$ David-Semmes problem for the Riesz transform to the gradient of the single layer potential.
Together with a previous result of Conde-Alonso, Mourgoglou and Tolsa, this shows that, given $E\subset\R^{n+1}$ with
finite Hausdorff measure $\HH^n$, 
if $T_{\HH^n|_E}$ is bounded in $L^2(\HH^n|_E)$, then $E$ is $n$-rectifiable.
Further, as an application we show that if the elliptic measure associated to the above PDE is absolute continuous with respect to surface measure, then it must be rectifiable, analogously to what happens with harmonic measure.

\end{abstract}

 \maketitle



\tableofcontents

\section{Introduction}

The purpose of this paper is to extend the solution of the codimension $1$ David-Semmes problem for the Riesz transform
to operators defined by gradients of singular layer potentials associated with elliptic PDE's in divergence form with H\"older continuous coefficients. The single layer potential and its gradient play an important role  in the
solvability of this type of equations and also in the study of the corresponding elliptic measure.
Recall that the David-Semmes problem deals with the connection between the Riesz transforms and rectifiability. 
This was solved in 1996 for the $1$-dimensional Riesz transform (or equivalently, for the Cauchy transform) by Mattila, Melnikov and Verdera in \cite{MMV} by using the connection between
Menger curvature and the Cauchy kernel.
The case of codimension $1$ was solved more recentely by Nazarov, Tolsa and Volberg in \cite{NTV_acta} by different methods, relying on the harmonicity of the codimension $1$ Riesz kernel. The David-Semmes problem is still open in the remaining dimensions $n\in [2,d-2]$ in $\R^d$.


Given a Borel measure  $\mu$  in $\mathbb{R}^{d}$ (from now on we assume all measures to be Borel in the paper), recall that its $n$-dimensional Riesz transform is 
defined by
\begin{equation}
\mathcal{R}^n\mu (x)=\int \frac{x-y}{|x-y|^{n+1}}\,d\mu(y),
\end{equation}
whenever the integral makes sense. Also, for a function $f\in L^1_{loc}(\mu)$, we write $\mathcal{R}^n_\mu f(x)= \mathcal{R}^n(f\mu) (x)$.


The $n$-dimensional Hausdorff measure is denoted by $\mathcal{H}^n$.
A set $E\subset \R^d$ is called $n$-rectifiable if there are Lipschitz maps $f_i:\mathbb{R}^{n}\to\R^d$, $i=1,2,\ldots$, such that
\begin{equation}
\mathcal{H}^n \Big( E\setminus\bigcup_i f_i(\mathbb{R}^n)\Big)=0.
\end{equation}
A set $F$ is called purely $n$-unrectifiable if $\mathcal{H}^n(F\cap E)=0$ for every $n$-rectifiable set $E$. As for sets, one can define a notion of rectifiabilty also for measures: a measure $\mu$ is said to be $n$-rectifiable if it vanishes outside an $n$-rectifiable set $E\subset\R^d$ and, moreover, it is absolutely continuous with respect to $\mathcal{H}^n|_{E}$.

In most of this work we deal with measures that present a certain degree of regularity.
A measure $\mu$ in $\R^d$ is called $n$-AD-regular (or just AD-regular or Ahlfors-David regular) if there exists some
constant $C_0>0$ such that
$$C_0^{-1}r^n\leq \mu(B(x,r))\leq C_0\,r^n\quad \mbox{ for all $x\in
\supp(\mu)$ and $0<r\leq \diam(\supp(\mu))$.}$$
 A set $E\subset\R^d$ is $n$-AD-regular if the measure $\HH^n|_E$ is $n$-AD-regular.

The set $E$ is  called uniformly  $n$-rectifiable if it is 
$n$-AD-regular and
there exist $\theta, M >0$ such that for all $x \in E$ and all $r>0$ 
there is a Lipschitz mapping $g$ from the ball $B_n(0,r)$ in $\R^{n}$ to $\R^d$ with $\Lip(g) \leq M$ such that$$
\HH^n (E\cap B(x,r)\cap g(B_{n}(0,r)))\geq \theta r^{n}.$$
A measure $\mu$ is called uniformly $n$-rectifiable if it is $n$-AD-regular and its support is uniformly $n$-rectifiable.

It is easy to check that if a set (or a measure) is uniformly $n$-rectifiable, then it is also $n$-rectifiable. The converse implication
is false. In fact, uniform $n$-rectifiability is a quantitative version of the notion of $n$-rectifiability introduced by David and Semmes \cite{david_semmes}.  One of their motivations to introduce this notion was the desire to find a good framework where one can study the $L^2(\mu)$ boundedness of singular integral operators.
Indeed, they showed that if $\mu$ is $n$-AD-regular, the fact that $\mu$ is uniformly $n$-rectifiable is equivalent to the
$L^2(\mu)$-boundedness of a sufficiently big class of singular integral operators with an odd and smooth enough Calder\'on-Zygmund kernel. In particular, if $\mu$ is uniformly $n$-rectifiable, then the $n$-dimensional Riesz transform $\mathcal{R}^n_\mu$
is bounded in $L^2(\mu)$.

 The David-Semmes problem consists in proving that the converse statement holds. That is, that 
under the background assumption of $n$-AD-regularity on the measure $\mu$, the $L^2(\mu)$ boundedness of  the  Riesz transform $\mathcal{R}^n_\mu$ implies the uniform $n$-rectifiability of $\mu$. As mentioned above, the answer is only known (and positive) in the cases $n=1$ and $n=d-1$ in $\R^d$, by  \cite{MMV} and \cite{NTV_acta}, respectively.

The solution of the David-Semmes problem has had important applications to the solution of other relevant questions.
In the dimension $1$ case in the plane, this has played an essential role in the geometric characterization 
of removable singularities
 for bounded analytic functions, and in particular in the solution of 
Vitushkin's conjecture for sets with finite length by David \cite{David-vitus}. 
In the codimension $1$ case, 
the analogous result involving the removable singularities for Lipschitz harmonic functions has been solved
in \cite{NTV_publ}.
Other remarkable applications of the solution of the David-Semmes problem in codimension $1$ deal with the metric and geometric properties of harmonic measure. In particular, this is a key ingredient in the recent solution of two
problems about harmonic measure raised by Christopher Bishop in the early 1990's \cite{Bishop-conjectures}. The first one is the fact that the mutual absolute continuity of harmonic measure for an open set $\Omega\subset\R^{n+1}$ with respect to the surface measure $\HH^n$ in a subset of $\partial\Omega$ implies the rectifiability of that subset \cite{AHM3TV}. The second one is the solution of the so called two-phase problem 
in the works \cite{AMT} and \cite{AMTV}.

The results just mentioned also make sense for solutions of elliptic equations and for the elliptic measure. So in view of potential applications, it is natural to try to extend the solution of the David-Semmes problem to gradients of single layer potentials, which are the analogues of the Riesz transform in the context of elliptic PDE's.

Next we introduce the precise ellipic PDE's in which we are interested.
 Let $A=(a_{ij})_{1\leq i,j \leq n+1}$ be an $(n+1)\times (n+1)$ matrix whose entries $a_{ij}\colon\R^{n+1} \to \R$  are measurable functions in $L^\infty(\R^{n+1})$. Assume also that there exists $\Lambda>0$ such that
\begin{align}\label{eqelliptic1}
&\Lambda^{-1}|\xi|^2\leq \langle A(x) \xi,\xi\rangle,\quad \mbox{ for all $\xi \in\R^{n+1}$ and a.e. $x\in\R^{n+1}$,}\\
&\langle A(x) \xi,\eta \rangle  \leq\Lambda |\xi| |\eta|, \quad \mbox{ for all $\xi, \eta \in\R^{n+1}$ and a.e. $x\in\R^{n+1}$.} \label{eqelliptic2}
\end{align}
We consider the elliptic equation
\begin{equation}\label{eq:ellipticeq}
L_A u(x)\coloneqq -\mathrm{div}\left(A(\cdot) \nabla u (\cdot) \right)(x)=0,
\end{equation}
which should be understood in the distributional sense.
We say that a function $u \in W^{1,2}_{\rm loc}(\Omega)$ is a {\it solution} of \eqref{eq:ellipticeq} or {\it $L_A$-harmonic} in an open set $\Omega \subset \R^{n+1}$ if
$$
\int A \nabla u \cdot \nabla \varphi = 0, \quad \mbox{ for all $\varphi \in C_c^\infty(\Omega)$.}
$$

We denote by $\mathcal{E}_A(x,y)$, or just by $\mathcal{E}(x,y)$ when the matrix $A$ is clear from the context, the {\it fundamental solution} for $L_A$ in $\R^{n+1}$, so that $L_A \mathcal{E}_A(\cdot,y) = \delta_y$ in the distributional sense, where $\delta_y$ is the Dirac mass at the point $y \in \R^{n+1}$. For a construction of the fundamental solution under the assumption \eqref{eqelliptic1} and \eqref{eqelliptic2} on the matrix $A$ we refer to \cite{HK}.
For a measure $\mu$, the function $f(x)=\int \mathcal{E}_A(x,y)\,d\mu(y)$  is usually known as the {\it single layer potential} of $\mu$.
We consider the singular integral operator $T$ whose kernel is
\begin{equation}\label{eq:Kdef}
K(x,y) = \nabla_1 \mathcal{E}_A(x,y)
\end{equation}
(the subscript $1$ means that we take the gradient with respect to the first variable), so that \begin{equation}\label{eq:Tmudef}
T\mu(x) = \int K(x,y) \,d\mu(y)
\end{equation}
when $x$ is away from $\mathrm{supp}(\mu)$.   That is, $T\mu$ is the gradient of the single layer potential of $\mu$.

Given a function $f\in L^1_{loc}(\mu)$,
we set also
\begin{equation}\label{eq:Tfdef}
T_\mu f(x) = T(f\,\mu)(x) = \int K(x,y) f(y)\,d\mu(y),
\end{equation}
and, for $\ve>0$, we consider the $\ve$-truncated version
$$
T_{\ve}\mu(x) = \int_{|x-y|>\ve} K(x,y) \,d\mu(y).
$$
We also write $T_{\mu,\ve} f(x) = T_\ve (f\mu)(x)$. We say that the operator $T_\mu$ is bounded in $L^2(\mu)$ if the operators $T_{\mu,\ve}$ are bounded in $L^2(\mu)$ uniformly on $\ve>0$.

In the special case when $A$ is the identity matrix, $-L_A$ is the Laplacian and $T$ is the $n$-dimensional Riesz transform up to a constant factor depending only on the dimension $n$.

Without any hypothesis on the smoothness of the coefficients of the matrix $A$, one cannot expect the kernel 
$K(\cdot,\cdot)$ in \rf{eq:Kdef} to be of Calder\'on-Zygmund type, and thus we need to impose  some regularity condition on $A$. 
We say that the matrix $A$ is H\"older continuous with exponent $\alpha$ (or briefly $C^\alpha$ continuous), if there exists $\alpha>0$ and $C_h>0$ such that
\begin{equation}\label{eq:Holdercont}
|a_{ij}(x)-a_{ij}(y)| \leq C_h |x-y|^\alpha\quad \mbox{ for all $x,y \in \mathbb{R}^{n+1}$ and $1 \leq i,j \leq n+1$.}
\end{equation}
Under this assumption on the coefficients, the kernel $K(\cdot,\cdot)$ turns out to be locally of Calder\'on-Zygmund type (see Lemma \ref{lemcz} for more details). However, we remark that in general $K(\cdot,\cdot)$ is neither homogeneous (of degree $-n$)  nor antisymmetric (even locally).

\vspace{1mm}
Our main result is the following.

\begin{theor} \label{teo1}
Let $\mu$ be a compactly supported $n$-AD-regular measure in $\mathbb R^{n+1}$.  Let $A$ be an elliptic matrix satisfying \eqref{eqelliptic1}, \eqref{eqelliptic2} and \eqref{eq:Holdercont}, and let $T_\mu$ be the associated operator given by \eqref{eq:Tfdef}.
The operator $T_\mu$ is bounded in $L^2(\mu)$ if and only if $\mu$ is uniformly $n$-rectifiable.
\end{theor}

The assumption that $\mu$ is compactly supported in the theorem above is necessary and it is due to the fact that the $C^\alpha$ continuity of the matrix $A$ is a property which is not scale invariant.
We also remark that it is already known that $T_\mu$ is bounded in $L^2(\mu)$ if $\mu$ is uniformly $n$-rectifiable (see Theorem 2.5 from \cite{CMT}). Our contribution is the converse statement. 

Theorem \ref{teo1} should be compared to a recent result obtained by Conde-Alonso, Mourgoglou and Tolsa in \cite{CMT}, which in a sense complements our theorem. The precise result is the following.

\begin{theora*}[\cite{CMT}]
Let $\mu$ be a non-zero Borel measure in $\mathbb R^{n+1}$.  Let $A$ be an elliptic matrix satisfying \eqref{eqelliptic1}, \eqref{eqelliptic2} and \eqref{eq:Holdercont}, and let $T_\mu$ be the associated 
operator.
Suppose that
 the upper density $\limsup_{r\to0}\frac{\mu(B(x,r))}{(2r)^n}$  is positive $\mu$-a.e.\ in $\R^{n+1}$, and the lower density $\liminf_{r\to0}\frac{\mu(B(x,r))}{(2r)^n}$ vanishes $\mu$-a.e.\ in $\R^{n+1}$. Then $T_\mu$ is not bounded in $L^2(\mu)$.
\end{theora*}

Notice that, in the case $\mu=\HH^n|_E$, the assumptions on the upper and lower densities in the theorem above
imply that $E$ is purely $n$-unrectifiable. 
This theorem extends an analogous result proved previously by Eiderman, Nazarov and Volberg \cite{ENV} for the $n$-dimensional Riesz transform.

Our proof of Theorem \ref{teo1} follows the same scheme as the proof of the corresponding result for
the Riesz transform in \cite{NTV_acta}. In particular, it also relies on a variational argument which
uses the fact that $L_A$-harmonic functions satisfy a maximum principle. It also uses the so-called BAUP criterion
of David and Semmes \cite[p.\ 139]{david_semmes}. However, there are some important
differences between our arguments and the ones in \cite{NTV_acta}. An important one is that we use
a martingale difference decomposition in terms of the David-Semmes lattice, instead of the quasiorthogonality arguments
in \cite{NTV_acta}. We think that using a martingale decomposition makes the whole construction much more transparent. Further, the quasiorthogonality arguments seem to require the antisymmetry of the kernel, which does not hold in our case. On the other hand, the fact that the matrix $A$ is non-constant makes
our arguments and estimates more involved and technical. For example, the reflection trick required to apply later the variational argument is more delicate, as well as  
 the  approximation techniques used
to transfer estimates among different measures (see Section \ref{choosingdelta} below).
 The reader can find the scheme of the proof of Theorem \ref{teo1} at the end of Section \ref{sec4}.

By combining Theorem \ref{teo1} and  Theorem A from \cite{CMT}, we are also able to derive the following rectifiability result for general sets.

\begin{theor}\label{teo2}
Let $E\subset\Rn1$ be a compact set with $\mathcal{H}^n(E)<\infty$. 
Let $A$ and $T$ be as in Theorem \ref{teo1}.
If $T_{\mathcal{H}^n{|_E}}$ is bounded in $L^2(\mathcal{H}^n{|_E})$, then $E$ is $n$-rectifiable.
\end{theor}

The analogous result in case that $A$ is the identity and $T$ is the $n$-dimensional Riesz transform
(modulo some constant factor) has been proved in \cite{NTV_publ}. Theorem \ref{teo2} is proved almost in the same way as in \cite{NTV_publ}: by an argument inspired by a covering theorem of Pajot,
one decomposes $\mu=\HH^n|_E$ into a measure $\mu_0$ with vanishing lower density and a countable collection of measures $\mu_k$ such that each $\mu_k$ can be extended to another $n$-AD-regular measure $\wt \mu_k$
such that $T_{\wt\mu_k}$ is bounded in $L^2(\wt\mu_k)$.
Theorem A implies that $\mu_0\equiv0$, and Theorem \ref{teo1} implies that each measure $\wt\mu_k$ is uniformly $n$-rectifiable. The only specific feature of the Riesz kernel that is used in
\cite{NTV_publ} is its antisymmetry. As mentioned above, we cannot ensure that the kernel $K(\cdot,\cdot)$
is antisymmetric. However, this is not a problem in our case because by Lemma \ref{lemantisym} below it turns out that, for
any measure $\mu$ with growth of degree $n$ (see \rf{eqrwth56} for the definition), $T_\mu$ is bounded in $L^2(\mu)$ if and only if the operator 
$T_\mu^{(a)}$ associated with the antisymmetric part of $K(\cdot,\cdot)$ is bounded in $L^2(\mu)$. 
Then, in order to prove Theorem \ref{teo2} we just apply the same arguments as  in \cite{NTV_publ} to 
$T_\mu^{(a)}$ instead of the $n$-dimensional Riesz transform.

\vspace{1mm}

An important application of Theorem \ref{teo2} deals with elliptic measure. Given a Wiener regular open set $\Omega \subset \R^{n+1}$, the elliptic measure (or $L_A$-harmonic measure) for $\Omega$ with pole at $p\in \Omega$ is the probability measure
$\omega_{L_A}^p$ supported on $\partial\Omega$ such that, for every $f \in C_0(\partial \Omega)$,
$\int f\,d\omega_{L_A}^p$ equals the value at $p$ of the $L_A$-harmonic extension of $f$ to $\Omega$.
For a basic reference on elliptic measure, see \cite{Kenig-CBMS}, and for some additional background see \cite[Section 2.4]{AGMT}, for example. Analogously to harmonic measure, the connection between the metric properties of elliptic measure
and the geometric properties of $\Omega$ (in particular, the rectifiability of $\partial\Omega$) has been a subject
of intense investigation in the last years. See for example the works \cite{ABHM}, \cite{AGMT}, \cite{HKMP}, \cite{HoMiT},
\cite{HMT}, \cite{KKiPT}. Our result in connection with elliptic measure is the following.

\begin{theor}\label{teo3}
Let $n\geq 2$ and let $A$ be an elliptic matrix satisfying \eqref{eqelliptic1}, \eqref{eqelliptic2} and \eqref{eq:Holdercont}.
Let $\Omega\subsetneq\R^{n+1}$ be a bounded open  connected Wiener regular set, let $p\in\Omega$, and let $\omega_{L_A}^p$ be the elliptic measure in $\Omega$ associated with $L_A$, with pole $p$.
Suppose that there exists a set $E\subset\partial\Omega$ such that $0<\HH^n(E)<\infty$ 
and that the elliptic measure $\omega_{L_A}^p|_E$ is absolutely continuous with respect to $\HH^n|_{E}$. Then
 $\omega_{L_A}^p|_E$ is $n$-rectifiable.
\end{theor}

Remark that  $\omega_{L_A}^p|_E$ being $n$-rectifiable means that it is concentrated on an $n$-rectifiable
set and it is absolutely continuous with respect to $\HH^n|_E$.
In the case of $-L_A$ being the Laplacian and $\omega_{L_A}$ the harmonic measure, the same result has been proved in \cite{AHM3TV}, and it can be considered as a kind of converse of the famous Riesz brothers theorem on harmonic measure in planar simply connected domains. The preceding result follows from Theorem \ref{teo2} by  essentially the same arguments as the ones for harmonic measure in \cite{AHM3TV}. Nevertheless, for the
reader's convenience the arguments are sketched in the final Section \ref{sec12*}.


\section{Preliminaries}

\subsection{General notation}

We use the standard notation $a\lesssim b$ if there is a fixed constant $C>0$ (depending on other fixed parameters, such as the ambient dimension)
such that $a\leq C b$. To make the dependence of the constant on a parameter $t$ explicit, we will also write $a\lesssim_t b.$ We will also write $b\gtrsim a$ if $a\lesssim b$ and $a\approx b$ if both $a\lesssim b$ and $b\lesssim a.$

We use the notation $B(x,r)$ for the open ball in $\Rn1$ centered at $x$ of radius $r$. For a ball $B = B(x,r)$ and $a>0$
we write $aB = B(x,ar)$ for the centered rescaling of the ball.
For $0<r<R$, we denote by
\begin{equation}
A(x,r,R)\coloneqq \{y\in\Rn1:r<|x-y|<R\}
\end{equation}
the open annulus centered at $x$ with radii $r$ and $R$. Also, given $t>0$ and a set $E$, we write
\begin{equation}
\mathcal{U}_t(E)\coloneqq \{x\in\Rn1:\dist(x,E)\leq t\}
\end{equation}
for the closed $t$-neighborhood of $E$. 

Given a measure $\mu$, we write $\langle \cdot, \cdot\rangle_\mu$ for the scalar product in $L^2(\mu)$ and 
$m_{\mu,E}f\coloneqq \mu(E)^{-1}\int_E fd\mu$ for the $\mu$-average of a measurable function $f$ on a set $E$. 

{We denote by $AD(n,C_0,\mathbb{R}^{d})$ the set of $n$-$AD$-regular measures on $\mathbb{R}^{d}$ with constant $C_0$.
We say that $\mu$ has growth of degree $n$ (or $n$-growth) if
\begin{equation}\label{eqrwth56}
\mu(B(x,r))\leq C\,r^n\quad \mbox{ for all $x\in\R^{n+1}$.}
\end{equation}
 We denote the Lebesgue measure in $\R^{n+1}$ by $\LL^{n+1}$. Quite often we will also use the standard notations $dx$ or $dy$ when integrating against this measure.

}
Given a matrix $A(\cdot)$ with variable coefficients, we denote by $A^T(\cdot)$ its transpose.


\subsection{David-Semmes dyadic cubes}

In this section we collect some standard definitions and results that we need throughout the rest of the paper. Let us start by introducing a dyadic system of (so-called) cubes associated with an AD-regular measure $\mu.$ They were introduced by David (see \cite{david_cubes}, \cite[Appendix 1]{david_wavelets} and also the work of Christ \cite{christ}). We remark that in the general case they are not euclidean cubes, so that in case of ambiguity we also refer to them as David-Semmes cubes or $\mu$-cubes.
\begin{defin}[David-Semmes lattice $\D_\mu$]Let $\mu\in AD(n,C_0,\mathbb{R}^{n+1})$. The David and Semmes' lattice $\D_\mu$ associated with $\mu$ is a countable disjoint union of families of Borel sets, that we denote as $\D_\mu^j.$ The elements of $\D^j_\mu$ are called dyadic $\mu$-cubes (or just cubes) of the $j$-th generation and satisfy the following properties:
\begin{enumerate}
\item $\D^j_\mu$ is a partition of $\supp\mu$. This means that $\supp\mu=\bigcup_{Q\in\D^j_\mu}Q$ and $Q\cap Q'=\varnothing$ for every $Q,Q'\in\D^j_\mu$ with $Q\neq Q'.$
\item If $Q\in\D^j_\mu$ and $Q'\in\D^k_\mu$ for $k\geq j$, then either $Q'\subset Q$ or $Q\cap Q'=\varnothing.$
\item For every $k$ and $Q\in\D^k_\mu$ we have
\begin{equation}
2^{-k}\lesssim \diam Q\leq 2^{-k}
\end{equation}
and
\begin{equation}
\mu(Q)\approx 2^{-kn}.
\end{equation}
\item The cubes have thin boundary, i.e.\ there exist two constants $C,\gamma_0>0$ depending on $C_0$ and the dimension $n$ such that for every $\varepsilon>0$ and $Q\in\DD_\mu^k$ we have
\begin{equation}\label{thin_bdry}
\mu\{x\in Q\colon\dist(x,\supp\mu\setminus Q)< \varepsilon 2^{-k}\}+\mu\{x\in\supp\mu\setminus Q\colon\dist(x,Q)<\varepsilon2^{-k}\}\leq C\varepsilon^{\gamma_0}\mu(Q).
\end{equation}
\item For $Q\in\D^k_\mu$ there exists a point $x_Q\in Q$, also called \textit{center} of $Q$, such that
\begin{equation}
\dist(x_Q,\supp \mu\setminus Q)\gtrsim 2^{-k}.
\end{equation}
\end{enumerate}
\end{defin}

We need to associate a typical side length to each cube. For $Q\in\D^k_\mu,$ the natural temptation is to define $\ell(Q)\coloneqq  2^{-k}.$ However, we have to take into account that a cube may belong to $\D^j_\mu\cap\D^k_\mu$ for some $j\neq k$. A solution to this problem is to think about a cube as a couple $(Q,k),$ so that the side length is now well defined. Bearing this in mind, in what follows we decide to omit this occurrence and simply indicate a cube by $Q$. We also associate the ball $B_Q\coloneqq B(x_Q,\ell(Q))$  with $Q$.

For $Q\in\D^k_\mu$, we denote by
\begin{equation}
\Ch(Q)\coloneqq \{P\in\D^{k+1}_\mu\colon P \subset Q\}
\end{equation}
the family of children of $Q$. Also, for $N>1$, $\D_\mu^N(Q)$ is the family of the $N$-descendants of $Q$, that is  $\D_\mu^N(Q) = \{P\in\D^{k+N}_\mu\colon P \subset Q\}$.

\subsection{$\beta$ and $\alpha$-numbers} 
Let us consider a ball $B=B(x,r)\subset \mathbb{R}^{n+1}$ and a Radon measure $\mu$. For a hyperplane $L$ in $\mathbb{R}^{n+1}$, we set
\begin{equation}
\beta^L_\mu(B)\coloneqq \sup_{x\in\supp\mu\cap B}\frac{\dist(x,L)}{r},\qquad
\beta^L_{\mu,1}(B)\coloneqq \frac1{r^n}\int_B\frac{\dist(x,L)}{r}\,d\mu(x),
\end{equation}
and taking the infimum over all the hyperplanes $L$ in $\Rn1$, we define
\begin{equation}
\beta_\mu(B)\coloneqq \inf_{L}\beta^L_\mu(B),\qquad
\beta_{\mu,1}(B)\coloneqq \inf_{L}\beta^L_{\mu,1}(B).
\end{equation}

Let $\mu,\nu$ be two Radon measures on $\mathbb{R}^{n+1}$. We define the distance
\begin{equation}\label{distanceballs}
d_B(\mu, \nu)=\sup_f\int fd(\mu-\nu),
\end{equation}
where the supremum is taken over all $1$-Lipschitz functions whose support is contained in $B$. Given a hyperplane $L$, we define
\begin{equation}\label{alpha_number_def}
\alpha^L_\mu(B)\coloneqq \frac{1}{r^{n+1}}\inf_{c\geq 0}d_B(\mu, c\mathcal{H}^n|_{L})
\end{equation}
and
\begin{equation}
\alpha_\mu(B)\coloneqq \inf_L \alpha^L_\mu(B),
\end{equation}
where the infimum is taken over all hyperplanes.  

For an $n$-AD-regular measure $\mu$ and a ball $B$ such that $\frac12B\cap \supp\mu\neq\varnothing$, the following
 inequalities are standard (see \cite[p.\ 27]{david_semmes} and \cite{Tolsa-alfa}):
$$\beta_\mu^L(B)^{n+1} \lesssim \beta_{\mu,1}^L\big(\tfrac32B\big) \lesssim \alpha_\mu^L(2B).$$

Given a hyperplane $H$ through the origin, we also denote
$$\beta_\mu^{(H)}(B) = \inf_L \beta_\mu^L(B),\qquad \alpha_\mu^{(H)}(B) = \inf_L \alpha_\mu^L(B),$$
where in both cases the infimum is taken over all hyperplanes $L$ which are parallel to $H$.


\subsection{Carleson packing condition and Riesz families}

The following are standard definitions.
\begin{defin}[Carleson packing condition] We say that $\mathcal{F}\subset \D_\mu$ is a Carleson family if there exists a constant $C>0$ such that for every $P\in\D_\mu$ we have
\begin{equation}
\sum_{Q\in\mathcal{F},\,Q\subset P}\mu(Q)\leq C\mu(P).
\end{equation}
\end{defin}
\begin{defin}[Riesz families and Riesz systems] Let $\{\psi_Q\}_{Q\in\D_\mu}$ be a family of functions in $L^2(\mu).$ We say that $\{\psi_Q\}_{Q\in\D_\mu}$ forms a Riesz family with constant $C>0$ if
\begin{equation}
\norm[\bigg]{\sum_{Q\in\D_\mu}a_Q\psi_Q}^2_{L^2(\mu)}\leq C\sum_{Q\in\D_\mu}a_Q^2
\end{equation}
for any sequence $\{a_Q\}_Q$ of real numbers with finitely many non-zero terms.
The family $\{\Psi_Q\}_{Q\in\D_\mu}$ of sets of functions is said to be a Riesz system with constant $C>0$ if $\{\psi_Q\}_{Q\in\D_\mu}$ is a Riesz family with constant $C$ for every choice of $\psi_Q\in\Psi_Q$.
\end{defin}
A particular Riesz system that is useful for our purposes is the so-called \textit{Haar system.} Let $N$ be a positive integer. Given $Q\in\D_\mu$ and $C>0$, we define $\Psi^{Haar}_Q(N)$ as the set of functions $\psi$ such that
\begin{enumerate}
\item $\supp \psi\subset Q.$
\item $\psi$ is constant on every $\mu$-cube $Q'$ which is $N$ levels down from $Q$, that is, $\psi$ is constant in each cube from $\D_\mu^N(Q)$.
\item $\int \psi d\mu=0$ and $\int \psi^2d\mu\leq C.$ 
\end{enumerate}
The set of functions $\Psi^{Haar}_{Q}(N)$ forms a Riesz family with constant $C$.

Let $\{\Psi_Q\}_{Q\in\D_\mu}$ be a Riesz system. For any $Q\in\D_\mu$ and $\widetilde{M}>1$ we define
\begin{equation}\label{xi_Q}
\xi_{\widetilde{M}}(Q)\coloneqq \inf_{{\underset{\mu(E)<+\infty}{E: E\supset \widetilde{M}B_Q},}}\,\,\sup_{\psi\in\Psi_Q}\mu(Q)^{-1/2}|\langle T_\mu\chi_E,\psi\rangle_\mu|.
\end{equation}

\begin{lemm}\label{lemma_carleson_family}
Let $\delta>0$ and $\widetilde{M}>1$. If $T_\mu$ is bounded in $L^2(\mu)$, then 
the family
\begin{equation}
\mathcal{F}_\delta\coloneqq \{Q\in\D_\mu:\xi_{\widetilde{M}}(Q)>\delta\}
\end{equation}
is Carleson.
\end{lemm}

\begin{proof}
See \cite[Section 14]{NTV_acta}. There the proof is presented in the case of the Riesz transform, but it works without any difference in our framework.
\end{proof}

\subsection{Partial Differential Equations}

For any uniformly elliptic matrix $A$ with H\"older continuous coefficients, one can show that
$K(x,y)=\nabla_1 \E(x,y)$ is locally a Calder\'on-Zygmund kernel:

\begin{lemm}\label{lemcz}
Let $A$ be an elliptic matrix with H\"older continuous coefficients satisfying \eqref{eqelliptic1}, \eqref{eqelliptic2} and \eqref{eq:Holdercont}. If $K(\cdot,\cdot)$ is given by \eqref{eq:Kdef}, then it is locally a Calder\'on-Zygmund kernel. That is, 
 for any given $R>0$,
\begin{itemize}
\item[(a)] $|K(x,y)| \lesssim |x-y|^{-n}$ for all $x,y\in \R^{n+1}$ with $x \not=y$ and $|x-y|\leq R$.
\item[(b)] $|K(x,y)-K(x,y')| + |K(y,x) - K(y',x)| \lesssim |y-y'|^{\alpha} |x-y|^{-n-\alpha}$ for all $y,y'\in B(x,R)$ with $2|y-y'| \leq |x-y|$.
\item[(c)] $|K(x,y)| \lesssim |x-y|^{{(1-n)}/2}$  for all $x,y\in \R^{n+1}$ with $|x-y|\geq 1$.
\end{itemize}
All the implicit constants in (a), (b) and (c) depend on $\Lambda$ and  $C_h$, while the ones in (a) and (b) depend also on $R$.
\end{lemm}

The statements above are rather standard. For more details, see Lemma 2.1 from \cite{CMT}.

Let $\omega_{n}$ denote the surface measure of the unit sphere of $\mathbb{R}^{n+1}$.
For any elliptic matrix $A_0$ with constant coefficients, we have an explicit expression for the fundamental solution of $L_{A_0}$, which we denote by $\Theta(x,y;A_0).$ More precisely, $\Theta(x,y;A_0)=\Theta(x-y;A_0)$ with
\begin{equation}\label{fund_sol_const_matrix}
\Theta(z;A_0)=\Theta(z;A_{0,s})=
\begin{cases}\displaystyle
\frac{-1}{(n-1)\omega_n\sqrt{\det A_{0,s}}}\frac{1}{(A_{0,s}^{-1}z\cdot z)^{(n-1)/2}} \;\;\;\text{ for }n\geq 3, \\\\
\displaystyle
\frac{1}{4\pi\sqrt{\det A_{0,s}}}\log\big(A_{0,s}^{-1}z\cdot z\big)\;\;\text{ for }n=2,
\end{cases}
\end{equation}
where $A_{0,s}$ is the symmetric part of $A_0$, that is, $A_{0,s}=\frac12(A+A^T)$.

As a consequence of \eqref{fund_sol_const_matrix}, we have
\begin{equation}\label{grad_const_coeff_sol}
\nabla \Theta(z;A_0)=\frac{1}{\omega_n\sqrt{\det A_{0,s}}}\frac{A_{0,s}^{-1}z}{(A_{0,s}^{-1}z\cdot z)^{(n+1)/2}}.
\end{equation}

The next result is proven in Lemma 2.2 of \cite{Kenig-Shen}.

\begin{lemm}\label{lemm_freezing}
Let $A$ be an elliptic matrix with H\"older continuous coefficients satisfying \eqref{eqelliptic1}, \eqref{eqelliptic2} and \eqref{eq:Holdercont}. Let also  $\Theta(\cdot,\cdot; \cdot)$ be given by  \eqref{fund_sol_const_matrix}. Then,
for all $x,y\in \R^{n+1}$ with $x \not=y$ and $|x-y|\leq R$,
\begin{enumerate}
\item $|\mathcal{E}_A (x,y) - \Theta(x,y;A(x))| \lesssim |x-y|^{\alpha-n+1}$,
\item $|\nabla_1\mathcal{E}_A (x,y) - \nabla_1\Theta(x,y;A(x))| \lesssim |x-y|^{\alpha-n}$,
\item $|\nabla_1\mathcal{E}_A (x,y) - \nabla_1\Theta(x,y;A(y))| \lesssim |x-y|^{\alpha-n}$.
\end{enumerate}
Similar inequalities hold if we reverse the roles of $x$ and $y$ and we replace $\nabla_1$ by $\nabla_2$.
All the implicit constants depend on $\Lambda$,  $C_h$, and  $R$.
\end{lemm}



The following lemma is an easy consequence of the preceding result.

\begin{lemm}
Let $\mu$ be a compactly supported $n$-AD-regular measure in $\mathbb R^{n+1}$.  Let $A$ be an elliptic matrix satisfying \eqref{eqelliptic1}, \eqref{eqelliptic2} and \eqref{eq:Holdercont}, and let $T_\mu$ be the associated operator given by \eqref{eq:Tfdef}. Let 
$A_s=\frac12 (A+A^T)$ be the symmetric part of $A$. Consider the operator
$$T_\mu^s f(x) =\int \nabla_1\E_{A_s}(x,y)\,f(y)\,d\mu(y).$$
Then, $T_\mu - T_\mu^s$ is compact in $L^p(\mu)$, for $1<p<\infty$. In particular, $T_\mu$ is bounded in $L^2(\mu)$ if and only if 
$T_\mu^s$ is bounded in $L^2(\mu)$.
\end{lemm}

Recall that $\E_{A_s}$ stands for the fundamental solution of $L_{A_s} u\coloneqq  -\mathrm{div}\left(A_s \nabla u  \right).$

\begin{proof}
For any function $f\in L^p(\mu)$, we have
$$T_\mu f(x)- T_\mu^s f(x) =\int \bigl(\nabla_1\mathcal{E}_A (x,y) - \nabla_1\mathcal{E}_{A_s} (x,y)\bigr)
f(y)\,d\mu(y).$$
By \eqref{fund_sol_const_matrix} $\Theta(x,y;A(x)) = \Theta(x,y;A_s(x))$ and thus,
by Lemma \ref{lemm_freezing}, the kernel of $T_\mu- T_\mu^s$ satisfies,
for all $x,y\in\R^{n+1}$ with $x \not=y$ and $|x-y|\leq R$,
\begin{align}
|\nabla_1\mathcal{E}_A (x,y) - \nabla_1\mathcal{E}_{A_s} (x,y)| &
\leq  |\nabla_1\mathcal{E}_A (x,y) - \nabla_1\Theta(x,y;A(x))| \\
&\quad + 
|\nabla_1\Theta(x,y;A_s(x))- \nabla_1\mathcal{E}_{A_s} (x,y)|\\
& \lesssim \frac1{|x-y|^{n-\alpha}}.
\end{align}
By standard arguments, using the AD-regularity of $\mu$,
this implies that $T_\mu - T_\mu^s$ is compact, and thus bounded in $L^p(\mu)$.
\end{proof}

Because of the preceding lemma, it is clear that to prove Theorem \ref{teo1} we can assume
that the matrix $A$ is symmetric. So in the rest of the paper {\em we will assume $A$ to be symmetric.}

By almost the same arguments as above we derive that
\begin{equation} \label{remantis}
\bigl|\nabla_1\mathcal{E}_A (x,y) + \nabla_1\mathcal{E}_A (y,x)\bigr|\lesssim \frac1{|x-y|^{n-\alpha}}
\quad\mbox{ for all $x,y\in \R^{n+1}$ with $x \not=y$ and $|x-y|\leq R$.}
\end{equation}
So, modulo the regularizing kernel ${|x-y|^{-(n-\alpha)}}$, $\nabla_1\mathcal{E}_A (x,y)$ behaves
as if it were antisymmetric. In particular, we have the following result.

\begin{lemm}\label{lemantisym}
Let $\mu$ be a compactly supported $n$-AD-regular measure in $\mathbb R^{n+1}$.  Let $A$ be an elliptic matrix satisfying \eqref{eqelliptic1}, \eqref{eqelliptic2} and \eqref{eq:Holdercont}, and let $T_\mu$ be the associated operator given by \eqref{eq:Tfdef}, with kernel $K(x,y) = \nabla_1 \mathcal{E}_A(x,y)$.
Consider the antisymmetric operator $T_\mu^{(a)}$ and the symmetric operator $T_\mu^{(s)}$ associated  with the kernels
$$K^{(a)}(x,y) = \frac12\bigl(K(x,y) - K(y,x)\bigr)\qquad\mbox{ and }\qquad
K^{(s)}(x,y) = \frac12\bigl(K(x,y) + K(y,x)\bigr)$$
respectively, so that $T_\mu = T_\mu^{(a)} + T_\mu^{(s)}$. Then the operator $T_\mu^{(s)}$
is compact in $L^p(\mu)$, for $1<p<\infty$. In particular, $T_\mu$ is bounded in $L^2(\mu)$ if and only if 
$T_\mu^{(a)}$ is bounded in $L^2(\mu)$.
\end{lemm}

Contrarily to the natural temptation at this point, in the rest of the paper we \textit{do not} assume the kernel to be antisymmetic. This is because our proof heavily relies on a maximum principle (see, for example, Lemma \ref{Le:PrMax}), which cannot be ensured to hold if we work just with the antisymmetric part.

From Lemma \ref{lemantisym} we derive the existence of a ``weak limit operator'':

\begin{prop}\label{teoconvdebil}
Let $\mu$ be a compactly supported $n$-AD-regular measure in $\mathbb R^{n+1}$.  Let $A$ be an elliptic matrix satisfying \eqref{eqelliptic1}, \eqref{eqelliptic2} and \eqref{eq:Holdercont}, and let $T_\mu$ be the associated operator given by \eqref{eq:Tfdef}. Suppose that $T_\mu$ is bounded in $L^2(\mu)$.
 Then, for all $1<p<\infty$ and $f\in L^p(\mu)$, 
 $T_{\mu,\ve} f$ has a weak limit in $L^p(\mu)$ as $\ve\to0$. Further, denoting by $T_\mu^w f$
 such a weak limit, the operator $T_\mu^w$ is bounded in $L^p(\mu)$ for $1<p<\infty$ and, for all $f\in L^p(\mu)$,
\begin{equation}\label{eqsjt55} 
T_\mu f(x) = T_\mu^w f(x)\quad \mbox{ for $\mu$-a.e.\ $x\in \supp\mu\setminus \supp f$.}
\end{equation}
 \end{prop}

Recall that saying that
$T_{\mu,\ve} f$ has a weak limit $T_\mu^w f$ in $L^p(\mu)$ as $\ve\to0$
means that for all $g\in L^{p'}(\mu)$,
$$\lim_{\ve\to 0}\int T_{\mu,\ve}f\,g\,d\mu = \int T_{\mu}^wf\,g\,d\mu.$$

\begin{proof}
Consider the antisymmetric and symmetric operators $T_\mu^{(a)}$, $T_\mu^{(s)}$ from Lemma \ref{lemantisym},
so that, for all $\ve>0$,
$$T_{\mu,\ve} f = T_{\mu,\ve}^{(a)} f + T_{\mu,\ve}^{(s)} f.$$
Since $T_{\mu}^{(a)}$ is antisymmetric, for all $f\in L^p(\mu)$, $1<p<\infty$, the functions $T_{\mu,\ve}^{(a)} f$ converge weakly in $L^p(\mu)$ as $\ve\to0$. This was shown by Mattila and Verdera in \cite{MV} and
an alternative argument is provided in \cite{NTV_acta}. 

Concerning the symmetric operator $T_\mu^{(s)}$, from the estimate \rf{remantis} it easily follows that
$T_{\mu,\ve}^{(s)}f$ converges to $T_{\mu}^{(s)}f = T_{\mu,0}^{(s)}f$ strongly in $L^p(\mu)$, and thus also
weakly in $L^p(\mu)$. Hence, $T_{\mu,\ve} f$ admits a weak limit in $L^p(\mu)$ as $\ve\to0$.

The last statement in the lemma follows by standard arguments.
\end{proof}

From now on, for $\mu$ and $T_\mu$ as above, when $T_\mu$ is bounded in $L^2(\mu)$ we will identify $T_\mu$ with the weak limit operator $T_\mu^{w}$, so that for any function $f\in L^p(\mu)$, $T_\mu f$ makes sense as a function in $L^p(\mu)$.


\section{The flattening lemmas and the alternating layers}

From this section until the end of Section \ref{sec12} we assume that $\mu$ is an $n$-AD-regular measure with compact support and
that $T_\mu$ is bounded in $L^2(\mu)$. In order to prove Theorem \ref{teo1} we have to show that
$\mu$ is uniformly $n$-rectifiable.

\subsection{Existence of balls with small $\beta$-number}

We want to prove that in any ball centered at a point of $\supp\mu$ either we can find a ball, which is not too small, in which the measure is very flat or we have a lower bound for a regularized two-sided truncation of $T\mu$ at some point and at proper scales.

Let $\psi_0\colon[0,+\infty)\to [0,1]$ be a continuous function such that $\psi_0(x)=1$ for $x\leq 1$ and $\psi_0(x)=0$ for $x\geq 2$.
For $z\in\Rn1$ and $0<r_1<r_2,$ we define
\begin{equation}
\psi_{z,r_1,r_2}(x)\coloneqq \psi_0\Big(\frac{|z-x|}{r_2}\Big)-\psi_0\Big(\frac{|z-x|}{r_1}\Big).
\end{equation}
We have that $\supp\psi_{z,r_1,r_2}\subset B(z,r_2)\setminus B(z,r_1)$ and $0\leq \psi_{z,r_1,r_2}\leq 1.$
The proof of the following lemma relies on a touching point argument and it is based on the scheme of the proof of \cite[Lemma 3.3]{tolsa_unif_measures}. We remark that this can also be proved via a variation on the blow-up argument  in \cite[Lemma 5]{NTV_acta}.

\begin{lemm}\label{prop_flat_beta}
Let $\mu\in AD(n,C_0,\Rn1),$ $R\leq 4$ and let $B=B(x,R)$ be a ball centered at $\supp\mu.$ Let $K,\varepsilon>0.$ There is $\rho=\rho(K,\varepsilon,C_0)$ small enough such that at least one of the two following conditions is verified:
\begin{enumerate}
\item There exists a ball $B(x',r)\subset B$  centered at $\supp\mu$ with $r\in [\rho R, R]$ such that
\begin{equation}
\beta_\mu(B(x',r))\leq \varepsilon.
\end{equation}
\item There is a point $z\in \supp\mu\cap B(x,R/4)$ and $r\in [\rho R,R]$, such that
\begin{equation}
|T(\psi_{z,\rho R, r}\mu)(z)|>K.
\end{equation}
\end{enumerate}
\end{lemm}
Before reporting the proof, we remark that the assumption $R\leq 4$ in the statement of the lemma is justified by the fact that we are interested in applying this result to the balls associated with David-Semmes cubes with small enough side length.

\begin{proof}
Suppose that the alternative (1) in the statement of the lemma does not hold. Then
\begin{equation}\label{big_beta_contradiction_hp}
\beta_\mu(B(x',r))>\varepsilon
\end{equation}
for every $x'\in\supp\mu\cap B$ and $r\in[\rho R, R]$ such that $B(x',r)\subset B$.

Being the measure $\mu$ $n$-AD-regular, by standard arguments it follows that there exists an open ball $B'$ contained in $\frac14 B$ such that $B'\cap \supp\mu=\varnothing$ and $r(B')\geq c_1 R$ with $c_1=c_1(n,C_0)$. Possibly by taking a dilation of this ball, we can suppose that $B'\cap \supp\mu=\varnothing$ but there is at least a point $z\in\partial B'\cap \supp\mu$. Without loss of generality, let $z=0$ and suppose that  $\vec{n}\coloneqq (0,\ldots,0,1)$ is the outer normal vector to $\partial B'$ at $z$. Since $B'\subset \frac14B$, we also have that $r(B')\leq R/4$.

We denote by $L$ the hyperplane $\{x\colon x\cdot \vec{n}=0\}$, by $U$ the upper half space $\{x\colon x\cdot\vec{n}>0\}$ and by $D$ the lower one $D\coloneqq \mathbb{R}^{n+1}\setminus (U\cup L)$. For $0<\rho\ll 1$ to be chosen later and for $j\geq 0$, we denote by $B_j$ the ball centered at $0$ and with radius
\begin{equation}
r(B_j)\coloneqq \Big(\frac{2}{\varepsilon}\Big)^j\rho \,R.
\end{equation}
Let $j$ be such that $r(B_j)\leq r(B')$. Short geometric computations prove the inequality
\begin{equation}\label{dist_d_bj}
\dist(y,L)\leq \frac{1}{2}\frac{r(B_j)^2}{r(B')}\quad\mbox{ for every $y\in D\cap B_j\setminus B'.$}
\end{equation}

We denote $\vec{v}\coloneqq A(0)^{T}\vec{n}$.
Using the definition of $\vec{v}$ and \eqref{grad_const_coeff_sol}, we get that there exists $c_2>0$
such that
\begin{equation}\label{gt0}
\vec{v}\cdot \nabla_1\Theta(0,y;A(0))\geq c_2 \frac{\vec{v}\cdot A(0)^{-1}y}{|y|^{n+1}}= c_2 \frac{\vec{n}\cdot y}{|y|^{n+1}}>0\quad\mbox{ for every $y\in U$.}
\end{equation}
 Choose now an integer $N>1$ such that $r\coloneqq r(B_N)\leq r(B').$
As a direct application of Lemma \ref{lemm_freezing} and the growth of $\mu$, we can find two constants $c_3,c_3'>0$ such that
\begin{equation}\label{approx_U_freezing}
\begin{split}
\Big|\int_{D\cap B(0,r)}\vec{v}\cdot \big(\nabla_1\E(0,y)-\nabla_1\Theta(0,y;A(0))\big)\psi_{0,\rho R, r}(y)d\mu(y)\Big|
\leq c_3' \int_{D\cap B}\frac{1}{|y|^{n-\alpha}}d\mu(y)\leq c_3 R^\alpha.
\end{split}
\end{equation}
Let $\chi_{0,r_1, r_2}$ be the characteristic function of the annulus centered at 0 with inner and outer radius $r_1$ and $r_2$ respectively. 
Then, choosing $\rho$ small enough to get $r>2\rho R$ and using \eqref{gt0}, we have that
\begin{equation}\label{ineq_psi_chi}
\begin{split}
\int_{U\cap B(0,r)}\vec{v}\cdot \nabla_1\Theta(0,y;A(0))\psi_{0,\rho R, r}(y)d\mu(y)
\geq \int_{U\cap B(0,r)}\vec{v}\cdot \nabla_1\Theta(0,y;A(0))\chi_{0,2\rho R, r}(y)d\mu(y).
\end{split}
\end{equation}
Since $\beta_\mu(B_j)\geq\varepsilon$ by hypothesis \eqref{big_beta_contradiction_hp}, we have that there exists $y\in \supp\mu\cap B_j$ whose distance from $L$ is greater than $\varepsilon r(B_j)$. As a consequence of \eqref{dist_d_bj}, the point $y$ cannot belong to $D$ if
\begin{equation}
\varepsilon \,r(B_j)\geq \frac{1}{2}\frac{r(B_j)^2}{r(B')},
\end{equation}
which implies that $y \in U\cap B_j$ for every $r(B_j)\leq 2\varepsilon r(B').$
 Since $\mu\in AD(n,C_0,\Rn1),$ assuming $\varepsilon$ small enough if necessary, it follows that 
 $$\mu\big(U\cap B_{j+1}\setminus (B_{j-1}\cup\mathcal{U}_{\varepsilon r(B_j)/2}(L))\big)\geq C_0^{-1}c(\varepsilon)r(B_j)^n,$$
 for some constant $c(\varepsilon)>0$, where $\mathcal{U}_{\varepsilon r(B_j)/2}(L)$ stands for the $\varepsilon r(B_j)/2$-neighborhood of $L$. Taking into account \eqref{gt0}, for $j\geq 0$ we deduce that
 \begin{equation}
 \begin{split}
\int_{U\cap B_{j+1}\setminus B_{j-1}} \vec{v} \cdot \nabla_1\Theta(0,y;A(0))d\mu(y)&\geq \mu\big(U\cap B_{j+1}\setminus (B_{j-1}\cup\mathcal{U}_{\varepsilon r(B_j)/2}(L))\big)\frac{\varepsilon r(B_j)}{2r(B_{j-1})^{n+1}}\\&\geq C_0^{-1}c(\varepsilon).
\end{split}
\end{equation}
for some constant $c(\varepsilon)$. 
Therefore 
\begin{equation}\label{ineq_sum_balls_j}
\begin{split}
&\int_{U\cap B(0,r)}\vec{v}\cdot \nabla_1\Theta(0,y;A(0))\chi_{0,2\rho R, r}(y)d\mu(y)\\
&\qquad=\sum_{j=1}^N\int_{U\cap B_j\setminus B_{j-1}}\vec{v}\cdot \nabla_1\Theta(0,y;A(0))d\mu(y)
\geq C_0^{-1} \sum_{j=2}^{N-1} c(\varepsilon)=C_0^{-1} c(\varepsilon)(N-2).
\end{split}
\end{equation}
Now we need to study the analogous integrals for the lower half-space. As in \eqref{approx_U_freezing}, we have
\begin{equation}\label{approx_freezing_2}
\Big|\int_{U\cap B(0,r)}\vec{v}\cdot \big(\nabla_1\E(0,y)-\nabla_1\Theta(0,y;A(0))\big)\psi_{0,\rho R, r}(y)d\mu(y)\Big|\leq c_4 R^\alpha
\end{equation}
for some $c_4>0$. Moreover, by \eqref{dist_d_bj} and the growth of $\mu$,
\begin{equation}\label{sum_balls_j_down}
\begin{split}
&\sum_{j=1}^N\Big|\int_{D\cap B_j\setminus B_{j-1}}\vec{v}\cdot \nabla_1\Theta(0,y;A(0))d\mu(y)\Big|\\
&\qquad\leq c_5\sum_{j=1}^N\Big|\int_{D\cap B_j\setminus B_{j-1}}\frac{\dist(y,L)}{|y|^{n+1}}d\mu(y)\Big|
\leq c_5 C_0 c(\varepsilon)\sum_{j=1}^N\frac{r(B_j)}{r(B')}\leq C_0 c_6(\varepsilon).
\end{split}
\end{equation}
Gathering \eqref{approx_U_freezing}, \eqref{ineq_psi_chi}, \eqref{ineq_sum_balls_j}, \eqref{approx_freezing_2} and \eqref{sum_balls_j_down} we get
\begin{equation}
\vec{v}\cdot T(\psi_{z,\rho R, r}\mu)(z)\geq C_0^{-1} c(\varepsilon)(N-2) - (c_3+c_4)R^\alpha - C_0 c_6(\varepsilon),
\end{equation}
which gives the desired estimate for $N$ big enough (which forces $\rho$ to be small enough).
\end{proof}

\subsection{Existence of balls and cubes with small $\alpha$-number}
Our proof of the existence of balls and cubes with small $\alpha$-number relies on the following result by Girela-Sarri\'on and Tolsa.

\begin{lemm}[Existence of $\alpha$-flat balls]\label{existence_alpha_flat_balls}
Let $\mu\in AD(n,C_0,\mathbb{R}^{n+1}).$ 
Let $B=B(x,r)$ be a ball centered at $x\in\supp\mu$ with $\beta^L_\mu(B)<\varepsilon$ for some hyperplane $L$ and some $\ve$ small enough. For every $M>10$ and $\tve>0$, there exists $\wt B=B(\tilde x, \tilde r)$ with $\tilde{x}\in\supp\mu$ such that:
\begin{enumerate}
\item $M\wt B\subset B(x,r)$.
\item $\tilde{r}>\sigma r$ for some constant $\sigma$ depending on $\widetilde{\varepsilon}.$
\item $\alpha^L_\mu (M\wt B)\leq \widetilde{\varepsilon}.$
\end{enumerate}
\end{lemm}

This is an immediate consequence of \cite[Lemma 3.2]{GirTolsa}. We remark that this lemma was originally stated in a setting which is more general than the one of AD-regular measures. Note also that the $L^2(\mu)$-boundedness of any singular integral operator is not required in the lemma, so the statement is purely geometric.

\vspace{1mm}

The scheme of the proof of the next lemma resembles that of \cite[Section 15]{NTV_acta}.

\begin{lemm}\label{flat_alpha_prop}
For every $M>1$ and $\bar{\varepsilon}>0$ there exist an integer $N$, a finite set $\mathcal{H}$ of hyperplanes through the origin and a Carleson family $\mathcal{F}\subset\D_\mu$ with the following property. If $P\in \D_\mu\setminus\mathcal{F},$ there exist $H\in\mathcal{H}$ and a cube $Q\subset P$ at most $N$ levels down from $P$ for which
\begin{equation}\label{small_alpha_cube}
\alpha^{(H)}_\mu\bigl(MB_Q\big)\leq \bar{\varepsilon}.
\end{equation}
\end{lemm}

\begin{proof}
The idea is to combine Lemma \ref{prop_flat_beta} and Lemma \ref{existence_alpha_flat_balls}. We fix a cube and we show that either the condition (1) in Lemma \ref{prop_flat_beta} is verified, so that we can find a ball with small $\beta$-number and apply Lemma \ref{existence_alpha_flat_balls}, or the cube belongs to a Carleson family. We define the family $\mathcal{F}$ as the collection of cubes for which condition (1) in Lemma \ref{prop_flat_beta} does not apply.

Let $P\in\D_\mu$ and let $R\coloneqq \ell(P).$ Let $\varepsilon$ and $K$ be as in Lemma \ref{prop_flat_beta}, to be chosen later. We analyze the two different cases starting from the ``flat" one.

\subsubsection*{Case (1).} Suppose that there is $\rho>0$ such that $r>\rho R$ and we can find a ball $B(z,r)\subset B(x_P,R)$ with
\begin{equation}
\beta^L_\mu(B(z,r))\leq \varepsilon
\end{equation}
for some hyperplane $L$. Let $H$ be a hyperplane through the origin whose normal spans an angle at most $\varepsilon$ with the normal to $L$. Elementary geometric considerations lead to
\begin{equation}
\beta^{(H)}_\mu(B(z,r))\leq 2\varepsilon.
\end{equation}
It is possible to suppose that $H$ belongs to a finite family $\mathcal{H}$ of hyperplanes: it suffices to define $\mathcal{H}$ as the family of hyperplanes whose normal vectors form an $\varepsilon$-net on the unit sphere $\mathbb{S}^n$.

 By Lemma \ref{existence_alpha_flat_balls} for every $\widetilde{\varepsilon}>0$ to be chosen later and $\varepsilon$ small enough (depending on $\widetilde{\varepsilon}$) there are $\sigma>0$ and a ball $B(\widetilde{z},2(M+2)\widetilde{r})$ such that $\widetilde{r}>\sigma r$ and
\begin{equation}\label{small_epsilon_tilde}
\alpha^{(H)}_\mu \big(B(\widetilde{z},2(M+2)\widetilde{r})\big)\leq \widetilde{\varepsilon}.
\end{equation}
Take a point $z'\in\supp\mu$ such that $|\widetilde{z}-z'|<\widetilde{\varepsilon}\widetilde{r}$. We choose the cube $Q\in\D_\mu$ as the one such that $z'\in Q$ and $\widetilde{r}\leq \ell(Q)\leq 2\widetilde{r}.$ For $\widetilde{\varepsilon}<1$ we have
\begin{equation}
|\tilde z-x_Q|\leq |z'-x_Q|+|z'-\tilde z|< \ell(Q) + \widetilde{\varepsilon}\widetilde{r}<2\ell(Q).
\end{equation}
Now we use the stability of the $\alpha$-number under small shifts and proper rescalings to compare $\alpha^{(H)}_\mu(MB_Q)$ to $\alpha^{(H)}_\mu(B(\widetilde{z},2(M+2)\widetilde{r}))$ and, hence, to prove that it is small. Being $M>1,$ we have $(M+2)/3<M$. So, using the inclusions
\begin{equation}
MB_Q\subset B(x_Q, 2M\widetilde{r})\subset B(\tilde{z}, 2(M+2)\widetilde{r}),
\end{equation}
for some plane $L$ parallel to $H$
we can write
\begin{equation}
\begin{split}
\alpha^{(H)}_\mu(MB_Q) = \alpha^{L}_\mu(MB_Q)&=\frac{1}{(M\ell(Q))^{n+1}}\inf_{c\geq 0}d_{MB_Q}(\mu,c\mathcal{H}^n{|_L})\\
&\leq \frac{2^{n+1}}{(2M\widetilde{r})^{n+1}}\inf_{c\geq 0}d_{B(x_Q,2M\widetilde{r})}(\mu,c\mathcal{H}^n{|_L})\\
&\leq \Big(\frac{2(M+2)}{M}\Big)^{n+1}\frac{1}{(2(M+2)\widetilde{r})^{n+1}}\inf_{c\geq 0}d_{B(\widetilde{z},2(M+2)\widetilde{r})}(\mu,c\mathcal{H}^n{|_L})\\
&\leq {6^{n+1}} \alpha^{(H)}_\mu \big(B(\widetilde{z},2(M+2)\widetilde{r})\big).
\end{split}
\end{equation}
Then, recalling \eqref{small_epsilon_tilde} we have
\begin{equation}
\alpha^{(H)}_\mu\big(MB_Q\big)\leq 6^{n+1}\widetilde{\varepsilon}.
\end{equation}
The proof of \eqref{small_alpha_cube} is completed by choosing $\varepsilon$ such that $\bar{\varepsilon}=6^{n+1}\widetilde{\varepsilon},$ where $\bar{\ve}$ is as in the statement of the lemma.
The cube $Q$ is at most $N$ levels down from $P$ for some $N$ that, being $\ell(Q)\geq \ell(P)\sigma\rho/2$, satisfies
\begin{equation}\label{bound_n}
N\leq \log_{2}\frac{\ell(P)}{\ell(Q)}\leq 1- \log_{2}\rho -\log_{2}\sigma.
\end{equation}
Again, we remark that the estimate in the right hand side of \eqref{bound_n} depends just on $M$ and $\bar{\varepsilon}$.
\subsubsection*{Case (2).} Let $z$ be a point in $\supp\mu\cap B(x,R/4)$, such that
\begin{equation}\label{bound_psi_below}
|T(\psi_{z,\rho R, r}\mu)(z)|>K.
\end{equation}
Let $Q$ be the largest $\mu$-cube containing $z$ with $\ell(Q)<r/32$ and let $Q'$ be the largest $\mu$-cube containing $z$ with $\ell(Q')<\rho R/32.$ Then $Q'\subset Q\subset P.$

The idea of this part of the proof is to apply Lemma \ref{lemma_carleson_family} to prove that the family $\mathcal{F}$ of $\mu$-cubes $P$ for which case (2) applies is Carleson. To this purpose, consider the set $E=10B_P$, which contains $B(z,2R)$.
We claim that there is a constant $\widetilde{C}$ such that
\begin{equation}\label{diff_mean_E}
|m_{\mu,Q}(T_\mu \chi_E)-m_{\mu, Q'}(T_\mu \chi_E)|\geq K - \widetilde{C}.
\end{equation}
To prove this, we consider two continuous functions $f_1$ and $f_2$ with $|f_1|,|f_2|\leq 1$ and such that
\begin{equation}
\chi_E=f_1+\psi_{z,\rho R, r}+f_2,
\end{equation}
$\supp f_1\subset B(z,2\rho r)$ and $\supp f_2\cap B(z,r)=\varnothing.$

Using the $L^2(\mu)$-boundedness of $T_\mu$, the regularity of the measure and the fact that $Q'\subset Q$, we have
\begin{equation}
\begin{split}
\int |T_\mu f_1|^2d\mu&\lesssim \int |f_1|^2 d\mu\leq \mu(\supp f_1)\\
&\leq \mu(B(z,2\rho R))\lesssim (\rho R)^n\lesssim \ell(Q')^n\lesssim \mu(Q')\leq \mu(Q),
\end{split}
\end{equation}
which yields that there exists a constant $C_1>0$ such that
\begin{equation}\label{diff_mean_f_1}
|m_{\mu,Q}(T_\mu f_1)-m_{\mu, Q'}(T_\mu f_1)|\leq |m_{\mu,Q}(T_\mu f_1)|+|m_{\mu, Q'}(T_\mu f_1)|\leq C_1.
\end{equation}
Using $L^2(\mu)$-boundedness again we have
\begin{equation}
\norm{T_\mu\psi_{z,\rho R, r}}_{L^2(\mu)}\lesssim \norm{\psi_{z,\rho R, r}}_{L^2(\mu)}\leq \mu(B(z,2r))^{1/2}\lesssim r^{n/2}\lesssim \ell(Q)^{n/2}\lesssim \mu(Q)^{1/2},
\end{equation}
which implies that there exists a constant $C_2>0$ such that 
\begin{equation}\label{mean_Q}
|m_{\mu,Q}(T_\mu \psi_{z,\rho R, r})|\leq C_2.
\end{equation}
By the choice of $Q'$, we have that $Q'\subset B(z,\rho R/2).$ Indeed
\begin{equation}\label{inclusion_chains}
Q'\subset B(z',8 \ell(Q'))\subset B(z',\rho R/4)\subset B(z,\rho R/2).
\end{equation}
Being $B(z,\rho R)\cap \supp\mu=\varnothing,$ we have the following estimate for the H\"older norm:
\begin{equation}\label{holder_norm_psi}
\norm{T_\mu\psi_{z,\rho R,r}}_{C^\alpha(B(z,\rho R/2))}\lesssim (\rho R)^{-\alpha},
\end{equation}
so there exists a constant $C_3>0$ such that for every $y\in Q'$ 
\begin{equation}\label{mean_Q'}
\begin{split}
&|m_{\mu, Q'}(T_\mu\psi_{z,\rho R,r})|\geq |T_\mu(\psi_{z,\rho R,r})(y)|-|m_{\mu, Q'}(T_\mu\psi_{z,\rho R,r}) - T_\mu(\psi_{z,\rho R,r})(y)|\\
&\qquad \geq K - \norm{T_\mu \psi_{z,\rho R,r}}_{C^{\alpha}(B(z,\rho R/2))}\dist(Q',\supp \psi_{z,\rho R, r})^\alpha\geq K - C_3.
\end{split}
\end{equation}
Gathering \eqref{mean_Q} and \eqref{mean_Q'} we get
\begin{equation}\label{diff_mean_psi}
|m_{\mu, Q'}(T_\mu\psi_{z,\rho R,r}) - m_{\mu,Q}(T_\mu \psi_{z,\rho R, r})|\geq K - C_2-C_3.
\end{equation}

Let us estimate the difference between the averages of $T_\mu f_2$ over the $\mu$-cubes $Q$ and $Q'$. Arguing as in \eqref{inclusion_chains} and \eqref{holder_norm_psi}, we have that $Q\subset B(z,r/2)$ and
\begin{equation}
\norm{T_\mu f_2}_{C^\alpha(B(z,r/2))}\lesssim \ell(Q)^{-\alpha},
\end{equation}
so there exists a constant $C_4>0$ such that
\begin{equation}\label{diff_mean_f_2}
 |m_{\mu,Q}(T_\mu f_2)-m_{\mu, Q'}(T_\mu f_2)|\leq C_4.
\end{equation}
Gathering \eqref{diff_mean_f_1}, \eqref{diff_mean_psi} and \eqref{diff_mean_f_2}, we prove the claim \eqref{diff_mean_E}. Now, if we choose $\theta>0$ and we define\begin{equation}
\psi_P\coloneqq \big(\theta \ell(P)\big)^{n/2}\Big(\frac{\chi_Q}{\mu(Q)}-\frac{\chi_{Q'}}{\mu(Q')}\Big),
\end{equation}
as a consequence of \eqref{diff_mean_E} we get
\begin{equation}\label{geq_below}
\begin{split}
&\mu(P)^{-1/2}|\langle T_\mu \chi_E ,\psi_P\rangle_\mu|\\
&\qquad=\mu(P)^{-1/2}\big(\theta \ell(P)\big)^{n/2}|m_{\mu,Q}(T_\mu\chi_E)-m_{\mu, Q'}(T_\mu \chi_E)|\geq C_0^{-1/2}\theta^{n/2}(K-\widetilde{C}).
\end{split}
\end{equation}
We remark that $\theta$ serves as a normalizing factor in order to get a bound on the $L^2(\mu)$ norm of $\psi_P.$
In this way, we have that $\psi_P$ belongs to the Haar system $\Psi^{Haar}_P(N)$ of depth 
\begin{equation}
N=\log_{2} (\ell(P)/\ell(Q))\leq \log_{2}\theta^{-1}+\widetilde{C}
\end{equation}
so that we can combine \eqref{geq_below} and Lemma \ref{lemma_carleson_family}. Indeed, recalling the definition of $\xi_{\widetilde{M}}(P)$ provided in \eqref{xi_Q}, \eqref{geq_below} proves that $\xi_5(P)\geq C_0^{-1/2}\theta^{n/2}(K-\widetilde{C})$, which implies that $\mathcal{F}$ is a Carleson family for $K$ big enough.
\end{proof}

As an immediate consequence of the preceding Corollary \ref{flat_alpha_prop} we get the following.

\begin{coroll}\label{flat_alpha_prop'}
For every $M>1$ and $\bar{\varepsilon}>0$ there exist an integer $N'$ and a finite set $\mathcal{H}$ of hyperplanes through the origin with the following property: for every $P\in \D_\mu,$ there exist $H\in\mathcal{H}$ and a cube $Q\subset P$ at most $N'$ levels down from $P$ for which
$
\alpha^{(H)}_\mu\big(MB_Q\big)\leq \bar{\varepsilon}.
$
\end{coroll}

\begin{proof}
Consider the family $\mathcal{F}$ in the preceding lemma. Since this is a Carleson family, for any $P\in\D_\mu$ there exists
some $P'\in\D_\mu\setminus \mathcal F$ contained in $P$ with $\ell(P')\approx\ell(P)$. Then, by definition, there exists
a cube $Q\subset P'$, with $\ell(Q)\approx\ell(P')\approx\ell(P)$ and such that $\alpha^{(H)}_\mu\big(MB_Q\big)\leq \bar{\varepsilon}$  for some hyperplane $H\in\mathcal{H}$.
\end{proof}


\subsection{The alternating layers}

A general feature of non-Carleson families is that, for every positive integer $K_0$, it is possible to find a $\mu$-cube and $(K_0+1)$ layers of finitely many cubes so that each of them tiles up the initial cube up to a set of small measure (for the details see \cite[Lemma 7]{NTV_acta}). This result can be refined by finding intermediate layers of very flat cubes using Corollary \ref{flat_alpha_prop'}. For the proof of the following lemma we refer to \cite[Section 16]{NTV_acta}. 

\begin{lemm}\label{altern_nb_flat}
Let $\varepsilon>0,$ $M>1$ and let $H$ be a hyperplane through the origin in $\Rn1$. Let $\mathcal A\subset\D_\mu$ be a non-Carleson family such that each $Q\in\mathcal A$ contains a cube $Q'\in\D_\mu$  at most $N'$ levels down from $Q$ such that $\alpha^{(H)}_\mu(MB_{Q'})<\ve$. Then, for every positive integer $K$ and every $\eta>0$ there exist a cube $R_0\in\mathcal A$ and $(K+1)$ alternating pairs of finite layers $\NB_k$ and $\FL_k$ in $\mathcal{D}_\mu$ with $k=0,1,\ldots,K$ such that the following properties hold
\begin{enumerate}
\item $\NB_0=\{R_0\}.$
\item $\NB_k\subset \{Q\in\D_\mu\colon Q\subset R_0\}\cap \mathcal A$ for any $k=0,\ldots,K$.
\item for every $k=0,\ldots, K$ and $Q\in \FL_k$ we have
\begin{equation}
\alpha^{(H)}_\mu (MB_Q)<\varepsilon.
\end{equation}
\item for every $k=0,\ldots,K$ and $Q\in \FL_k$ there exists a cube $P\in \NB_k,$ $P\supset Q.$
\item for every $k=1,\ldots,K$ and $P'\in \NB_k$ there exists a cube $Q\in \FL_{k-1},$ $P'\subset Q.$
\item \label{property_6} $\sum_{Q\in \FL_K}\mu(Q)\geq (1-\eta)\mu(R_0).$
\end{enumerate}
\end{lemm}

We will apply Lemma \ref{altern_nb_flat} to the study of non-BAUP cubes (see the next section for the definition); this explains the choice of the notation `$\NB_k$' for some layers. The other layers are denoted as `$\FL_k$' to indicate that they consist of quite flat cubes (i.e. with a small $\alpha$-number).
\begin{rem}
The property \ref{property_6} in the lemma says that $\FL_K$ tiles up $R_0$ up to a set of negligible measure. If follows that the same holds for any $\FL_k$ for every $k=0,\ldots,K$. Moreover, as a consequence of the inductive construction in \cite{NTV_acta}, the lattice
\begin{equation}
\FL=\bigcup_k \FL_k
\end{equation}
has only finitely many elements.\footnote{Each of the so-called \textit{non-Carleson layers} $\{\mathcal{L}_m\}_{m=0}^M$ appearing in \cite[Section 13]{NTV_acta} is finite.}
 This is useful for technical purposes.
\end{rem}


\section{The non-BAUP cubes and the martingale difference decomposition}\label{section_baup}\label{sec4}

{
%
The acronym BAUP referred to a $\mu$-cube literally stands for \textit{Bilaterally Approximable by a Union of Planes}. Being more suitable to our purposes, in what follows we prefer to formulate the equivalent definition of non-BAUP cubes as in \cite[Section 22]{NTV_acta}, instead of the original definition of David and Semmes in \cite{david_semmes}.

\begin{defin}[Non-BAUP cube]\label{definition_NB_cube} A cube $Q\in\D_\mu$ is said to be non-BAUP with parameter $\delta>0$ (or non-$\delta$-BAUP) if there exists a point $z_Q^a\in Q\cap\supp\mu$ such that for every affine hyperplane $L$ passing through $z_Q^a$ we can find a point $z_Q^b\in L\cap B(z_Q^a,\ell(Q))$ such that $B(z_Q^b,\delta \ell(Q))\cap \supp\mu=\varnothing.$
\end{defin}

 A geometric criterion for uniform rectifiability provided by David and Semmes (see \cite{david_semmes}) asserts that if,
 for any parameter $\delta>0,$ the cubes which are non-$\delta$-BAUP form a Carleson family, then $\mu$ is uniformly rectifiable.

\vspace{1mm}
To prove Theorem \ref{teo1} we will use the BAUP criterion. {\em We will assume that, for some $\delta>0$, the family of
non-BAUP cubes with parameter $\delta$ is non-Carleson and we will get a contradiction.} Our assumption implies that, for some $H\in\mathcal H$ and all $\ve>0$, $M>1$ (to be chosen below), the family 
$\mathcal A = \mathcal A(M,\varepsilon,H,N')$ of cubes $Q\in\D_\mu$ which are non-BAUP with parameter $\delta$ and contain a cube $Q'\in\D_\mu$  at most $N'$ levels down from $Q$ such that $\alpha^{(H)}_\mu(MB_{Q'})<\ve$ is also non-Carleson.
So we can apply Lemma \ref{altern_nb_flat} with this family $\mathcal A$ to construct the layers of cubes $\NB_k$ and $\FL_k$ with the parameters $\eta$ and $K$ in the lemma to be chosen below. 

 We remark now a property that will be used later on: for $R\in\FL_k$ and $Q\subset R$ such that $Q\in\NB_{k+1}$ for some $k$, 
we have 
\begin{equation}\label{Delta}
\ell(Q)\leq C\varepsilon\delta^{-1}\ell(R).
\end{equation}
In particular, for any $\Delta>0$, choosing $\varepsilon\delta^{-1}\ll\Delta$,  one has $\ell(Q)\ll\Delta\ell(R)$.

Let $R_0\in\D_\mu$ be as in Lemma \ref{altern_nb_flat}. We are interested in partitioning the collection of cubes contained in $R_0$ and below a suitable subfamily of cubes that we denote $\Top_1$ (see \eqref{def_top_k} for its definition) into subfamilies (the so-called \textit{trees}) with intermediate layers of non-$\delta$-BAUP cubes like in \cite{NTV_acta}. We proceed via a stopping time argument.

A collection  $\mathcal{T} \subset \D_\mu$ is a \textit{tree} if the following properties hold:
\begin{itemize}
\item $\mathcal{T}$ has a maximal element (with respect to inclusion) $Q(\mathcal{T})$ which contains all the other elements of $\mathcal{T}$ as subsets of $\Rn1$.
The cube $Q(\mathcal{T})$ is called the \textit{root} of $\mathcal{T}$.
\item If $Q, Q_0$ belong to $\mathcal{T}$ and $Q \subset Q_0$, then any cube $Q' \in \D_\mu$
such that $Q \subset Q' \subset Q_0$ also belongs to $\mathcal{T}$.
\item If $Q \in \mathcal{T}$, then either all cubes in $\Ch(Q)$ belong to $\mathcal{T}$ or none of them do.
\end{itemize}
Now we proceed to build the trees.
For $1\leq k\leq K-1$, we denote
\begin{equation}\label{def_top_k}
\Top_k:=\{Q\in\Ch(Q'): Q'\in \FL_k\}
\end{equation}
and, for $Q\in\Top_k$,
$${\NB(Q):=\{Q'\in \NB_{k+1}: Q'\subset Q\} \quad\mbox{ and } } \quad \Stop(Q):=\{Q'\in \FL_{k+1}: Q'\subset Q\}.$$
Note that { $\NB(Q)$ and $\Stop(Q)$ are finite families because $\NB_{k+1}$ and $\FL_{k+1}$ are} finite.

We write $\Top\coloneqq \bigcup_{k=1}^{K-1}\Top_k.$
Now, for every $Q\in\Top$ we let $\Tree(Q)$ be the collection of $\mu$-cubes which are contained in $Q$ and are not strictly contained in any cube from $\Stop(Q)$. Clearly $Q$ is the root of $\Tree(Q)$.



For $f\in L^2(\mu)$ and $Q\in\D_\mu$ we denote 
\begin{equation}\label{eqdq1}
\Delta_Q f=\sum_{S\in\Ch(Q)}m_{\mu,S}(f)\chi_S-m_{\mu,Q}(f)\chi_Q,
\end{equation}
so that we have the orthogonal expansion
$$\chi_{R_0} \bigl(f - m_{\mu,R_0}(f)\bigr) = \sum_{Q\in\DD_\mu:Q\subset R_0}\Delta_Q f,$$
in the $L^2(\mu)$-sense. 
Then, taking $f=T\mu$ (recall that this function makes sense because of Proposition
\ref{teoconvdebil})
and using the notation $T_R\mu\coloneqq \sum_{Q\in\Tree(R)}\Delta_Q T\mu$ for $R\in\Top,$ we can write
\begin{equation}
\int_{R_0} |T\mu - m_{\mu,R_0}(T\mu)|^2\,d\mu =\sum_{Q\in\DD_\mu:Q\subset R_0} \|\Delta_Q T\mu\|^2_{L^2(\mu)}\geq\sum_{R\in\Top}\|T_R\mu\|^2_{L^2(\mu)}.
\end{equation}
Since $T_\mu$ is bounded from $L^\infty(\mu)$ to $BMO(\mu)$,  the left hand side is bounded above by $\mu(R_0)$, and thus we get
\begin{equation}\label{eqgscu88}
\sum_{R\in\Top}\|T_R\mu\|^2_{L^2(\mu)} \leq C\,\mu(R_0).
\end{equation}

Let $0<\eta\ll 1$ (to be chosen later) be the parameter defining the lattice of alternating layers from Lemma \ref{altern_nb_flat}. Denote by $\Nice$ the subfamily of the cubes $R\in\Top$ such that
\begin{equation}\label{nice_cubes}
\sum_{Q\in \Stop(R)}\mu(Q)\geq (1-\eta^{1/2})\,\mu(R).
\end{equation}
The following easy lemma concerns the abundance of $\Nice$ cubes.

\begin{lemm} \label{lem768}
We have
\begin{equation}
\sum_{R\in\Top\setminus\Nice}\mu(R)\leq (K-1)\eta^{1/2}\mu(R_0).
\end{equation}
\end{lemm}

\begin{proof}
By construction, the cubes $R\in\Top\setminus \Nice$ satisfy
$$\mu(R) \leq \frac1{\eta^{1/2}}\,\mu\Bigg(R\setminus \bigcup_{Q\in\Stop(R)}Q\Bigg)\leq \frac1{\eta^{1/2}}\,
\mu\Bigg(R\setminus \bigcup_{Q\in\FL_K}Q\Bigg)$$
Thus, recalling that $\sum_{Q\in\FL_K}\mu(Q)\geq (1-\eta)\mu(R_0)$
and that there are $K-1$ layers of cubes in the family $\Top$,
we get
\begin{equation}
\begin{split}
\sum_{R\in\Top\setminus\Nice}\mu(R)&\leq \frac{1}{\eta^{1/2}}\sum_{R\in\Top\setminus\Nice}\mu\Bigg(R\setminus \bigcup_{Q\in\FL_K}Q\Bigg)\\
&\leq \frac{K-1}{\eta^{1/2}} \mu\Bigg(R_0\setminus \bigcup_{Q\in\FL_K}Q\Bigg)
\leq
\frac{(K-1)\eta}{\eta^{1/2}}\,\mu(R_0)= (K-1)\eta^{1/2}\mu(R_0).\qedhere
\end{split}
\end{equation}
\end{proof}
\vspace{1mm}

The main ingredient for the proof of Theorem \ref{teo1} is the following result.

\begin{prop}  \label{prop3}
Assume that $\ve$ and $\eta$ are chosen small enough in the construction of the alternating layers in Lemma
\ref{altern_nb_flat}, depending on $\delta$.
Then there is $c_1>0$ depending also on $\delta$ such that for every $R\in\Nice$ with $\ell(R)$ small enough we have 
\begin{equation}\label{eqnice28}
\|T_R\mu\|^2_{L^2(\mu)}\geq c_1\mu(R).
\end{equation}
\end{prop}
\vspace{1mm}

We remark that the smallness condition on the $\Nice$ cubes in the proposition depends just on $\delta$, the H\"older and elliptic conditions on the matrix $A$ and the AD-regularity of $\mu$.

\begin{proof}[\bf Proof of Theorem \ref{teo1} using Proposition \ref{prop3}]
By Lemma~\ref{lem768} and the property (6) in Lemma~\ref{altern_nb_flat}, assuming $\eta\leq 1/4$, we have
\begin{align}
\sum_{R\in\Nice}\mu(R) & \geq \sum_{R\in\Top}\mu(R) - (K-1)\eta^{1/2} \mu(R_0) \\
& \geq \sum_{k=1}^{K-1} \sum_{Q\in\FL_k}\mu(Q) - (K-1)\eta^{1/2}\,\mu(R_0) \\
& \geq (K-1)(1-\eta-\eta^{1/2}) \mu(R_0) \geq \frac14 (K- 1)\,\mu(R_0).
\end{align}

Denote by $\Nice'$ the family of $\Nice$ cubes $R$ which are small enough so that \rf{eqnice28} holds for them. Clearly
$$\sum_{R\in\Nice}\mu(R) \leq \sum_{R\in\Nice'}\mu(R) + C'\,\mu(R_0),$$
with $C'$ depending on the smallness condition for $R$ and on $\diam(\supp\mu)$.
By \rf{eqgscu88}, we have
$$\sum_{R\in\Nice'}\mu(R)\leq c_1(\delta)^{-1}\sum_{R\in\Nice'} \|T_R\mu\|^2_{L^2(\mu)}
\leq c_1(\delta)^{-1} C\,\mu(R_0).$$
Thus 
$$\frac14(K- 1)\,\mu(R_0)\leq C'\,\mu(R_0) + c_1(\delta)^{-1} C\,\mu(R_0).$$
So
we get a contradiction if $K$ is chosen big enough. Hence, the initial assumption that the family of non-$\delta$-BAUP cubes is not Carleson cannot be true.
\end{proof}

Proposition \ref{prop3} will be proved along the next Sections \ref{sec5}-\ref{sec12}.

\subsection{Scheme of the proof of Proposition \ref{prop3}}

We argue by contradiction, assuming that
\begin{equation}\label{contrad_hyp}
\|T_R\mu\|^2_{L^2(\mu)}\ll \mu(R).
\end{equation}
First, it is important to determine how $L_A$ and its associated objects transform under a change of variable. For this reason, we include the relevant formulas in \textit{Section \ref{sec5}}.

Then in \textit{Section \ref{sec6}} we show that it suffices to prove the proposition with the additional assumption $A(x_R)=Id$ and $H$ equal to the horizontal hyperplane through the origin; this 
puts us in a simpler geometric situation
and makes the other technicalities in the rest of the proof  more transparent.

A measure $\sigma$ supported on hyperplanes which approximate $\mu$ at the level of the children of cubes from $\Stop(R)$ is introduced in \textit{Section \ref{sec7}}.

In \textit{Section \ref{choosingdelta}} we construct the auxiliary matrix $\widehat A$, that we define via reflections with respect to a suitable hyperplane, and we study the gradient of its associated single layer potential $\widehat T_\mu$. We assume the hyperplane to be horizontal. In particular, we prove that the horizontal component of $\widehat T\sigma(x)$ is very close, in some $L^2(\sigma)$ sense, to that of $\widehat T\sigma(x^*)$, $x^*$ denoting the reflection of $x$ with respect to the horizontal plane. This proof relies on $R$ belonging to $\Nice$, the properties of $\widehat A$, and the contradiction hypothesis \eqref{contrad_hyp}.

\textit{Section \ref{secnou9}} and \textit{Section \ref{secpsi}} contain the definitions and the properties of a new approximating measure $\nu$, a vector field $\Psi$, and other mathematical objects important for the conclusion of the proof. In particular, we highlight that Section \ref{secpsi} uses the intermediate non-BAUP layers.

\textit{Section \ref{sec12}} concludes the proof of Proposition \ref{prop3} via a variational argument. This method produces a pointwise inequality that, integrated against the vector field $\Psi$ constructed in Lemma \ref{lemapsi}, gives the desired contradiction.


\section{The change of variable}\label{sec5}

The fact that we are considering a matrix $A$ which is uniformly elliptic and symmetric allows to perform a particular change of variables.
The following lemma and its corollary are standard. For the proofs we refer to \cite[Lemma 4.8]{AM}.
\begin{lemm}\label{lem:Pullback}
Let $\Omega \subset \R^{n+1}$ be an open set, and assume that $A$ is a uniformly elliptic matrix in $\Omega$ with real entries and $\phi:\R^{n+1} \to \R^{n+1}$ is a bi-Lipschitz map. If we set
$$A_\phi\coloneqq   |\det D{(\phi)}|\, D(\phi^{-1}) (A\!\circ\!\phi) D(\phi^{-1})^T,$$ 
where $D$ denotes the differential matrix,
then $A_\phi$ is a uniformly elliptic matrix in $\phi^{-1}(\Omega)$ and $u:\Omega\to\R$ is a weak solution of $L_{A} u =0$ in $\Omega$ if and only if $\wt{u}=u\circ \phi$ is a weak solution of $L_{A_\phi} \wt u =0$ in $\phi^{-1}(\Omega)$.
\end{lemm}


\begin{coroll}\label{cor:A(x0)=id}
Let $\Omega \subset \R^{n+1}$ be an open set, and assume that $A$ is a uniformly elliptic symmetric matrix in $\Omega$ with real entries. Let $O:\R^{n+1}\to\R^{n+1}$ be a rotation. For a fixed point $y_0 \in \Omega$ define $S= \sqrt{A(y_0)}\,O$. If  
\[
{A_S}(\cdot) = S^{-1} (A\circ S)(\cdot) (S^{-1})^T,
\]
then ${A}_S$ is uniformly elliptic in $S^{-1}(\Omega)$ and  $A_S(z_0) = Id$ for $z_0 = S^{-1}y_0$. Further, $u$ is a weak solution of $L_{A} u =0$ in $\Omega$ if and only if $\wt{u}=u\circ S$ is a weak solution of $L_{A_S} \wt u =0$ in $S^{-1}(\Omega)$ . 
\end{coroll}

In Corollary \ref{cor:A(x0)=id} we identified $S$ with its associated linear map. The matrix $S$ is well defined because $A$ is symmetric and uniformly elliptic, so that it admits a 
unique square root with the property of being symmetric, uniformly elliptic and having real entries. Further, we have
$$A_S(z_0) = (\sqrt{A(y_0)}\,O)^{-1} A(S(z_0))((\sqrt{A(y_0)}\,O)^{-1})^T = Id.
$$


Some standard linear algebra gives that $S^{-1}$ is a special bi-Lipschitz change of variables that takes balls to ellipsoids and its eigenvalues determine lengths of semi-axes. Denoting by $\lambda_{\max}$  and $\lambda_{\min}$ respectively the maximal and the minimal eigenvalues of $S^{-1}$, the maximum eccentricity of the image of a ball is 
$\sqrt{\lambda_{\max} / {\lambda_{\min}}}$. The ellipticity allows to bound it from below by $\sqrt{\Lambda}^{-1}$ and above by $\sqrt{\Lambda}$. 

It follows that  $\Lambda^{-1/2} \leq \|S^{-1} \|\leq   \Lambda^{1/2}$, so that $S^{-1}$ distorts distances by at most a constant depending on ellipticity. 
The collection $\widetilde\DD_{\mu}\coloneqq  \{ S^{-1}(Q) \}_{Q \in \DD_\mu}$  forms a dyadic grid on $S^{-1}(\supp\mu) =\supp (S^{-1}_\sharp\mu)$ of cubes of David-Semmes type, 
where the involved constants depend on the ones in $\DD_\mu$ and ellipticity.

The next easy lemma shows how the fundamental solution and the gradient of the single layer potential transform after a change of variable. 

\begin{lemm}\label{ChangeOfVar}
Let $\phi:\R^{n+1}\to\R^{n+1}$ be a locally bilipschitz map and let $\E_A$ be the fundamental solution of $L_A=-\Div(A\nabla\cdot).$ 
Set $A_\phi=|\det D(\phi)|\,D(\phi^{-1})(A\circ\phi)D(\phi^{-1})^T$. Then 
$$\E_{A_\phi}(x,y)=\E_A(\phi(x),\phi(y))$$
and $$\quad \nabla_1\E_{A_\phi}(x,y)= D(\phi)^T(x)\,\nabla_1\E_A(\phi(x),\phi(y))
\;\;\mbox{ for }x,\;y\in\R^{n+1}.$$ 
\end{lemm}
\begin{proof}
 The proof is an application of the change of variable formula for the integral. Let $f\in C^\infty_c(\Rn1)$. For every $x\in\Rn1$, the definition of fundamental 
 solution gives
\begin{equation}
f(\phi(x))=\int A(y)\nabla_2 \E_A(\phi(x),y)\cdot \nabla f(y)\,dy.
\end{equation}
Set  $E(x,y)\coloneqq \E_A(\phi(x),\phi(y)).$ If we denote $y'\coloneqq \phi^{-1}(y)$ and use the standard change of variable formula together with the chain rule, we get
\begin{equation}
\begin{split}
f(\phi(x))&=\int |\det D(\phi)(y')|\,A(\phi(y'))\nabla_2\E_A(\phi(x),\phi(y'))\cdot \nabla f(\phi(y'))\,dy'\\
&= \int |\det D(\phi)(y')\,|A(\phi(y'))D(\phi^{-1})^T(y')\nabla_2 E(x,y')\cdot D(\phi^{-1})^T(y')\nabla (f\circ \phi)(y')\,dy'\\
&=\int {A}_\phi(y')\nabla_2 E(x,y')\cdot \nabla (f\circ \phi)(y')\,dy',
\end{split}
\end{equation}
which proves the first identity in the lemma. The second identity follows from the chain rule.
\end{proof}

Define 
\begin{equation}\label{eq:tphinu}
T_{\phi}\nu(x)=\int\nabla_1\E_{A_\phi}(x,y)\,d\nu(y).
\end{equation}
Analogously, define the operator $T_{\phi,\nu}$ as in \eqref{eq:Tmudef}.
Then, by the previous lemma we have:

\begin{lemm}\label{lemfac*1}
Let $\phi:\R^{n+1}\to\R^{n+1}$ be a bilipschitz map, $\nu$ a Radon measure,  and $\phi_{\sharp}\nu$ its image measure. Then,
 $$T_{\phi}\nu(x)=D(\phi)^T(x)\,T\phi_{\sharp}\nu(\phi(x)).$$
\end{lemm}

\begin{proof}
The proof is an immediate application of Lemma \ref{ChangeOfVar} and the change of variable formula. Indeed
\begin{equation}
\begin{split}
&T_\phi\nu(x)=\int\n1\E_{{A_\phi}}(x,y)\,d\nu(y)=\int D(\phi)^T(x)\n1\E_A\big(\phi(x),\phi(y)\big)d\nu(y)\\
&\qquad =D(\phi)^T(x)\int\n1\E_A(\phi(x),z)\,d(\phi_{\sharp}\nu)(z)=D(\phi)^T(x)\, T\phi_{\sharp}\nu(\phi(x)).\qedhere
\end{split}
\end{equation}
\end{proof}


\section{Reduction to the case $A(x_R)=Id$ and $H$ horizontal}\label{sec6}
From now on, unless specified, we will denote by $R$ a given cube in $\Nice$.
In this section we will show that to prove Proposition \ref{prop3} we may assume that $A(x_R)=Id$ and the hyperplane $H$ in Lemma \ref{altern_nb_flat} to be horizontal. Indeed,
let $O:\R^{n+1}\to\R^{n+1}$ be a rotation which transforms the horizontal hyperplane (through the origin) $H'$ into $(\sqrt{A(x_R)})^{-1}H$. Consider
 the linear map $\phi\colon\R^{n+1}\to\R^{n+1}$ associated with the matrix $S= \sqrt{A(x_R)} \,O$ and, as in Corollary \ref{cor:A(x0)=id}, set
\[
{A}_\phi(\cdot) = S^{-1} (A\circ \phi)(\cdot) (S^{-1})^T,
\]
so that ${A}_\phi$ is uniformly elliptic and  $A_\phi(y_R) = Id$ for $y_R = S^{-1}x_R$.
Consider also the measure $\nu=(\phi^{-1})_\sharp\mu$ and the operator $T_\phi$ defined in \eqref{eq:tphinu}.
By Lemma \ref{lemfac*1}, 
\begin{equation}\label{eqff45}
T_{\phi}\nu(x)=S \cdot\,T\phi_{\sharp}\nu(\phi(x)) = S \cdot\,T\mu(\phi(x)).
\end{equation}
 Also, for any function $f$,
 $$T_{\phi}(f\nu)(x)=S \cdot\,T(\phi_{\sharp}(f\nu))(\phi(x)) = S \cdot\,T\big((f\circ \phi^{-1})\mu\big)(\phi(x)).$$
Therefore, by the $L^2(\mu)$-boundedness of $T_\mu$,
\begin{align}
\int \big|T_{\phi}(f\nu)(x)\big|^2\,d\nu(x) & \approx  \int \big|T((f\circ \phi^{-1})\mu)(\phi(x))\big|^2\,d \phi^{-1}_\sharp\mu(x)\\
& =  \int \big|T((f\circ \phi^{-1})\mu)(y)\big|^2\,d\mu(y)\\
& \leq  C\int |f\circ \phi^{-1}|^2\,d\mu = C\int |f|^2\,d\nu.
\end{align}
So $T_{\phi,\nu}$ is bounded in $L^2(\nu)$.

Let $\D_\nu$ be the lattice $\D_\nu=\{\phi^{-1}(Q):Q\in\D_\mu\}$.
Momentarily, use the notation $\Delta_Q^\mu$ instead of $\Delta_Q$, which we used in \eqref{eqdq1}, and define $\Delta_{Q'}^\nu$ analogously for $Q'\in\DD_\nu$.
Write also
$$T_{\phi,\phi^{-1}(R)} \nu = \sum_{Q\in\Tree(R)} \Delta_{\phi^{-1}(Q)}^\nu T_{\phi}\nu.$$
Assuming Proposition \ref{prop3} to hold in the case $A_\phi(y_R)=Id$ (applied to $\nu$ and $T_\phi$), we deduce that
\begin{equation}\label{eqcp2}
\|T_{\phi,\phi^{-1}(R)} \nu\|^2_{L^2(\nu)}\geq c\,\nu(\phi^{-1}(R))=c\,\mu(R),
\end{equation}
taking into account that the BAUP property is stable by homothecies, as well as the smallness of the $\alpha$-numbers for the stopping cubes and the root of the tree.

We claim that
\begin{equation}\label{eqcp3}
\|T_{\phi,\phi^{-1}(R)} \nu\|^2_{L^2(\nu)}\approx \|T_R \mu\|^2_{L^2(\mu)}
\end{equation}
with the implicit constant in \eqref{eqcp3} independent of the cube $R$.

Together with \eqref{eqcp2} this implies that $\|T_R \mu\|^2_{L^2(\mu)}\gtrsim \mu(R)$ and proves Proposition \ref{prop3}
in full generality. The proof of \eqref{eqcp3} is a routine task which we show now for the reader's convenience.  Observe
that for any cube $Q\in\D_\mu$, by \eqref{eqff45},
\begin{align}
m_{\nu,\phi^{-1}(Q)} (T_\phi\nu) & = \frac1{\nu(\phi^{-1}(Q))} \int_{\phi^{-1}(Q)} T_\phi\nu\,d\nu\\
& = \frac1{\mu(Q)} \int_{Q} T_\phi\nu(\phi^{-1}(x))\,d\mu(x)\\
&=  \frac1{\mu(Q)} \int_{Q}  S \cdot\,T\mu(x)\,d\mu(x) = S\cdot m_{\mu,Q}(T\mu).
\end{align}
Denote by ${\Ch}_{\Stop}(R)$ the family of all children of cubes from $\Stop(R)$. By the preceding identity, we obtain
\begin{align}
\|T_{\phi,\phi^{-1}(R)} \nu\|^2_{L^2(\nu)} &= \sum_{Q\in{\Ch}_{\Stop}(R)} \bigl|m_{\nu,\phi^{-1}(Q)} (T_\phi\nu)\,
  - m_{\nu,\phi^{-1}(R)}(T_\phi\nu)\bigr|^2\,\nu(\phi^{-1}(Q))\\
& = \sum_{Q\in{\Ch}_{\Stop}(R)} \bigl|S\cdot\bigl( m_{\mu,Q} (T\mu) - m_{\mu,R} (T\mu)\bigr)\bigr|^2\,\mu(Q)\\
& \approx  \sum_{Q\in{\Ch}_{\Stop}(R)} \bigl| m_{\mu,Q} (T\mu) - m_{\mu,R} (T\mu)\bigr|^2\,\mu(Q)\\
& = \|T_{R} \mu\|^2_{L^2(\mu)},
\end{align}
as claimed.

Remark also that if $\mu$ is well approximated in some cube $Q\in\DD_\mu$ by some measure of the form $c\HH^n|_L$, where $L$ is some hyperplane parallel to $H$, then it follows that
$\nu = (\phi^{-1})_\sharp \mu$ is well approximated in $\phi^{-1}(Q)$ by a measure of the form
$$\phi^{-1}_\sharp(c\HH^n|_L) = c'\HH^n|_{\phi^{-1}(L)}.$$
Observe that $\phi^{-1}(L)$ is a hyperplane parallel to the horizontal hyperplane $H'$, by the definition of $O$.
Using this fact, the reader can check that if $\alpha^{(H)}_\mu (MB_Q)<\varepsilon$, then $\alpha^{(H')}_\mu (\phi^{-1}(MB_Q))<c''\varepsilon$.


\section{The approximating measure}\label{sec7}


From now on, in order to prove Proposition \ref{prop3} for a given $R\in\Nice$, we assume that $A(x_R)=Id$ and that $H$ is the horizontal hyperplane through the origin.
Recall also that we assume $A$ to be symmetric.
In this section  we will construct a measure $\sigma$ which should be considered as an approximation of $\mu$,
in a sense. 

For every $Q\in{\Ch}_{\Stop}(R)$, $R\in\Nice$, let $L_Q$ be a hyperplane parallel to $H=\{x\in\R^{n+1}\colon x_{n+1}=0\}$ such that $\alpha_\mu^{L_Q}(MB_Q)\leq C\ve$.
Let $\tilde \ve,t>0$ be some parameters  to be chosen later, with $\ve\ll\tilde \ve\ll t\ll1$ and such that
$\beta_{\infty,\mu}^{L_Q}(MB_Q) + \beta_{\infty,\mu}^{L_Q}(B_Q)\leq\tve/10$ for all $Q\in{\Ch}_{\Stop}(R)$, $R\in\Nice$.

Denote $$Q_{(t)} = \{x\in Q:\dist(x,\supp\mu\setminus Q) \geq t\,\ell(Q)\}.$$ 
Now for $Q\in{\Ch}_{\Stop}(R)$ with $R\in\Nice$, set $\tilde\mu_Q=\mu|_{Q_{(t)}}$ and $\tmu=\underset{\tiny{Q\in{\Ch}_{\Stop}(R)}}
\sum\tmu_Q$. 
Let $\varphi$ be some $C^\infty$ radial
function supported on $B(0,1)$ such that $\int\varphi(x)d\HH^n|_H(x)=1$
and, for $r>0$, set $\varphi_{r}(x)=r^{-n}\varphi(x/r)$. Denote by $\Pi_{L_Q}$ the orthogonal
projection on $L_Q$, define 
$$\tsigma_Q=\Pi_{L_{Q}\sharp}\mu|_{Q_{(t)}}\quad \mbox{ and }\quad
\sigma_Q=(\tilde\sigma_Q*\varphi_{2\tve\ell(Q)})\HH^n|_{L_Q},$$
and then set
$$\sigma=\underset{\tiny{Q\in{\Ch}_{\Stop}(R)}}\sum\sigma_Q.$$ 

Observe that $\|\sigma_Q\|=\|\tilde\sigma_Q\|=\|\mu_{Q_{(t)}}\|$ for every $Q\in\Ch_{\Stop}(R)$, so
$$
\|\sigma\|=\|\tmu\|.
$$
Moreover, using the thin boundary condition and the abundance parameter $\eta$,
\begin{equation}\label{eqdif79}
\|\tmu-\mu|_R\|\leq \mu\biggl(R\setminus\bigcup_{Q\in{\Ch}_{\Stop}(R)} Q\biggr) + \sum_{Q\in{\Ch}_{\Stop}(R)}
\mu(Q\setminus Q_{(t)}) \leq \eta^{1/2} \mu(R)+ Ct^{\gamma_0}\mu(R)\lesssim t^{\gamma_0}\mu(R),
\end{equation}
taking $\eta\ll t$.

Note also that, for each $Q\in{\Ch}_{\Stop}(R)$,  by the definition of $\sigma_Q$, 
\begin{equation}\label{eqsigma75}
\supp\sigma_Q \subset \mathcal U_{3\tve\ell(Q)}(\supp \Pi_{L_Q\sharp}\mu|_{Q_{(t)}})
\subset \mathcal U_{3\tve\ell(Q)}(\mathcal U_{3\tve\ell(Q)}(Q_{(t)})) = \mathcal U_{6\tve\ell(Q)}(Q_{(t)}).
\end{equation}
As a consequence, for $P,Q\in{\Ch}_{\Stop}(R)$ with $P\neq Q$, we have
\begin{equation}\label{eqdistsigma}
\begin{split}
\dist(\supp\sigma_P,\supp\sigma_Q) &\geq \dist\bigl(\mathcal U_{6\tve\ell(P)}(P_{(t)}),\mathcal U_{6\tve\ell(Q)}(Q_{(t)})\bigr)\geq \dist(P_{(t)}, Q_{(t)}) - 6\tve\,(\ell(P)+\ell(Q))\\& \geq t\,\max(\ell(P),\ell(Q)) - 6\tve\,(\ell(P)+\ell(Q))
\geq \frac t2\,\max(\ell(P),\ell(Q)).
\end{split}
\end{equation}

We will need the following lemma:

\begin{lemm}\label{dif}
Let $Q\in{\Ch}_{\Stop}(R)$. If $f\in\Lip_\alpha\Big(\UU_{10\tve\ell(Q)}\big(Q_{(t)}\big)\Big)$, 
then $$\left|\int f(x)d(\sigma_Q-\tmu_Q)(x)\right|\lesssim M^{\alpha}\Lip_\alpha(f)\tve^{\alpha}\ell(Q)^\alpha\mu(Q).$$ 
\end{lemm}
\begin{proof}
Write 
$$\left|\int fd(\sigma_Q-\tmu_Q)\right|\leq\left|\int fd(\sigma_Q-\tsigma_Q)\right|+\left|\int fd(\tsigma_Q-\tmu_Q)\right|= T_1+T_2.$$

By the definition of $\tsigma_Q$ and the fact that $Q\in\Ch_{\Stop}(R)$ and therefore $\beta_{\infty,\mu}(MB_Q)\leq\tve$,
$$T_2=\left|\int\big(f(\Pi_{L_Q}(x))-f(x)\big)d\tmu_Q(x)\right|\lesssim M^{\alpha}\Lip_\alpha(f)\tve^\alpha\ell(Q)^\alpha\mu(Q).$$

For the other term, by Fubini
\begin{equation}
\begin{split}
T_1&=\left|\int f(y)d\tsigma_Q(y)-\int f(y)d(\tsigma_Q*\varphi_{2\tve\ell(Q)})\HH^n|_{L_Q}(y)\right|\\&
=\left|\int\big( f(y)-f*(\varphi_{2\tve\ell(Q)}\HH^n|_{L_Q})(y)\big)d\tsigma_Q(y)\right|\\&\leq\int\sup_{|z|\leq 2\tve\ell(Q)}\big|f(y)-f(y+z)\big|d\tsigma_Q(y)
\lesssim M^{\alpha}\Lip_\alpha(f)\tve^\alpha\ell(Q)^\alpha\mu(Q).\qedhere
\end{split}
\end{equation}
\end{proof}

Next we show that $\sigma$ has $n$-growth.

\begin{lemm}\label{lemgrowthsigma}
The measure $\sigma$ has polynomial growth of degree $n$. That is,
$$\sigma(B(x,r))\leq C\,r^n\quad\mbox{ for all $x\in\R^{n+1}$, $r>0$.}$$
\end{lemm}

\begin{proof}
First we will check that $\sigma_Q$ has $n$-growth for each $Q\in{\Ch}_{\Stop}(R)$. Denoting $g_Q=\tilde\sigma_Q*\varphi_{2\tve\ell(Q)}$ and since $\sigma_Q=g_Q\,\HH^n|_{L_Q}$, this is equivalent to showing that $\|g_Q\|_\infty\lesssim1$.
To prove this, for $x\in L_Q$, using that $\Pi_{L_Q}(x)=x$, we write
\begin{align}
g_Q(x) & = \int \varphi_{2\tve\ell(Q)}(x-y)\,d\Pi_{L_Q\sharp}\mu|_{Q_{(t)}}(y)  = \int \varphi_{2\tve\ell(Q)}(x-\Pi_{L_Q}(y))\,d\mu|_{Q_{(t)}}(y) \\
& = \int_{Q_{(t)}} \bigl(\varphi_{2\tve\ell(Q)}\circ\Pi_{L_Q}\bigr)(x-y)\,d\mu(y) \lesssim \frac1{(\tve\ell(Q))^n}\,\mu\Bigl(Q \cap \Pi_{L_Q}^{-1}\bigl(B(x,2\tve\ell(Q))\bigr)\Bigr).
\end{align} 
Since $\beta_{\infty,\mu}^{L_Q}(B_Q)\leq\tve/10$, there is some constant $C$ depending at most on $n$ such that 
$$\mu\Bigl(Q \cap \Pi_{L_Q}^{-1}\bigl(B(x,2\tve\ell(Q)\bigr)\Bigr)\leq \mu(B(x,C\,\tve\ell(Q)))\lesssim (\tve\ell(Q))^n,$$
which ensures that $\|g_Q\|_\infty\lesssim1$, as wished.

Next, for a fixed ball $B(x,r)$, let $I$ be the family of cubes $Q\in{\Ch}_{\Stop}(R)$ such that $2B_Q\cap B(x,r)\neq\varnothing$.
We split $I=I_1\cup I_2$, where $I_1$ is the subfamily of the cubes from $I$ with side length at most $r$ and $I_2=I\setminus I_1$.
Then we have
$$\sigma(B(x,r))\leq \sum_{Q\in I_1}\|\sigma_Q\| + \sum_{Q\in I_2} \sigma_Q(B(x,r)).$$
For each $Q\in I_1$, we have $\supp\sigma_Q\subset 2B_Q\subset B(x,4r)$, and thus
$$\sum_{Q\in I_1}\|\sigma_Q\|\leq C\sum_{Q\in I_1}\mu(Q) \leq C \mu(B(x,4r))\leq C\,r^n.$$
On the other hand, it is immediate to check that there is a bounded number of cubes $Q\in I_2$, with the bound depending on the 
parameters of the lattice $\D_\mu$ and thus on the
AD-regularity constant of $\mu$. Hence, using also the $n$-growth of $\sigma_Q$,
$$\sum_{Q\in I_2} \sigma_Q(B(x,r))\leq C\sum_{Q\in I_2}r^n\leq C\,r^n,$$
which completes the proof of the lemma.
\end{proof}


\section{Approximation argument and reflection}\label{choosingdelta}

\subsection{The matrix $\wh A$ and its associated operators $\wh T$ and $S$}


Recall that we assume that $A$ is a symmetric matrix such that $A(x_R)=Id$. Given a parameter $\Delta\in (0,1/10)$ to be chosen below, we set 
$d=\Delta \ell(R)$ and we assume that a ``good'' approximating hyperplane for $\supp\mu\cap B_R$ is $L_R=\{x\in\R^{n+1}\colon x_{n+1}=2d\}$. That is, $\alpha_\mu^{L_R}(MB_R)\leq\ve$. We also take $H=\{x\in\R^{n+1}\colon x_{n+1}=0\}$, so that $L_R$ is a translation of $H$ along the $(n+1)$-th direction. Further, we suppose that
$B_R\subset B(0,2\ell(R))$.

Given $x\in\R^{n+1}$ we denote by $x^*$ the reflection of $x$ with respect to $H$, that is $x^*=(x_1,x_2,\ldots,x_n,-x_{n+1})$.
Now we define a matrix $\wh A$ which satisfies some kind of invariance under this reflection. First, we consider an auxiliary matrix
$B$ defined on $\{x:x_{n+1}\geq 0\}$ by
$$\label{auxiliary_matrix_B}
B(x)=\left\{\begin{array}{l}
                A(x)\;\;\;\;\mbox{ if }x_{n+1}\geq d\,\\\\
                A(x)\frac{x_{n+1}}{d}+Id\left(1-\frac{x_{n+1}}{d}\right)\;\;\;\;\mbox{ if }0\leq x_{n+1}\leq d.
               \end{array}\right.$$
Notice that $B(0)=Id$.
For $x_{n+1}<0$ we set
$$B(x) = \left(\;\;\begin{matrix} 
                     b_{1,1}(x^*)&\cdots& b_{1,n}(x^*) & -b_{1,n+1}(x^*)\\
                      b_{2,1}(x^*)&\cdots& b_{2,n}(x^*) &-b_{2,n+1}(x^*)\\
                      \vdots&\ddots&\vdots&\vdots\\
                       b_{n,1}(x^*)&\cdots& b_{n,n}(x^*) &-b_{n,n+1}(x^*)\\
                      -b_{n+1,1}(x^*)&\cdots&-b_{n+1,n}(x^*) &\hspace{.65cm}b_{n+1,n+1}(x^*)
                     \end{matrix}\right),
$$
where $b_{ij}(x^*)$ are the coefficients of $B(x^*)$.
In this way, for $\phi(x) = x^*$, it holds 
$$B = 
|\det D(\phi)|\, D(\phi^{-1}) (B\!\circ\!\phi) D(\phi^{-1})^T.$$
Observe that 
$$D(\phi^{-1}) = D(\phi^{-1})^T=
\left(\;\;\begin{matrix} 
                     1&0&\cdots&0& 0\\
                     0&1&\cdots&0& 0\\
                      \vdots&\vdots&\ddots&\vdots&\vdots\\
                      0&0&\cdots&1&0\\
                      0&0&\cdots&0&\!\!\!\!-1
                     \end{matrix}\right).
$$
Next we define
$$\wh A(x)=\left\{\begin{array}{l}
                B(x)\;\;\;\;\mbox{ if }|x|\leq 100\ell(R),\\\\
                 \left(2 - \frac{|x|}{100\ell(R)}\right)B(x) \,+\left(\frac{|x|}{100\ell(R)}-1\right)Id\;\;\;\;\mbox{ if }100\ell(R)\leq |x|\leq 200\ell(R)\\\\
                Id\;\;\;\;\mbox{ if }|x|\geq 200\ell(R)
               \end{array}\right.$$
Note that, for $\phi(x) = x^*$, we still have 
$$\wh A = 
|\det D(\phi)|\, D(\phi^{-1}) (\wh A\!\circ\!\phi) D(\phi^{-1})^T.$$
So, denoting  $D = D(\phi^{-1}) = D(\phi^{-1})^T$, we have
\begin{equation}\label{eqref54}
\wh A(x)  = 
D\,\wh A(x^*)\,D.
\end{equation}

\begin{lemm}\label{holder/2}
 For $\ell(R)$ small enough, the matrix $\wh A(x)$ just defined is H\"older continuous with exponent $\alpha/2$ in $R$.
\end{lemm}

\begin{proof} 
As a first step, we prove that the auxiliary matrix $B$ defined above is $C^{\alpha/2}$ inside the ball $B(0,200 \ell(R))$.
Because of the definition of $B$, it suffices to check the H\"older regularity condition for $0\leq x_{n+1},y_{n+1}\leq d.$ In this case
 \begin{align}
 \big|B(x)-B(y)\big|&=\left|A(x)\frac{x_{n+1}}{d}-A(y)\frac{y_{n+1}}{d}+ Id\left(1-\frac{x_{n+1}}{d}\right)- Id\left(1-\frac{y_{n+1}}{d}\right)\right|\\
 &=\left|(A(x)-Id)\frac{x_{n+1}}{d}-(A(y)-Id)\frac{y_{n+1}}{d}\right|\\
 & \leq 
\frac{|x_{n+1}-y_{n+1}|}d\, |A(x)-Id| + \frac{y_{n+1}}d\,|A(x)-A(y)|\\
& \leq C\frac{|x_{n+1}-y_{n+1}|}d\, \ell(R)^\alpha +
 C\,|x-y|^{\alpha},
 \end{align}
where we took into account that
$$|A(x)-Id| = |A(x)-A(x_R)|\leq C\,\ell(R)^\alpha.$$
 Now we write
$$ \frac{|x_{n+1}-y_{n+1}|}d\, \ell(R)^\alpha\leq \frac{|x_{n+1}-y_{n+1}|^\alpha}{d^\alpha}\, \ell(R)^\alpha
 = \frac1{\Delta^\alpha}\,|x_{n+1}-y_{n+1}|^\alpha \leq \frac{\ell(R)^{\alpha/2}}{\Delta^\alpha}\,|x_{n+1}-y_{n+1}|^{\alpha/2}.$$
 Thus for $\ell(R)$ small enough, we have ${\ell(R)^{\alpha/2}}/{\Delta^\alpha}\leq1$ and we get
 $$ \big|B(x)-B(y)\big|\leq C\,|x-y|^{\alpha/2} + C\,|x-y|^\alpha\leq  C\,|x-y|^{\alpha/2},$$
 since $|x-y|\lesssim \ell(R)\lesssim 1$. This proves the $(\alpha/2)$-H\"older regularity in the ball $B(0,200\ell(R))$.
 
The next step is to prove that the matrix $\wh A$ is $C^{\alpha/2}$ inside the ball $B(0,200\ell(R)).$ The regularity inside $B(0,100\ell(R))$ follows from the regularity of $B$. Consider $x,y\in B(0,200\ell(R))\setminus B(0,100 \ell(R)).$ Exploiting the definition of $\wh A$ together with the H\"older regularity of the matrix $B$ inside $B(0,200\ell(R))$ we have
\begin{equation}
\begin{split}
\big|\wh A(x)-\wh A(y)\big|&=\Big|\Big(2-\frac{|x|}{100\ell(R)}\Big)B(x)-\Big(2-\frac{|y|}{100\ell(R)}\Big)B(y)+\Big(\frac{|x|-|y|}{100\ell(R)}\Big)Id\Big|\\
&\leq 2|B(x)-B(y)|+\Big|(B(x)-Id)\,\frac{|x|}{100\ell(R)}-(B(y)-Id)\,\frac{|y|}{100\ell(R)}\Big|\\
&\leq 2|B(x)-B(y)|+|B(x)- Id|\,\frac{\big||x|-|y|\big|}{100\ell(R)}+|B(x)-B(y)|\,\frac{|y|}{100\ell(R)}\\
&\leq C|x-y|^{\alpha/2}+ |B(x)-Id|\,\frac{|x-y|}{100\ell(R)}+ C|x-y|^{\alpha/2}\frac{|y|}{100\ell(R)}
\end{split}
\end{equation}
so that, being $x,y\in B(0,200\ell(R))$ and $B(0)=Id$, we can write
\begin{equation}
\begin{split}
\big|\wh A(x)-\wh A(y)\big|&\leq C|x-y|^{\alpha/2}+ C|x|^{\alpha/2}\frac{|x-y|}{100\ell(R)}+ C|x-y|^{\alpha/2}\frac{|y|}{100\ell(R)}\\
&\leq C|x-y|^{\alpha/2}+ C\ell(R)^{\alpha/2}\,\frac{|x-y|^{\alpha/2}}{\ell(R)^{\alpha/2}}+ C|x-y|^{\alpha/2}\leq C|x-y|^{\alpha/2}.
\end{split}
\end{equation}
The matrix $\wh A$ is trivially $C^{\alpha/2}$ in $\Rn1\setminus B(0,200\ell(R))$. To finish the proof, take $x$ with $|x|\leq 200\ell(R)$, $y$ with $|y|\geq 200\ell(R)$ and 
choose a point $\tilde{y}$ with $|\tilde{y}|=200\ell(R)$ and $|x-\tilde{y}|\leq |x-y|$. Then write
\begin{equation}
\begin{split}
\big|\wh A(x)-\wh A(y)\big|=\big|\wh A(x)-Id|=\big|\wh A(x)-\wh A(\tilde{y})\big|\leq C|x-\tilde{y}|^{\alpha/2}\leq C|x-y|^{\alpha/2}. \qedhere
\end{split}
\end{equation}
\end{proof}

From now on we assume that $\ell(R) \leq 1$ so that the estimates in Lemma \ref{lemcz} hold for all $x,y\in R$. 
Also, the estimates in Lemma \ref{lemcz} and  Lemma \ref{lemm_freezing} hold for $\wh A$ with $\alpha/2$ replacing $\alpha$. Further, we will 
take $0<\ve\ll \Delta\ll1$, so that $A(x) = \wh A(x)$ for all $x$ in a neighborhood of $R$.

Let $\E_{\wh A}$ be the fundamental solution associated with $L_{\wh A}$, set 
$\wh K(x,y)=\n1\E_{\wh A}(x,y)$, and define
$$\wh T\mu(x)=\int\wh K(x,y)d\mu(y).$$
Note that, by \rf{eqref54}
and Lemma \ref{ChangeOfVar}, for $x,\;y\in\R^{n+1}$ we have
\begin{equation}\label{eqref55}
\E_{\wh A}(x,y)=\E_{\wh A}(x^*,y^*)
\quad\text{ and } \quad \wh K(x,y)= \nabla_1\E_{\wh A}(x,y)= D\,\nabla_1\E_{\wh A}(x^*,y^*)=D\,\wh K(x^*,y^*)
. 
\end{equation}

Define now the operator 
$$S{\mu}(x)=\int K_S(x,y)d\mu(y),$$
associated with the kernel  $$K_S(x,y)=\wh K(x,y)-\wh K(x^*,y),$$
so that
$$S\mu(x) = \wh T\mu(x) - \wh T\mu(x^*).$$

\begin{rem}\label{reml2ts}
The operators $\wh T_\mu$ and $S_\mu$ are bounded in $L^2(\mu|_R)$.
Indeed, the $L^2(\mu|_R)$ boundedness of $\wh T_\mu$ follows from the one of $T_\mu$ and the fact that the difference between their kernels is bounded in modulus by $1/|x-y|^{n-\alpha/2}$, by a freezing argument using Lemma \ref{lemm_freezing}.
 Then to prove the $L^2(\mu|_R)$ boundedness
of $S_\mu$ it suffices to show that the operator $U_\mu$ defined by 
$$U_\mu f(x) = \wh T_\mu f(x^*)$$
is bounded in $L^2(\mu|_R)$. To show this, write
$$\int_R |U_\mu f(x)|^2\,d\mu(x) = \int_R |\wh T_\mu f(x^*)|^2\,d\mu(x) = \int |\wh T_\mu f(y)|^2\,d\phi_\sharp\mu (y),$$
where $\phi_\sharp \mu$ is the image measure of $\mu|_R$ by the reflection $\phi:x\mapsto x^*$.
Since  $\wh T_\mu$ is bounded in $L^2(\mu|_R)$ and $\phi_\sharp\mu$ has $n$-polynomial growth, it follows that
$\wh T_\mu$ is bounded from $L^2(\mu|_R)$ to $L^2(\phi_\sharp\mu)$, which implies that $U_\mu$ is bounded in $L^2(\mu|_R)$, as wished.
\end{rem}

\vspace{2mm}
Recall that $H=\{x\colon x_{n+1}=0\}$. We denote by $\Pi_H$ the orthogonal projection on $H$, we set
$$\wh T^H\mu(x) = \Pi_H(\wh T\mu(x)),\quad S^H\mu(x) = \Pi_H(S\mu(x)),$$
and we define similarly  $\wh T^H_\mu$, $S^H_\mu$, etc. The kernel of $\wh T^H$ is $\wh K^H(x,y) \coloneqq  \Pi_H(\wh K(x,y))$ and the one of $S^H$ is $K_S^H(x,y)\coloneqq\Pi_H(K_S(x,y))$. Note that, from the second identity in \rf{eqref55}, we get
\begin{equation}\label{eqref56}
\wh K^H(x,y) = \wh K^H(x^*,y^*)\quad \mbox{ for all $x,y\in\R^{n+1}$ with $x\neq y$.}
\end{equation}

\vspace{2mm}

\subsection{The approximation lemmas}

This section is devoted to announce some technical approximation lemmas.

\begin{lemm}[First Approximation Lemma] \label{lemapprox0}
For every $R\in\Nice$ we have
\begin{equation}\label{eqsh421}
\| \wh T\sigma\|^2_{L^2(\sigma)}\leq C\,\mu(R)
\end{equation}
and
\begin{equation}\label{eqsh422}
\| S\sigma\|^2_{L^2(\sigma)}\leq C\,\mu(R).
\end{equation}
\end{lemm}

For the horizontal operator $S^H$ we have a much better estimate:

\begin{lemm}[Second Approximation Lemma] \label{lemapprox}
Let $R\in\Nice$.
Let $\ve_1,\ve_2>0$ and suppose that $\|T_R{\mu}\|^2_{L^2(\mu)}\leq \ve_1\,\mu(R)$.
Then 
\begin{equation}\label{eqsh423}
\| S^H\sigma\|^2_{L^2(\sigma)}\leq \varepsilon_2\,\mu(R),
\end{equation}
if $\ve_1$, $\ell(R)$, $t$, and $\Delta$ are small enough and $M$ is big enough.
\end{lemm}

Essentially, the estimates in the above lemmas hold because $\sigma$ is a very good approximation of the measure
$\mu$ at the scales and location of $\Tree(R)$. Further, in the case of Lemma \ref{lemapprox} the 
reflection involved in the definition of $S$ plays an essential role in the localization that allows to transfer the
estimates
from the measure $\mu$ to the compactly supported measure $\sigma$ with a small error.

\vspace{2mm}

The proof of Lemmas \ref{lemapprox0} and \ref{lemapprox} follows from the next three auxiliary lemmas.

\begin{lemm}\label{lemftr1}
Let $R\in\Nice$ and let  $$f = \sum_{Q\in{\Ch}_{\Stop}(R)} m_{\mu,Q}(S\mu)\,\chi_Q\quad \mbox{and}\quad f^H = \sum_{Q\in{\Ch}_{\Stop}(R)} m_{\mu,Q}(S^H\mu)\,\chi_Q.$$
Then, 
\begin{equation}\label{eqsh431}
\|T_R\mu-f\|_{L^2(\mu)}^2\lesssim \mu(R),
\end{equation}
and, for any $\ve_3>0$,
\begin{equation}\label{eqsh432}
\|T_R^H\mu-f^H\|_{L^2(\mu)}^2\leq \ve_3\,\mu(R),
\end{equation}
if  $\ve$, $\tilde \ve$, and $\ell(R)$ are small enough and $M$ is big enough.
\end{lemm}

\begin{lemm}\label{lemftr1.5}
Let $R\in\Nice$, denote  $$\tilde f = \sum_{Q\in{\Ch}_{\Stop}(R)} m_{\mu,Q}(S\tmu)\,\chi_Q\quad \mbox{and}\quad \tilde f^H = \sum_{Q\in{\Ch}_{\Stop}(R)} m_{\mu,Q}(S^H\tmu)\,\chi_Q,$$
and let $f$, $f^H$ be as in Lemma \ref{lemftr1}.
Then, for any $\ve_4>0$, if $t$ and $\Delta$ are small enough,
$$\|\tilde f - f\|_{L^2(\mu)}^2 + \|\tilde f^H- f^H\|_{L^2(\mu)}^2\leq \ve_4\,\mu(R).$$
\end{lemm}

\begin{lemm}\label{lemftr2}
Let $R\in\Nice$ and $\tilde f,\tilde f^H$ be as in Lemmas \ref{lemftr1} and \ref{lemftr1.5}.
Also, set
$$\tilde h = \sum_{Q\in{\Ch}_{\Stop}(R)} m_{\mu,Q}(\wh T\tilde \mu)\,\chi_Q.$$
Then, for any $\ve_5>0$ we have
\begin{equation}\label{eqsh452}
\|\wh T\sigma\|^2_{L^2(\sigma)}\leq C\,\|\tilde h\|_{L^2(\mu)}^2 + \ve_5\,\mu(R),
\end{equation}
\begin{equation}\label{eqsh453}
\|S\sigma\|^2_{L^2(\sigma)}\leq C\,\|\tilde f\|_{L^2(\mu)}^2 + \ve_5\,\mu(R),
\end{equation}
and
\begin{equation}\label{eqsh454}
\|S^H\sigma\|^2_{L^2(\sigma)}\leq C\,\|\tilde f^H\|_{L^2(\mu)}^2 + \ve_5\,\mu(R)
\end{equation}
if $\ve$, $\tilde \ve$, $t$ and $\ell(R)$ are small enough.
\end{lemm}

\begin{proof}[Proof of the Approximation Lemmas \ref{lemapprox0}, \ref{lemapprox} using Lemmas
\ref{lemftr1}, \ref{lemftr1.5}, \ref{lemftr2}]
The estimates \rf{eqsh422} and \rf{eqsh423} follow just by an immediate application of the three auxiliary
lemmas and the triangle inequality. For example, to show \rf{eqsh423},
assume $\|T_R{\mu}\|^2_{L^2(\mu)}\leq \varepsilon_1\,\mu(R)$ and 
then by \rf{eqsh454}, \rf{eqsh432}, and Lemma \ref{lemftr1.5},
\begin{equation}\label{prooftriangleineq}
 \begin{split}
  \|S^H\sigma\|_{L^2(\sigma)}^2&\leq C\|\tilde f^H\|^2_{L^2(\mu)}+\ve_5\mu(R)\\&\leq  
  C\|T_R^H\mu\|^2_{L^2(\mu)}+C\|T_R^H\mu-f^H\|^2_{L^2(\mu)}+ C \|f^H - \tilde f^H\|^2_{L^2(\mu)}+ \ve_5\mu(R)\\&\lesssim (\ve_1+\ve_3+\ve_4+\ve_5)\,\mu(R).
 \end{split}
\end{equation}
The proof of \rf{eqsh422} is analogous.

To show \rf{eqsh421} 
we just apply \rf{eqsh452} and use the fact that $\wh T_\mu$ is bounded in $L^2(\mu|_R)$:
\begin{align}
\| \wh T\sigma\|^2_{L^2(\sigma)} &\leq C\,\|\tilde h\|_{L^2(\mu)}^2 + \ve_5\,\mu(R) =
C \,\Big\|\sum_{Q\in{\Ch}_{\Stop}(R)} m_{\mu,Q}(\wh T\tilde \mu)\,\chi_Q\Big\|_{L^2(\mu)}^2 + \ve_5\,\mu(R)\\
&\leq C\,\|\wh T\tilde \mu\|_{L^2(\mu|_R)}^2 +  \ve_5\,\mu(R)\lesssim\mu(R).\qedhere
\end{align}
\end{proof}

\vspace{2mm}


\subsection{Proof of Lemma \ref{lemftr1}}

First we set $\wh T^{\phi}\mu(x)=\wh T\mu(\phi(x))=\wh T\mu(x^*)$ and $\wh T^{\phi,H}\mu(x)=\wh T^H\mu(\phi(x))=\wh T^H\mu(x^*)$, so that $S\mu(x)=\wh T\mu(x)-\wh T^{\phi}\mu(x)$ and $S^H\mu(x)=\wh T^H\mu(x)-\wh T^{\phi,H}\mu(x).$ In what follows we write $m_Q (f)=m_{\mu,Q}(f)$ to simplify the
notation. Denote by $x_R'$ the orthogonal projection of $x_R$ on $L_R$. Notice that
$$|x_R'-x_R|\lesssim \alpha_\mu^{L_R}(2B_R)^{1/(n+1)}\ell(R)\lesssim (M^{n+1}\ve)^{1/(n+1)}\ell(R)\ll\ell(R).$$
Consider a $C^1$ function $\wt\chi_{M,R}$, radial with respect to $x_R'$, and such that $\chi_{B(x_R',M\ell(R)/2)}\leq \wt\chi_{M,R}\leq
\chi_{B\big(x_R',\frac34M\ell(R)\big)}$ and $\|\nabla\wt\chi_{M,R}\|_{\infty}\lesssim(M\ell(R))^{-1}$.
For $x\in Q\in{\Ch}_{\Stop}(R)$ and $M>1$, we split the difference $T_R\mu(x)-f(x)$ as follows:
\begin{equation}
 \begin{split}
 T_R\mu(x)-f(x)&=m_Q(T\mu)-m_R(T\mu)-m_Q(S\mu)\\&=m_Q(T\mu)-m_R(T\mu)-m_Q(\wh T\mu)+m_Q(\wh T^{\phi}\mu)\\&=
 m_Q(T_\mu\wt\chi_{M,R})+m_Q(T_\mu(1-\wt\chi_{M,R}))-m_R(T_\mu\wt\chi_{M,R})\\
 &\quad -m_R(T_\mu(1-\wt\chi_{M,R}))-m_Q(\wh T_{\mu}\wt\chi_{M,R})-m_Q(\wh T_{\mu}(1-\wt\chi_{M,R}))\\&\quad+
 m_Q(\wh T^{\phi}_{\mu}\wt\chi_{M,R})+m_Q(\wh T^{\phi}_{\mu}(1-\wt\chi_{M,R})),
 \end{split}
\end{equation}
so that we have
\begin{equation}\label{8terms}
 \begin{split}
 \bigl|T_R\mu(x)-f(x)\bigr| &\leq |m_Q(T_\mu\wt\chi_{M,R})-m_Q(\wh T_{\mu}\wt\chi_{M,R})|\\&\quad + |m_Q(T_\mu(1-\wt\chi_{M,R}))-m_R(T_\mu(1-\wt\chi_{M,R}))| \\&\quad + |m_Q(\wh T^{\phi}_{\mu}(1-\wt\chi_{M,R}))-m_Q(\wh T_{\mu}(1-\wt\chi_{M,R}))| 
 \\&\quad  +|m_R(T_\mu\wt\chi_{M,R})|+
 |m_Q(\wh T^{\phi}_{\mu}\wt\chi_{M,R})|\\ & =: I_1+ I_2 + I_3 + I_4 + I_5.
 \end{split}
\end{equation}
We perform the analogous splitting for $\bigl|T_R^H\mu(x)-f^H(x)\bigr|$, so that
we have 
$$\bigl|T_R^H\mu(x)-f^H(x)\bigr| \leq I_1^H+ I_2^H + I_3^H + I_4^H + I_5^H,$$
with 
\begin{align}
I_1^H &= |m_Q(T_\mu^H\wt\chi_{M,R})-m_Q(\wh T_{\mu}^H\wt\chi_{M,R})|,\\
I_2^H &=|m_Q(T_\mu^H(1-\wt\chi_{M,R}))-m_R(T_\mu^H(1-\wt\chi_{M,R}))|, \\
I_3^H &= |m_Q(\wh T^{\phi,H}_{\mu}(1-\wt\chi_{M,R}))-m_Q(\wh T_{\mu}^H(1-\wt\chi_{M,R}))|,\\
I_4^H &= |m_R(T_\mu^H\wt\chi_{M,R})|,\\
I_5^H & = |m_Q(\wh T^{\phi,H}_{\mu}\wt\chi_{M,R})|
.\end{align}
 Obviously, $I_i^H\leq I_i$ for each $i$.

\vspace{2mm}
\noindent{\bf Estimate of $I_2$.}
Notice that for $x'\in Q$,
\begin{equation}
\begin{split}
|m_Q(T_\mu(1-\wt\chi_{M,R}))-m_R(T_\mu &(1-\wt\chi_{M,R}))|\\&\leq \frac1{\mu(Q)}\int_Q\left|T_\mu(1-\wt\chi_{M,R})(x)-T_\mu(1-\wt\chi_{M,R})(x')\right|d\mu(x)\\&\quad+
\frac1{\mu(R)}\int_R\left|T_\mu(1-\wt\chi_{M,R})(x)-T_\mu(1-\wt\chi_{M,R})(x')\right|d\mu(x)\\&\leq 2\sup_{y,y'\in R}|T_\mu(1-\wt\chi_{M,R})(y)-T_\mu(1-\wt\chi_{M,R})(y')|
\end{split}
\end{equation}
and to estimate this supremum, observe that for $y,y'\in R$, by Lemma \ref{lemcz},
\begin{equation}
 \begin{split}
  |T_\mu(1-\wt\chi_{M,R})(y)-T_\mu(1-\wt\chi_{M,R})(y')|&\leq\int_{\big(\tfrac12MB_R\big)^c}|K(y,z)-K(y',z)|d\mu(z)\\&\lesssim\int_{\big(\tfrac12MB_R\big)^c}\frac{\ell(R)^\alpha}{|x_R-z|^{n+\alpha}}d\mu(z)\lesssim M^{-\alpha},   
 \end{split}
\end{equation}
where the last inequality follows by standard estimates using the growth of the measure $\mu$.
Therefore $$I_2=|m_Q(T_\mu(1-\wt\chi_{M,R}))-m_R(T_\mu(1-\wt\chi_{M,R}))|\lesssim M^{-\alpha}.$$

\vspace{2mm}
\noindent{\bf Estimate of $I_3$.}
 By Lemma \ref{lemcz} and standard arguments, 
\begin{equation}
 \begin{split}
   I_3&=\frac1{\mu(Q)}\left|\int_Q\left(\wh T_\mu(1-\wt\chi_{M,R})(x^*)-\wh T_\mu(1-\wt\chi_{M,R})(x)\right)d\mu(x)\right|\\&
    \leq\sup_{x\in Q}\big|\wh T_\mu(1-\wt\chi_{M,R})(x^*)-\wh T_\mu(1-\wt\chi_{M,R})(x)\big|\\&\leq\sup_{x\in Q}\int_{\big(\tfrac12MB_R\big)^c}|\wh K(x,y)-\wh K(x^*,y)|d\mu(y)\\&
    \lesssim\int_{\big(\tfrac12MB_R\big)^c}\frac{\Delta^{\alpha/2}\ell(R)^{\alpha/2}}{|x_Q-y|^{n+{\alpha/2}}}d\mu(y)
    \lesssim\frac{\Delta^{\alpha/2}}{M^{\alpha/2}}.
 \end{split}
\end{equation}

\vspace{2mm}
\noindent{\bf Estimate of $I_1$.}
This term is estimated by a freezing argument. Indeed, recalling that $\ve\ll\Delta$, we have $\wh A(x)=A(x)$ for all 
$x\in Q$, and thus
\begin{align}
|\nabla_1\E_A(x,y)-\nabla_1\E_{\wh A}(x,y)| & \leq |\nabla_1\E_A(x,y)-
\nabla_1\Theta(x,y;A(x))| \\
&\quad  + |\nabla_1\E_{\wh A}(x,y) - \nabla_1\Theta(x,y;\wh A(x))|
\lesssim \frac{1}{|x-y|^{n-\alpha/2}},
\end{align}
Integrating with respect to $y\in MB_R$ we derive
$$|T_\mu\wt\chi_{M,R}(x)-\wh T_{\mu}\wt\chi_{M,R}(x)|
\lesssim \int_{MB_R}\frac{d\mu(y)}{|x-y|^{n-\alpha/2}}\lesssim (M\ell(R))^{\alpha/2},$$
and so
$$I_1=|m_Q(T_{\mu}\wt\chi_{M,R})-m_Q(\wh T_\mu\wt\chi_{M,R})|\lesssim (M\ell(R))^{\alpha/2}
\leq \ve_3,$$
for $\ell(R)$ small enough (depending on $M$).

\vspace{2mm}
\noindent{\bf Estimate of $I_4$.}
We write 
\begin{equation}\label{third}
m_R(T_\mu\wt\chi_{M,R})=m_R(T_\mu(\wt\chi_{M,R}-\chi_R))+m_R(T_{\mu}\chi_R).
\end{equation}
Concerning the term $m_R(T_{\mu}\chi_R)$, the antisymmetric part of the kernel of $T_\mu$ does not contribute to the average, hence we can write
$$m_R(T_{\mu}\chi_R)=\frac1{2\mu(R)}\iint_{R\times R}\left(K(y,z)+K(z,y)\right)d\mu(y)d\mu(z).$$
Using now the estimate \rf{remantis} and the $n$-growth of the measure $\mu$,  for any $y\in R$  we get
 $$\int_R|K(y,z)+K(z,y)|d\mu(z)\lesssim\int_R\frac{d\mu(z)}{|y-z|^{n-\alpha}}\lesssim\ell(R)^{\alpha},
 $$ and so
 \begin{equation}\label{eqmrpet}
|m_R(T_{\mu}\chi_R)|\lesssim\ell(R)^\alpha.
\end{equation}

To conclude  with $I_4$ it remains to estimate $\left|m_R(T_\mu(\wt\chi_{M,R}-\chi_R))\right|$.
Given some small constant $\kappa\in (0,1/10)$ to be chosen below,
let $\wt\chi_{\kappa,R}$ be a $C^1$ function
which equals 1 
on $\mathcal{U}_{\kappa\ell(R)}(R)$, vanishes out of $\mathcal{U}_{\kappa2\ell(R)}(R)$ and satisfies $\|\nabla\wt\chi_{\kappa,R}\|_{\infty}\lesssim(\kappa\ell(R))^{-1}$, and denote $\varphi=\wt\chi_{M,R} - \wt\chi_{\kappa,R}$.
In particular we have
$$\wt\chi_{M,R} - \chi_R =  \varphi +  \wt\chi_{\kappa,R} -\chi_R.$$
Then we split as follows:
\begin{equation}\label{abc}
 \begin{split}
   \left|\int_R T_\mu(\wt\chi_{M,R}-\chi_R)d\mu\right|&\leq \left|\int_RT_\mu\varphi\;d\mu\right| +   
   \int_R|T_\mu( \wt\chi_{\kappa,R} -\chi_R)|\,d\mu=:A+B.
 \end{split}
\end{equation}
The Cauchy-Schwarz inequality and the thin boundary condition \eqref{thin_bdry} of $R$ give us the estimate
\begin{align}\label{eqBB}
B&\leq\|T_\mu( \wt\chi_{\kappa,R} -\chi_R)\|_{L^2(\mu{|_R})}\,\mu(R)^{1/2}\lesssim\|\wt\chi_{\kappa,R} -\chi_R\|_{L^2(\mu)}\,\mu(R)^{1/2}\\&\leq\mu\left(\mathcal{U}_{2\kappa\ell(R)}(R)\setminus R\right)^{1/2}\mu(R)^{1/2}\lesssim\kappa^{\gamma_0/2}\mu(R).
\end{align}
Now it remains to estimate the term $A$. We consider another auxiliary function $\tilde\varphi$ supported on $\mathcal{U}_{\kappa\ell(R)/4}(R)$ such that $\tilde\varphi\equiv 1$ on 
$\mathcal{U}_{\kappa\ell(R)/8}(R)$ 
and $\|\nabla\tilde\varphi\|_{\infty}\lesssim(\kappa\ell(R))^{-1}$. Write
\begin{equation}\label{termC}
A=\left|\int_R T_\mu\varphi\;d\mu\right|\leq\left|\int\tilde\varphi \,T_\mu\varphi\;d\mu\right|+\left|\int(\chi_R-\tilde\varphi) T_\mu\varphi\;d\mu\right|
\end{equation}
For the second term above, notice that the definition of $\tilde\varphi$ and the thin boundary condition imply that
$\left\|\chi_R-\tilde\varphi\right\|_{L^2(\mu)}\lesssim \kappa^{\gamma_0/2}\mu(R)^{1/2}.$
Therefore, 
\begin{equation}\label{eqAA5}
\begin{split}
\left|\int(\chi_R-\tilde\varphi) T_\mu\varphi\;d\mu\right| & \lesssim \|\varphi\|_{L^2(\mu)}
\left\|\chi_R-\tilde\varphi\right\|_{L^2(\mu)}\\
&\lesssim \mu\big(B(x_R',\tfrac34M\ell(R))\big)^{1/2}\,\kappa^{\gamma_0/2}\mu(R)^{1/2} \lesssim M^{1/2}\kappa^{\gamma_0/2}\mu(R).
\end{split}
\end{equation}

To treat the first term in \eqref{termC}, taking $c_R\geq 0$, split it as follows
\begin{equation}\label{eqa123}
 \begin{split}
   \left|\int\tilde\varphi \,T_\mu\varphi\;d\mu\right|&\leq \left|\int\tilde\varphi \,T_\mu\varphi\;d\left(\mu-c_R\HH^n|_{L_R}\right)\right|+ 
   \left|c_R\int\tilde\varphi\, T_{\mu-c_R\HH^n|_{L_ R}}\varphi\;d\HH^n|_{L_R}\right|\\&\quad+ \left|c_R\int\tilde\varphi \,T_{c_R\HH^n|_{L_R}}\varphi\;d\HH^n|_{L_R}\right|=:A_1+A_2+A_3.
 \end{split}
\end{equation}

To estimate  $A_1$ we would like to use the $\alpha$-numbers. However, we can only guarantee that $T_\mu\varphi$
is H\"older continuous on $\supp\tilde\varphi$. So we convolve this function with a non-negative, radial, $C^\infty$ function $\theta$ 
supported on $B(0,\hat\kappa\ell(R))$, and such that $\int \theta\,d\LL^{n+1}=0$ and  $\|\nabla\theta\|_\infty\lesssim
(\hat\kappa\ell(R))^{-n+2}$, with $\hat\kappa\in(0,\kappa/20)$ to be chosen.
Then we write
\begin{align}
A_1 &\leq \left|\int\tilde\varphi \,\bigl[\theta* T_\mu\varphi\bigr]\;d\left(\mu-c_R\HH^n|_{L_R}\right)\right| + 
\left|\int\tilde\varphi \,\bigl[T_\mu\varphi - \theta * T_\mu\varphi\bigr]\;d\left(\mu-c_R\HH^n|_{L_R}\right)\right|\\
& =: A_{1,1} + A_{1,2}.
\end{align}
We turn first our attention to $A_{1,1}$:
\begin{equation}\label{eqalp92}
A_{1,1}\leq\bigl\|\nabla\bigl(\tilde\varphi \,\bigl[\theta* T_\mu\varphi\bigr]\bigr)\bigr\|_{\infty}M^{n+1}\ell(R)^{n+1}\alpha_\mu^{L_R}(MB_R).
\end{equation}
Notice that
$$\bigl\|\nabla\bigl(\tilde\varphi \,\bigl[\theta* T_\mu\varphi\bigr]\bigr)\bigr\|_{\infty}\leq
\bigl\|\nabla\bigl(\theta* T_\mu\varphi\bigr)\bigr]\bigr\|_{\infty,\supp\tilde\varphi} 
+
\|\nabla\tilde\varphi\|_{\infty}
\|\theta* T_\mu\varphi\|_{\infty,\supp\tilde\varphi}.$$
Since $\dist(\supp\varphi,\supp\tilde\varphi)\geq\kappa\ell(R)/4$ and $\supp\theta\subset B(0,\kappa\ell(R)/20)$, we derive
$$\|\theta* T_\mu\varphi\|_{\infty,\supp\tilde\varphi} \lesssim \frac{\mu\big(B(x_R',M\ell(R))\big)}{(\kappa\,\ell(R))^n}\lesssim
\frac {M^n}{\kappa^n}$$
and 
$$\bigl\|\nabla\bigl(\theta* T_\mu\varphi\bigr)\bigr\|_{\infty,\supp\tilde\varphi}\lesssim \frac {M^n}{\kappa^n}\,\|\nabla\theta\|_1 \lesssim \frac {M^n}{\kappa^n\,\hat\kappa \,\ell(R)}.$$
Hence, using also that $\|\nabla\tilde\varphi\|_{\infty}\lesssim(\kappa\ell(R))^{-1}$,
$$\bigl\|\nabla\bigl(\tilde\varphi \,\bigl[\theta* T_\mu\varphi\bigr]\bigr)\bigr\|_{\infty}\lesssim 
\frac {M^n}{\kappa^n\,\hat\kappa\,\ell(R)} + \frac {M^n}{\kappa^{n+1}\ell(R)} \lesssim \frac {M^n}{\kappa^n\,\hat\kappa\,\ell(R)}.$$
Plugging this estimate into \rf{eqalp92}, we obtain
$$A_{1,1}\lesssim\varepsilon\frac{M^{2n+1}}{\hat\kappa\kappa^n}\mu(R).$$

Concerning the term $A_{1,2}$, we have
$$A_{1,2}\leq \int\tilde\varphi \,\bigl|T_\mu\varphi - \theta * T_\mu\varphi\bigr|\;d\bigl|\mu-c_R\HH^n|_{L_R}\bigr|
\lesssim \bigl\|T_\mu\varphi - \theta * T_\mu\varphi\bigr\|_{\infty,\supp\tilde\varphi}\,\ell(R)^n.
$$
For each $x\in\supp\tilde\varphi$, we write
\begin{align}
|T_\mu\varphi(x) - \theta * T_\mu\varphi(x)|&\leq \sup_{y\in B(x,\hat\kappa\ell(R))} |T_\mu\varphi(x) -T_\mu\varphi(y)| \\& \leq \sup_{y\in B(x,\hat\kappa\ell(R))} \int_{\supp\varphi}|K(x,z) - K(y,z)|\,d\mu(z).
\end{align}
Using the fact that $\dist(x,\supp\varphi)\geq\kappa\ell(R)$ and the H\"older continuity of $K$, for $x$ and $y$ as above we get
$$ \int_{\supp\varphi}|K(x,z) - K(y,z)|\,d\mu(z)\lesssim \int_{|x-z|\geq\kappa\ell(R)}\frac{(\hat\kappa\ell(R))^{\alpha}}{|x-z|^{n+\alpha}}\,d\mu(z) \lesssim \frac{\hat\kappa^{\alpha}}{\kappa^{\alpha}},$$
and thus
$$A_{1,2}\lesssim \frac{\hat\kappa^{\alpha}}{\kappa^{\alpha}}\,\ell(R)^n.$$
Together with the estimates for $A_{1,1}$, choosing $\hat\kappa=\kappa^2$, this gives
$$
A_1\leq A_{1,1}+A_{1,2}\lesssim \mu(R)\left(
\varepsilon\frac{M^{2n+1}}{\kappa^{n+2}}
+\kappa^{\alpha}\right).$$

To deal with $A_2$, we write
\begin{equation}
 \begin{split}
    A_2&=\left|c_R\int\tilde\varphi(x) T_{\mu-c_R\HH^n|_{L_ R}}\varphi(x)\;d\HH^n|_{L_R}(x)\right|
    \approx \left|\int T_{\HH^n|_{L_R}}^{*}\tilde\varphi(x) \,d(\mu-c_R\HH^n|_{L_ R})(x)\right|,
     \end{split}
\end{equation}
where $T^*$ denotes the transpose of the gradient of the single layer potential.
Arguing as for the term $A_1$, essentially reversing the roles of $\varphi$ and $\tilde\varphi$, 
we get
$$A_2\lesssim \mu(R)\left(
\varepsilon\frac{M^{n+1}}{\kappa^{n+2}}
+M^n\,\kappa^{\alpha}\right).$$
We leave the details for the reader.

Now we  will estimate the term $A_3$ in \rf{eqa123}. To this end, first we take into account that
$$\left|\int\tilde\varphi\, T_{\HH^n|_{L_R}}\tilde\varphi\;d\HH^n|_{L_R}\right| \lesssim\ell(R)^{n+\alpha}.$$
This follows by the same argument used to prove that $|m_R(T_{\mu}\chi_R)|\lesssim\ell(R)^\alpha$ in \rf{eqmrpet}.
Then we have
\begin{equation}
 \begin{split}
    A_3&\approx\left|\int\tilde\varphi\, T_{\HH^n|_{L_R}}\varphi\;d\HH^n|_{L_R}\right|\\&\lesssim \ell(R)^{n+\alpha}+
    \left|\int\tilde\varphi \,T_{\HH^n|_{L_R}}(\varphi+\tilde\varphi)\;d\HH^n|_{L_R}\right|\\
    &\leq
    \ell(R)^{n+\alpha}+
    \left|\int\tilde\varphi\, T_{\HH^n|_{L_R}}\wt\chi_{M,R}\;d\HH^n|_{L_R}\right|+
    \left|\int\tilde\varphi \,T_{\HH^n|_{L_R}}(\wt\chi_{M,R}-\varphi-\tilde\varphi)\;d\HH^n|_{L_R}\right|
  \\&=\ell(R)^{n+\alpha}+A_{3,1}+A_{3,2}.
 \end{split}
\end{equation}

The Cauchy-Schwarz inequality and the $L^2(\HH^n|_{L_R})$-boundedness of $T_{\HH^n|_{L_R}}$ imply
$$A_{3,2}\lesssim\|\wt\chi_{M,R}-\varphi-\tilde\varphi\|_{L^2(\HH^n|_{L_R})}\,\ell(R)^{n/2}
= \|\wt\chi_{\kappa,R}-\tilde\varphi\|_{L^2(\HH^n|_{L_R})}\,\ell(R)^{n/2}
.$$
To estimate $\|\wt\chi_{\kappa,R}-\tilde\varphi\|_{L^2(\HH^n|_{L_R})}$ we use the $\alpha$-numbers and
the thin boundary condition of $R$ with respect to $\mu$:
\begin{align}
\|\wt\chi_{\kappa,R}-\tilde\varphi\|_{L^2(\HH^n|_{L_R})}^2 & \leq
\int |\wt\chi_{\kappa,R}-\tilde\varphi|^2\,d\mu + \left|\int |\wt\chi_{\kappa,R}-\tilde\varphi|^2\,d(\mu 
-\HH^n|_{L_R})\right|\\
& \lesssim \mu\big(\mathcal{U}_{2\kappa\ell(R)}(R)\setminus R\big) + \alpha_\mu^{L_R}(MB_R)\,(M\ell(B_R))^{n+1}\,
\|\nabla(|\wt\chi_{\kappa,R}-\tilde\varphi|^2)\|_\infty\\
& \lesssim \kappa^{\gamma_0}\,\mu(R) + \ve M^{n+1}\,\kappa^{-1}\,\ell(R)^n.
\end{align}
where we took into account that $\|\nabla(|\wt\chi_{\kappa,R}-\tilde\varphi|^2)\|_\infty\lesssim (\kappa \ell(R))^{-1}$. Thus,
$$A_{3,2}\lesssim \left(\kappa^{\gamma_0/2} + \ve^{1/2} M^{(n+1)/2}\,\kappa^{-1/2}\right)\,\mu(R).$$

Next we deal with $A_{3,1}$. To this end, we write
\begin{equation}\label{eqa31**}
\begin{split}
A_{3,1} &\lesssim \sup_{x\in L_R\cap B(x_R',2\ell(R))} |T_{\HH^n|_{L_R}}\wt\chi_{M,R}(x)|\,\,\mu(R)\\
&\lesssim \sup_{x\in L_R\cap B(x_R',2\ell(R))} |T_{x,\HH^n|_{L_R}}\wt\chi_{M,R}(x)|\,\,\mu(R) + M^{n+\alpha}\ell(R)^\alpha\mu(R),
\end{split}
\end{equation}
where $T_x$ denotes the frozen operator. To simplify notation we denote by $K_x(\cdot)=\nabla_1\Theta(\cdot,0;A(x))$ its kernel.
For any $x\in L_R\cap B(x_R',2\ell(R))$, by the change of variable $z=2x-y$,
\begin{equation}\label{eqtxh22}
T_{x,\HH^n|_{L_R}}\wt\chi_{M,R}(x) = \int K_x(x-y)\,\wt\chi_{M,R}(y)\, d\HH^n|_{L_R}(y) = 
\int K_x(z-x)\,\wt\chi_{M,R}(2x-z)\, d\HH^n|_{L_R}(z).
\end{equation}
Hence,
\begin{equation}\label{eqtxh2222}
\begin{split}
2\,T_{x,\HH^n|_{L_R}}\wt\chi_{M,R}(x) &=  \int K_x(x-y)\,\wt\chi_{M,R}(y)\, d\HH^n|_{L_R}(y)\\&\qquad +
\int K_x(y-x)\,\wt\chi_{M,R}(2x-y)\, d\HH^n|_{L_R}(y)\\
& =
\int K_x(x-y)\,\bigl(\wt\chi_{M,R}(y)-\wt\chi_{M,R}(2x-y)\bigr)\, d\HH^n|_{L_R}(y).
\end{split}
\end{equation}
To estimate the last integral, recall that $\wt\chi_{M,R}$ is radial with respect to $x_R'$, and hence
$$\wt\chi_{M,R}(2x-y) = \wt\chi_{M,R}(2x_R'-(2x-y)) = \wt\chi_{M,R}( y + 2(x_R'-x)).$$
Thus, for all $x\in L_R\cap B(x_R',2\ell(R))$,
$$\supp\bigl(\wt\chi_{M,R}-\wt\chi_{M,R}(2x-\cdot)\bigr) \subset A\bigl(x_R',\tfrac12M\ell(R),2M\ell(R)\bigr).$$
 Also, for all $y\in L_R$, since $\wt\chi_{M,R}$ is Lipschitz with constant $c/(M\ell(R))$, 
$$
\bigl|\wt\chi_{M,R}(y)-\wt\chi_{M,R}(2x-y)\bigr|  = \bigl|\wt\chi_{M,R}(y)-\wt\chi_{M,R}(y + 2(x_R'-x))\bigr|
\lesssim \frac{|x_R'-x|}{M\,\ell(R)} \lesssim  \frac1M.
$$
So we get
\begin{equation}\label{eq1mm}
|T_{x,\HH^n|_{L_R}}\wt\chi_{M,R}(x)|\lesssim \frac1M\int_{A\big(x_R',\tfrac12M\ell(R),2M\ell(R)\big)} |K_x(x-y)|\,d\HH^n|_{L_R}(y)
\lesssim \frac1M.
\end{equation}
Together with  \rf{eqa31**}, this gives
$$A_{3,1} \lesssim \bigl(M^{n+\alpha}\ell(R)^\alpha + M^{-1}\bigr)\,\mu(R)
.$$
Now, gathering this estimate with the one of $A_{3,2}$, we get
$$A_3 \lesssim \left(\kappa^{\gamma_0/2} + \ve^{1/2} M^{(n+1)/2}\,\kappa^{-1/2} + M^{n+\alpha}\ell(R)^\alpha +  M^{-1}+\ell(R)^\alpha\right)\,\mu(R),$$
and then, by \rf{eqAA5},
\begin{align}
A& \lesssim 
M^{1/2}\kappa^{\gamma_0/2}\mu(R) + A_1+ A_2 + A_3 \\& \lesssim 
M^{1/2}\kappa^{\gamma_0/2} \mu(R)
+ \left(
\varepsilon\frac{M^{2n+1}}{\kappa^{n+2}} + \kappa^{\alpha}\right) \mu(R)
+
 \left(
\varepsilon\frac{M^{n+1}}{\kappa^{n+2}}
+M^n\,\kappa^{\alpha}\right) \mu(R)\\& \quad
+ \left(\kappa^{\gamma_0/2} + \ve^{1/2} M^{(n+1)/2}\,\kappa^{-1/2} + M^{n+\alpha}\ell(R)^\alpha +  M^{-1}+\ell(R)^\alpha\right)\,\mu(R).
\end{align}
Note that if $M$ is chosen big enough, then $\kappa$ and $\ell(R)$ small enough, and finally $\ve$ small enough (in this order), we get
$$A\leq \ve_3\,\mu(R).$$

We can now conclude the estimate of the term $I_4$ in \rf{8terms}.
From \rf{eqmrpet}, the last estimate, and \rf{eqBB}, we obtain
$$I_4\leq
|m_R(T_\mu(\wt\chi_{M,R}-\chi_R))|+|m_R(T_{\mu}\chi_R)| \lesssim\,\ell(R)^\alpha + \frac 1{\mu(R)}\,(A+B)
\lesssim
\ell(R)^\alpha + \kappa^{\gamma_0/2} + \ve_3\lesssim \ve_3,$$
assuming again $\ell(R)$ and $\kappa$ to be small enough.

\vspace{2mm}
\noindent{\bf Estimate of $I_5$ and $I_5^H$.}
Recall that
 $I_5=|m_Q(\wh T^{\phi}_{\mu}\wt\chi_{M,R})|$
 and that, for $x\in Q\in{\Ch}_{\Stop}(R)$, by definition we have $\wh T_{\mu}^{\phi}\wt\chi_{M,R}(x)=\wh T_\mu\wt\chi_{M,R}(x^*)$. We split it as follows
 \begin{equation}
\begin{split}
\label{num7}
\bigl|\wh T\left(\wt\chi_{M,R}\mu\right)(x^*)\bigr|&\leq \bigl|\wh T(\wt\chi_{M,R}\mu)(x^*)-\wh T(\wt\chi_{M,R}c_R\HH^n|_{L_R})(x^*)\bigr|
\\&\quad+\bigl|c_R\wh T(\wt\chi_{M,R}\HH^n|_{L_R})(x^*)\bigr|.
\end{split}
\end{equation}

We consider now the first term in the right hand side of inequality \eqref{num7}. Let $\wh T_{x^*}$ be the frozen operator associated with the kernel $\wh K_{x^*}(\cdot)\coloneqq\nabla_1\Theta(\cdot,0;\wh A(x^*))$. Notice that by Lemma \ref{holder/2} and Lemma \ref{lemm_freezing},
\begin{equation}\label{eqspl5298}
 \begin{split}
  \bigl|\wh T_\mu(\wt\chi_{M,R})(x^*)-\wh T_{c_R\HH^n|_{L_R}}(\wt\chi_{M,R})(x^*)\bigr|
  &\lesssim \bigl|\wh T_{x^*,c_R\HH^n|_{L_R}}(\wt\chi_{M,R})(x^*)-\wh T_{c_R\HH^n|_{L_R}}(\wt\chi_{M,R})(x^*)\bigr|
  \\&\quad+\bigl|\wh T_\mu(\wt\chi_{M,R})(x^*)-\wh T_{x^*,\mu}(\wt\chi_{M,R})(x^*)\bigr|
  \\&\quad+\bigl|\wh T_{x^*,\mu}(\wt\chi_{M,R})(x^*)-\wh T_{x^*,c_R\HH^n|_{L_R}}(\wt\chi_{M,R})(x^*)\bigr|\\&\lesssim 
  \biggl|\int \wt\chi_{M,R}(y)\, \wh K_{x^*}(x^*-y)d(\mu-c_R\HH^n|_{L_R})(y)\biggr|\\&\quad+M^{\alpha/2}\ell(R)^{\alpha/2}.
 \end{split}
\end{equation}
 To estimate the remaining term in the last inequality, we will use the $\alpha$-numbers. To this end we consider an auxiliary smooth function $\psi$ which equals $1$ on 
 $\R^{n+1}\setminus B(x^*,\Delta\ell(R)/2)$ and vanishes in $B(x^*,\Delta\ell(R)/4)$,
 with $\|\nabla \psi\|_\infty\leq 1/(\Delta\ell(R))$. Then taking into account that $\psi\equiv1$ on $MB_R\cap\supp\mu$,
 the remaining term in the inequality above equals
\begin{equation}
 \begin{split}
   \biggl|\int \wt\chi_{M,R}(y)\,\psi(y)\, \wh K_{x^*}&(x^*-y)d(\mu-c_R\HH^n|_{L_R})(y)\biggr|\\
   &\leq
 \alpha_\mu^{L_R}(MB_R) \,(M\ell(R))^{n+1}\, 
 \bigl\|\nabla\bigl(\wt\chi_{M,R}\,\psi\, \wh K_{x^*}(x^*-\cdot)\bigr)\bigr\|_\infty.
 \end{split}
\end{equation} 
It is easy to check that
$ \bigl\|\nabla\bigl(\wt\chi_{M,R}\,\psi\, \wh K_{x^*}(x^*-\cdot)\bigr)\bigr\|_\infty\lesssim C(M,\Delta)\,\ell(R)^{-n-1}.$
Thus, the integral on the right hand side of \rf{eqspl5298} does not exceed $C(M,\Delta)\,\ve$, and so
\begin{equation}\label{eqdjgh47}
\left|\wh T_\mu(\wt\chi_{M,R})(x^*)-\wh T_{c_R\HH^n|_{L_R}}(\wt\chi_{M,R})(x^*)\right|\lesssim
C(M,\Delta)\,\ve + (M\ell(R))^{\alpha/2}.
\end{equation}

 To estimate the second term on the right hand side of inequality \eqref{num7}, 
we denote by $w$ the orthogonal projection of $x$ on $L_R$ (recall that $x\in Q$), by $w^*$ the reflection of $w$ with respect to $H$, and we split
\begin{equation}\label{eqsp934**}
\begin{split}
\bigl|\wh T(\wt\chi_{M,R}\HH^n|_{L_R})(x^*)\bigr|&\leq \bigl|\wh T(\wt\chi_{M,R}\HH^n|_{L_R})(x^*)-\wh T_{x^*}(\wt\chi_{M,R}\HH^n|_{L_R})(x^*)\bigr|
\\&\quad+\bigl|\wh T_{x^*}(\wt\chi_{M,R}\HH^n|_{L_R})(x^*) - \wh T_{x_R}(\wt\chi_{M,R}\HH^n|_{L_R})(x^*)
\bigr|\\
& \quad+ \bigl|\wh T_{x_R}(\wt\chi_{M,R}\HH^n|_{L_R})(x^*) - \wh T_{x_R}(\wt\chi_{M,R}\HH^n|_{L_R})(w^*)\bigr|\\
& \quad + \bigl|\wh T_{x_R}(\wt\chi_{M,R}\HH^n|_{L_R})(w^*)\bigr|,
\end{split}
\end{equation}
where $\wh T_{x_R}$ is the frozen operator associated with the kernel $\wh K_{x_R}(\cdot)\coloneqq\nabla_1\Theta(\cdot,0;\wh A(x_R))$.
Using Lemma \ref{lemm_freezing}, it is easy to check that the first term on the right hand side does not exceed
$C\,(M\ell(R))^{\alpha/2}$. For the third term, since
$$\dist(x^*,\supp\mu\cap 2MB_R)\approx\Delta\ell(R)\ll |x^*-w^*| = |x-w|\lesssim \ve^{1/(n+1)}\ell(R),$$
by standard arguments we derive
\begin{align}
\bigl|\wh T_{x_R}(\wt\chi_{M,R}\HH^n|_{L_R})(x^*) - \wh T_{x_R}(\wt\chi_{M,R}\HH^n|_{L_R})(w^*)\bigr|&
\lesssim \frac{|x^*-w^*|}{(\Delta\ell(R))^{n+1}}\,\HH^n|_{L_R}\big(B(x_R',M\ell(R))\big)\\ &\lesssim \ve^{1/(n+1)}\Delta^{-n}M^n.
\end{align}

Next we estimate the second term on the right hand side of \rf{eqsp934**}.
We have
$$\bigl|\wh T_{x^*}(\wt\chi_{M,R}\HH^n|_{L_R})(x^*) - \wh T_{x_R}(\wt\chi_{M,R}\HH^n|_{L_R})(x^*)
\bigr| \leq \int_{2MB_R} \!\!\!|\wh K_{x^*}(x^*-y) - \wh K_{x_R}(x^*-y)|\,
d\HH^n|_{L_R}(y).$$
By \rf{grad_const_coeff_sol}, we have
\begin{equation}\label{eqkern56}
\wh K_{x^*}(z) - \wh K_{x_R}(z) =
\frac{\omega_n^{-1}}{\sqrt{\det \wh A(x^*)}}\frac{\wh A(x^*)^{-1}z}{(\wh A(x^*)^{-1}z\cdot z)^{(n+1)/2}} 
- \frac{\omega_n^{-1}}{\sqrt{\det \wh A(x_R)}}\frac{\wh A(x_R)^{-1}z}{(\wh A(x_R)^{-1}z\cdot z)^{(n+1)/2}} .
\end{equation}
By standard estimates and the H\"older continuity of $\wh A$ it follows that, for any $z\in\R^{n+1}$,
\begin{equation}\label{estim_diff_kern_hold}
|\wh K_{x^*}(z) - \wh K_{x_R}(z)\bigr| \lesssim \frac{|x^*- x_R|^{\alpha/2}}{|z|^n}\lesssim
\frac{\ell(R)^{\alpha/2}}{|z|^n}.
\end{equation}
Since, for any $x\in R$, $\dist(x^*,L_R)\approx \Delta\,\ell(R)$, we deduce
\begin{align}
\int_{2MB_R} \bigl|\wh K_{x^*}(x^*-y) - \wh K_{x_R}(x^*-y)\bigr|\,
d\HH^n|_{L_R}(y)& \lesssim  \frac{\ell(R)^{\alpha/2}}{(\Delta\ell(R))^n}\,\HH^n(2MB_R\cap L_R) \\
&
\approx M^n\Delta^{-n}\ell(R)^{\alpha/2}.
\end{align}
Therefore, plugging all these estimates in \eqref{eqsp934**}, we get
\begin{equation}
\begin{split}
|\wh T(\wt\chi_{M,R}\HH^n|_{L_R})(x^*)\bigr|\lesssim 
M^{\alpha/2}\ell(R)^{\alpha/2} &+ \ve^{1/(n+1)}\Delta^{-n}M^n \\
&+
M^n\Delta^{-n}\ell(R)^{\alpha/2} + \bigl|\wh T_{x_R}(\wt\chi_{M,R}\HH^n|_{L_R})(w^*)\bigr|.
\end{split}
\end{equation}

To deal with the last term on the right hand side of \rf{eqsp934**}, we distinguish between the vertical and the horizontal components, so we set
$$\bigl|\wh T_{x_R}(\wt\chi_{M,R}\HH^n|_{L_R})(w^*)\bigr| \leq \bigl|\wh T_{x_R}^V(\wt\chi_{M,R}\HH^n|_{L_R})(w^*)\bigr|
+ \bigl|\wh T_{x_R}^H(\wt\chi_{M,R}\HH^n|_{L_R})(w^*)\bigr|.$$
Being $\wh A(x_R)=Id$, $\wh T_{x_R}$ coincides with the Riesz transform modulo some constant factor.
Hence its vertical component coincides with the Poisson transform modulo some constant factor, so that
\begin{equation}\label{eqvertical5}
\bigl|\wh T_{x_R}^V(\wt\chi_{M,R}\HH^n|_{L_R})(w^*)\bigr|\lesssim 1.
\end{equation}
The horizontal component
is estimated like the term $T_{x}^H(\wt\chi_{M,R}\HH^n|_{L_R})(x)$ in \rf{eqtxh22}. The reader can check that the same estimates hold just replacing
$x$ by either $w^*$ or $x_R$ appropriately, and $T_x$ by $\wh T_{x_R}$. A key point is that, for the kernel $\wh K^H_{x_R}$ of $\wh T^H_{x_R}$, the change of variable $z=2w-y$ gives us
\begin{align}
\wh T_{x_R,\HH^n|_{L_R}}^H\wt\chi_{M,R}(w^*) &= \int \wh K_{x_R}^H(w^*-y)\,\wt\chi_{M,R}(y)\, d\HH^n|_{L_R}(y)\\& = 
\int \wh K_{x_R}^H(z-(2w-x^*))\,\wt\chi_{M,R}(2w-z)\, d\HH^n|_{L_R}(z) \\ &= 
\int \wh K_{x_R}^H(z-w^*)\,\wt\chi_{M,R}(2w-z)\, d\HH^n|_{L_R}(z),
\end{align}
which is analogous to \rf{eqtxh22}. Notice that the last identity is only valid for the horizontal component of the kernel $\wh K_{x_R}$ (taking into account that $\wh K_{x_R}$ is the kernel of the Riesz transform modulo some constant factor, since $\wh A(x_R)=Id$). Then, as in \rf{eqtxh2222}, we can write
$$2\,\wh T_{x_R,\HH^n|_{L_R}}^H\wt\chi_{M,R}(w^*) = \int \wh K_{x_R}^H(w^*-y)\,\bigl(\wt\chi_{M,R}(y)-\wt\chi_{M,R}(2w-y)\bigr)\, d\HH^n|_{L_R}(y).
$$
Thus, as in \rf{eq1mm}, we get
$$ \bigl|\wh T_{x_R}^H(\wt\chi_{M,R}\HH^n|_{L_R})(w^*)\bigr| \lesssim \frac1M.$$
Together with \rf{eqdjgh47}, this yields
\begin{align}
&I_5  \leq \sup_{x\in R}
\bigl|\wh T_{\mu}\wt\chi_{M,R}(x^*)\bigr| \\ 
&\lesssim C(M,\Delta)\,\ve + M^{\alpha/2}\ell(R)^{\alpha/2}
+ M^{\alpha/2}\ell(R)^{\alpha/2} + \ve^{1/(n+1)}\Delta^{-n}M^n + 1+\frac1M +
M^n\Delta^{-n}\ell(R)^{\alpha/2}.
\end{align}
For $I_5^H$ we get almost the same estimate. The only difference is that we do not have to estimate the vertical term 
in \rf{eqvertical5}, and thus the summand $1$ does not appear in the last inequality. So we have
$$I_5^H \lesssim C(M,\Delta)\,\ve + M^{\alpha/2}\ell(R)^{\alpha/2}
+ M^{\alpha/2}\ell(R)^{\alpha/2} + \ve^{1/(n+1)}\Delta^{-n}M^n + \frac1M +
M^n\Delta^{-n}\ell(R)^{\alpha/2}.
$$
Thus, for $M$ big enough, $\ell(R)$ small enough and $\ve$ small enough, we get
$$I_5\lesssim 1\quad \mbox{ and }\quad I_5^H\lesssim \ve_3.$$

\vspace{2mm}
Recall that we showed that $I_i^H\leq I_i\lesssim\ve_3$ for $i=1,\ldots,4$, by choosing  the parameters $M$ and $\kappa$ properly and assuming $\ve$ and $\ell(R)$ small enough.
Then, gathering the estimates obtained for $I_1$, \ldots, $I_5$ and $I_5^H,$ the lemma follows. 

\vspace{2mm}
 
\subsection{Proof of Lemma \ref{lemftr1.5}}

Recall that
$$f- \tilde f =  \sum_{Q\in{\Ch}_{\Stop}(R)} m_{\mu,Q}(S(\mu-\tmu))\,\chi_Q$$
and
$$f^H- \tilde f^H =  \sum_{Q\in{\Ch}_{\Stop}(R)} m_{\mu,Q}(S^H(\mu-\tmu))\,\chi_Q.$$
So we have
$$\|f^H- \tilde f^H\|_{L^2(\mu)} \leq \|f- \tilde f\|_{L^2(\mu)} \leq \|S(\mu-\tmu)\|_{L^2(\mu|_R)}
\leq \|S(\mu|_R-\tmu)\|_{L^2(\mu|_R)} + \|S(\mu|_{R^c})\|_{L^2(\mu|_R)}.$$
To estimate  the first term on the right hand side we use the $L^2(\mu|_R)$ boundedness of $S_\mu$ and
\rf{eqdif79}:
$$\|S(\mu|_R-\tmu)\|_{L^2(\mu|_R)}\lesssim 
\|\tmu-\mu|_R\|\lesssim t^{\gamma_0}\mu(R).$$
To deal with the second term we split $R^c$ in two regions:
$$D_1 = \UU_{\Delta^{1/2} \ell(R)}(R)\setminus R,\qquad D_2= \R^{n+1}\setminus\UU_{\Delta^{1/2} \ell(R)}(R).$$
Then we have
$$\|S(\mu|_{R^c})\|_{L^2(\mu|_R)}\leq 
\|S(\chi_{D_1}\mu)\|_{L^2(\mu|_R)} + \|S(\chi_{D_2}\mu)\|_{L^2(\mu|_R)}.
$$
By the $L^2(\mu|_R)$ boundedness of $S_\mu$ and the thin boundary property, we have
$$\|S(\chi_{D_1}\mu)\|_{L^2(\mu|_R)}^2 \lesssim \mu(D_1) \lesssim \Delta^{\gamma_0/2}\,\mu(R).$$

To estimate $\|S(\chi_{D_2}\mu)\|_{L^2(\mu|_R)}$, recall that 
$$S(\chi_{D_2}\mu)(x) = \int_{D_2} \bigl(\wh K(x,y)- \wh K(x^*,y)\bigr)\,d\mu(y)$$
For $x\in R$, $y\in D_2$, we have
$$|x-x^*|\leq 2\Delta\,\ell(R)\ll \frac12\,\Delta^{1/2}\ell(R)\leq \frac12\,|x-y|,$$
and thus
$$\bigl|\wh K(x,y)- \wh K(x^*,y)\bigr|\lesssim \frac{\Delta^{\alpha/2}\ell(R)^{\alpha/2}}{|x-y|^{n+{\alpha/2}}}.$$
Therefore, by standard estimates using the $n$-growth of $\mu$, for $x\in R$,
$$\big|S(\chi_{D_2}\mu)(x)\big| \lesssim \int_{|x-y|>\frac12\Delta^{1/2}\ell(R)} 
\frac{\Delta^{\alpha/2}\ell(R)^{\alpha/2}}{|x-y|^{n+{\alpha/2}}}\,d\mu(y)\lesssim \frac{\Delta^{\alpha/2}\ell(R)^{\alpha/2}}{
(\Delta^{1/2}\ell(R))^{\alpha/2}} \lesssim \Delta^{{\alpha/4}}.
$$
Hence,
$$\|S(\chi_{D_2}\mu)\|_{L^2(\mu|_R)}^2\lesssim \Delta^{{\alpha/2}}\mu(R).$$
Together with the previous estimates, this yields
$$\|f^H- \tilde f^H\|_{L^2(\mu)} \leq \|f- \tilde f\|_{L^2(\mu)} \lesssim \big(t^{\gamma_0}+\Delta^{\min({\alpha/2},\gamma_0/2)}\big)\,\mu(R),$$
which proves the lemma.

\vspace{2mm}

 
\subsection{Proof of Lemma \ref{lemftr2}}
We will just prove \rf{eqsh454}. The arguments for the other inequalities \rf{eqsh452} and \rf{eqsh453} are totally analogous. Indeed, the reader can easily check that the operators $\wh T$, $S$, and $S^H$ are essentially interchangeable in the estimates below.

Recall that the measure $\sigma$ was defined in Section \ref{sec7}, and that $\tilde\ve$ is such that
$\beta_{\infty,\mu}(MB_Q)\leq\tve$ for all $Q\in{\Ch}_{\Stop}(R)$, $R\in\Nice$.

Let $\tau$ be a small number to be chosen below, with $\ve\ll\tau\ll\min(t,\Delta)\ll1$. For a fixed $Q\in{\Ch}_{\Stop}(Q)$ and $x\in\R^{n+1}$, let $\tchi_1(x)$ be a smooth radial function such that $\supp\tchi_1\subset B(0,\tau\ell(Q))$ and $\tchi_1(x)\equiv 1$ 
in $B(0,\frac12\tau\ell(Q))$. Let also $\tchi_2$ be a smooth radial function supported on the annulus $A\bigl(0, \frac12\tau\ell(Q),\frac12M\ell(Q)\bigr)$ and such that $\tchi_2\equiv 1$ in 
$A\bigl(0,\tau\ell(Q),\frac14M\ell(Q)\bigr)$. Finally, set $\tchi_3$ a smooth radial function supported on $B\bigl(0,\frac14M\ell(Q)\bigr)^c$, such that $\tchi_3\equiv 1$ in $B\bigl(0,\frac12M\ell(Q)\bigr)^c$.
We construct the functions $\tchi_i$ so that $\tchi_1+\tchi_2+\tchi_3=1$. Notice that they depend on the cube $Q$. Now denote $K_{S_{i,Q}^H}(x,y)=K_{S^H}(x,y)\tchi_i(|x-y|)$ for $i=1,2,3$, so that $K_{S^H}(x,y)=\sum_{i=1}^3K_{S_{i,Q}^H}(x,y)$, and denoting by $S^H_{i,Q}$ the operator associated with the truncated kernel $K_{S^H_{i,Q}}$, we also have $S^H=\sum_{i=1}^3S^H_{i,Q}.$ Further, we can write
$$S^H \sigma= \sum_{i=1}^3\sum_{Q\in {\Ch}_{\Stop}(R)} \chi_{\supp\sigma_Q}\cdot S_{i,Q}^H\sigma =: 
\sum_{i=1}^3 S_{i}^H\sigma\quad \mbox{ in $L^2(\sigma)$}$$
and 
$$S^H \mu= \sum_{i=1}^3\sum_{Q\in {\Ch}_{\Stop}(R)} \chi_{Q}\cdot S_{i,Q}^H\mu =: 
\sum_{i=1}^3 S_{i}^H\mu\quad \mbox{ in $L^2(\mu|_R)$}$$
To shorten the notation, we will write $K_{S_{i}^H}(x,y)$ instead of $K_{S_{i,Q}^H}(x,y)$ when $Q$ is clear from the context.
We split
\begin{equation}\label{eqspl9341}
\|S^H\sigma\|_{L^2(\sigma)}\leq \|S_1^H\sigma\|_{L^2(\sigma)} + \|S_2^H\sigma\|_{L^2(\sigma)}+ \|S_3^H\sigma\|_{L^2(\sigma)}.\end{equation}

\vspace{2mm}
\noindent {\bf Estimate of $ \|S_3^H\sigma\|_{L^2(\sigma)}$.}
 For $Q\in{{\Ch}_{\Stop}(R)}$ and $x,x'\in \UU_{10\tve\ell(Q)}(Q_{(t)})$,  we have
\begin{equation}\label{difS3}
 |S^H_{3,Q}\sigma(x)-S^H_{3,Q}\tmu(x')|\leq  |S^H_{3,Q}\sigma(x)-S^H_{3,Q}\tmu(x)|+ |S^H_{3,Q}\tmu(x)-S^H_{3,Q}\tmu(x')|=S_{31}+S_{32}.
\end{equation}
Notice that for $x\in\UU_{10\tve\ell(Q)}(Q_{(t)})$ and $y,\;y'\in\UU_{10\tve\ell(P)}(P_{(t)})$, $P\in{\Ch}_{\Stop}(R)$, by Lemma \ref{lemcz},
\begin{equation}\label{lip_alpha}
\bigl|K_{S_{3,Q}^H}(x,y)-K_{S_{3,Q}^H}(x,y')\bigr|\lesssim\frac{|y-y'|^{\alpha/2}}{C(t)(\ell(Q)+\ell(P)+\dist(P,Q))^{n+{\alpha/2}}}\lesssim\frac{\ell(P)^{\alpha/2}}{C(t)D(P,Q)^{n+{\alpha/2}}},
\end{equation}
where $D(P,Q)=\ell(Q)+\ell(P)+\dist(P,Q)$ and 
the $t$-dependence of $C(t)$ comes from the comparability $|x-y|\approx|x-y'|$, which depends on $t$ (due to $\tve\ll t$ being very small).
Applying now Lemma \ref{dif}, with the $\Lip_\alpha$-constant coming from \eqref{lip_alpha},
\begin{equation}\label{difS3}
\begin{split}
S_{31}&=|S^H_{3,Q}\sigma(x)-S^H_{3,Q}\tmu(x)|\leq\sum_{\tiny{P\in{\Ch}_{\Stop}(R)}}\left|\int K_{S^H_{3,Q}}(x,y)d(\sigma_P-\tmu_P)(y)\right|\\&\lesssim M^{\alpha/2}\tve^{\alpha/2} C(t)\underset{\tiny{P\in{\Ch}_{\Stop}(R)}}{\sum}\frac{\ell(P)^{\alpha/2}\mu(P)}{D(P,Q)^{n+{\alpha/2}}}.
\end{split}
\end{equation}

Concerning $S_{32}$, by standard estimates, one gets 
\begin{equation}\label{estim_32}S_{32}\leq\int| K_{S^H_{3,Q}}(x,y)-K_{S^H_{3,Q}}(x',y)|d\tmu(y)\lesssim\int_{|x-y|\geq \frac18M\ell(Q)}\frac{|x-x'|^{{\alpha/2}}}{|x_Q-y|^{n+{\alpha/2}}}d\tmu(y)\lesssim \frac1{M^{\alpha/2}}.\end{equation}
As a consequence of \eqref{difS3} and \eqref{estim_32},
$$|S^H_{3,Q}\sigma(x)-m_{\mu,Q}(S^H_{3,Q}\tmu)|\leq M^{\alpha/2}\tve^{\alpha/2} C(t)\underset{\tiny{P\in{\Ch}_{\Stop}(R)}}{\sum}\frac{\ell(P)^{\alpha/2}\mu(P)}{D(P,Q)^{n+{\alpha/2}}}+\frac{C}{M^{\alpha/2}}.$$
This  implies that for $x\in \supp\sigma_Q$,
$Q\in{\Ch}_{\Stop}(R)$,
\begin{equation}\label{3term}
|S^H_3\sigma(x)|\lesssim|m_{\mu,Q}(S^H_3\tmu)|+M^{\alpha/2}\tve^{\alpha/2} C(t)\underset{\tiny{P\in{\Ch}_{\Stop}(R)}}{\sum}\frac{\ell(P)^{\alpha/2}\mu(P)}{D(P,Q)^{n+{\alpha/2}}}+\frac{C}{M^{\alpha/2}}.
\end{equation}

Denote 
\begin{equation}\label{g}
g(x)=\sum_{\tiny{Q\in{\Ch}_{\Stop}(R)}}\sum_{\tiny{P\in{\Ch}_{\Stop}(R)}}\frac{\ell(P)^{\alpha/2}\mu(P)}{D(P,Q)^{n+{\alpha/2}}}\chi_Q(x).
\end{equation}
Since $\mu(Q)\approx\sigma(Q)$ for each $Q$, squaring and integrating \eqref{3term} with respect to $\sigma$, we obtain
\begin{equation}\label{norma2S3}
\begin{split}
\|S^H_3\sigma\|^2_{L^2(\sigma)}&\lesssim\underset{\tiny{Q\in{\Ch}_{\Stop}(R)}}{\sum}m_{\mu,Q}(S^H_3\tmu)^2\mu(Q)+M^{\alpha}\tve^{\alpha} C(t)\|g\|^2_{L^2(\mu)}+M^{-\alpha}\underset{\tiny{Q\in{\Ch}_{\Stop}(R)}}{\sum}\mu(Q)\\&
\approx\biggl\|\underset{\tiny{Q\in{\Ch}_{\Stop}(R)}}{\sum}m_{\mu,Q}(S^H_3\tmu)\chi_Q\biggr\|^2_{L^2(\mu|_R)}+M^{\alpha}\tve^{\alpha} C(t)\|g\|^2_{L^2(\mu)}+M^{-\alpha}\mu(R).
\end{split}
\end{equation}

We will estimate $\|g\|_{L^2(\mu)}$ by duality: for any non-negative funtion $h\in L^2(\mu)$ write
\begin{equation}\label{dual}
\begin{split}
\int g h\,d\mu&=\sum_{\tiny{Q\in{\Ch}_{\Stop}(R)}}\sum_{\tiny{P\in{\Ch}_{\Stop}(R)}}\frac{\ell(P)^{\alpha/2}\mu(P)}{D(P,Q)^{n+{\alpha/2}}}\int_Q hd\mu\\&
=\sum_{\tiny{P\in{\Ch}_{\Stop}(R)}}\mu(P)\sum_{\tiny{Q\in{\Ch}_{\Stop}(R)}}\frac{\ell(P)^{\alpha/2}}{D(P,Q)^{n+{\alpha/2}}}\int_Q hd\mu.
\end{split}
\end{equation}
Notice that for each $z\in P\in{\Ch}_{\Stop}(R)$, integrating on annuli we get
\begin{equation}\label{dualpart}
\begin{split}
\sum_{\tiny{Q\in{\Ch}_{\Stop}(R)}}\frac{\ell(P)^{\alpha/2}}{D(P,Q)^{n+{\alpha/2}}}\int_Q hd\mu&
\lesssim\int_Q\frac{\ell(P)^{\alpha/2} h(y)}{\left(\ell(P)+|z-y|\right)^{n+{\alpha/2}}}d\mu(y)\\&=\int_{|z-y|\leq\ell(P)}\frac{\ell(P)^{\alpha/2} h(y)}{\left(\ell(P)+|z-y|\right)^{n+{\alpha/2}}}d\mu(y)
\\&\quad+\sum_{i=1}^\infty\int_{2^{i-1}\ell(P)\leq|z-y|\leq 2^i\ell(P)}\frac{\ell(P)^{\alpha/2} h(y)}{(\ell(P)+|z-y|)^{n+{\alpha/2}}}d\mu(y)\\&\lesssim\sum_{i=0}^\infty \frac{2^{-i{\alpha/2}}\mu(B(z,2^i\ell(P)))}{\left(2^i\ell(P)\right)^n}m_{\mu,B(z,2^i\ell(P))}(h).
\end{split}
\end{equation}
Now let $M_\mu$ stand for the centered maximal Hardy-Littlewood operator with respect to $\mu$. Since $m_{\mu,B(z,2^i\ell(P))}(h)\lesssim M_\mu h(z)$ and 
$$\sum_{i=0}^\infty \frac{2^{-i\alpha/2}\mu(B(z,2^i\ell(P)))}{\left(2^i\ell(P)\right)^n}\leq C,$$ by \eqref{dual} and \eqref{dualpart},
$$\int g h\,d\mu\lesssim\sum_{\tiny{P\in{\Ch}_{\Stop}(R)}}\inf_{z\in P}M_\mu h(z)\,\mu(P)\leq\int M_\mu h\,d\mu\lesssim\|h\|_{L^2(\mu)}\mu(R)^{1/2}.$$
Therefore, $$\|g\|_{L^2(\mu)}\lesssim\mu(R)^{1/2}.$$
Plugging this into \eqref{norma2S3} we get
\begin{equation}\label{L2S3}
\begin{split}
\big\|S^H_3\sigma\big\|^2_{L^2(\sigma)} &\lesssim\biggl\|\underset{\tiny{Q\in{\Ch}_{\Stop}(R)}}{\sum}m_{\mu,Q}(S^H_3\tmu)\chi_Q\biggr\|^2_{L^2(\mu|_R)}+\left(M^{-\alpha}+M^{\alpha}\tve^{\alpha}C(t)\right)\mu(R)\\
& \lesssim\|\tilde f^H\|_{L^2(\mu)}^2 + \biggl\|\underset{\tiny{Q\in{\Ch}_{\Stop}(R)}}{\sum}m_{\mu,Q}(S^H_1\tmu)\chi_Q\biggr\|^2_{L^2(\mu|_R)}\\
&\quad+
\biggl\|\underset{\tiny{Q\in{\Ch}_{\Stop}(R)}}{\sum}m_{\mu,Q}(S^H_2\tmu)\chi_Q\biggr\|^2_{L^2(\mu|_R)}+\left(M^{-\alpha}+M^{\alpha}\tve^{\alpha}C(t)\right)\mu(R).
\end{split}
\end{equation}

\vspace{2mm}
\noindent {\bf Estimate of $ \|S_1^H\sigma\|_{L^2(\sigma)}$.}
Recall that, by \rf{eqsigma75}, for each $Q\in{\Ch}_{\Stop}(R)$,  
$$\supp\sigma_Q \subset \mathcal U_{3\tve\ell(Q)}(\supp \Pi_{L_Q\sharp}\mu|_{Q_{(t)}})
\subset \mathcal U_{6\tve\ell(Q)}(Q_{(t)}),$$
and, for $P,Q\in{\Ch}_{\Stop}(R)$ with $P\neq Q$, by \rf{eqdistsigma},
$$
\dist(\supp\sigma_P,\supp\sigma_Q) \geq \frac t2\,\max(\ell(P),\ell(Q)).
$$
Therefore, recalling that $\tau\ll t$,
\begin{equation}\label{eqsupp89}
S_1^H\sigma(x)=S_1^H\sigma_Q(x) \quad \mbox{ for all $x\in\supp\sigma_Q$}.
\end{equation}

 Let $J_Q$ be the convex hull of ${\mathcal U}_{10\tve\ell(Q)}(Q_{(3\tau)})\cap L_Q$. 
Then the following hold:\newline

\begin{enumerate}
 
 
 \item  By the thin boundary condition, we have $$\sigma_Q\big((J_Q)^c\big)\leq\mu(Q\setminus{\mathcal U}_{20\tve\ell(Q)}(Q_{(3\tau)}))\leq\mu(Q\setminus Q_{(4\tau)})\lesssim \tau^{\gamma_0}\mu(Q).$$ 

\vspace{2mm}
 \item Let $\psi_{3B_Q}$ be a smooth function that equals 1 in $2B_Q$, vanishes in $(3B_Q)^c$ and 
 such that $\|\nabla\psi_{3B_Q}\|_\infty\lesssim \ell(Q)^{-1}$. Then, for each $x\in L_Q\cap\UU_{\tau\ell(Q)}(J_Q)$, 
 $$ \chi_{B(x,3\tve\ell(Q))}\,\Pi_{L_Q\sharp}\mu_{|Q_{(t)}}=  \chi_{B(x,3\tve\ell(Q))}\, \Pi_{L_Q\sharp}(\psi_{3B_Q}\mu).$$
\end{enumerate}
\vspace{2mm}

Notice now that 
$$\|S_1^H\sigma\|^2_{L^2(\sigma)}=\sum_{\tiny{Q\in{\Ch}_{\Stop}(R)}}\|S_1^H\sigma\|^2_{L^2(\sigma_Q)},$$ and for each $Q\in{\Ch}_{\Stop}(R)$, 
$$\|S_1^H\sigma\|^2_{L^2(\sigma_Q)}=\|S_1^H\sigma\|^2_{L^2(\sigma_Q|_{J_Q})}+\|S_1^H\sigma\|^2_{L^2(\sigma_Q|_{(J_Q)^c})}= S_{11}+S_{12}.$$

Write $S_1^H\sigma(x)=\wh T_1^H\sigma(x)-\wh T_1^H\sigma(x^*)$. Since $\tau\ell(Q)\ll\Delta\ell(R)$, we have $\wh T_1^H\sigma(x^*)=0$. Therefore, 
by \rf{eqsupp89}, the Cauchy-Schwarz inequality, the property (1) of $J_Q$, and the $n$-growth of the measure $\sigma$,
\begin{equation}
 \begin{split}
S_{12}&=\big\|\wh T_1^H\sigma\big\|_{L^2(\sigma_Q|_{(J_Q)^c})}^2
=\big\|\wh T_1^H\sigma_Q\big\|_{L^2(\sigma_Q|_{(J_Q)^c})}^2
\\&\leq\big\|\wh T_1^H\sigma_Q\big\|^2_{L^4(\HH^n|_{L_Q})}\,\sigma_Q\big((J_Q)^c\big)^{1/2}
\lesssim\|\wh T_1\sigma_Q\|^2_{L^4(\HH^n|_{L_Q})}\,\tau^{\gamma_0/2}\mu(Q)^{1/2}.
\end{split}
\end{equation}
Recall that $\sigma_Q = g_Q\,\HH^n|_{L_Q}$ for some function $g_Q$ such that $0\leq g_Q\lesssim \chi_{2B_Q\cap L_Q}$. Since $\wh T_{1,\HH^n|_{L_Q}}^H$ is bounded in $L^4(\HH^n|_{L_Q})$ (by the uniform 
rectifiability of $L_Q$, for example), we have $\|\wh T_1\sigma_Q\|_{L^4(\HH^n|_{L_Q})}\lesssim \ell(Q)^{n/4}$, and thus
$$S_{12}\lesssim \tau^{\gamma_0/2}\mu(Q).$$

We treat now $S_{11}$. To this end, notice that for $x\in J_Q$, by \rf{eqsupp89}, a freezing argument and the antisymmetry of the kernel $\nabla_1\Theta(\cdot,\cdot;\wh A(x))$, we have
\begin{equation}\label{s1h}
  \begin{split}
     |S^H_1\sigma(x)|&=|S^H_1\sigma_Q(x)|\leq\left|\int \nabla_1\Theta(x,y;\wh A(x))d\sigma_Q(y)\right|+C\ell(Q)^{\alpha}\\&
     =\biggl|\int_{|x-y|\leq \tau\ell(Q)} \nabla_1\Theta(x,y;\wh A(x))(g_Q(y)-g_Q(x))d\HH^n|_{L_Q}(y)\biggr|+C\ell(Q)^{\alpha}\\&
     \lesssim\int_{|x-y|\leq \tau\ell(Q)}\frac{\Lip\Bigl(g_{Q|_{L_Q\cap\UU_{\tau\ell(Q)}(J_Q)}}\Bigr)}{|x-y|^{n-1}}\,d\HH^n|_{L_Q}(y)+\ell(Q)^{\alpha}\\&
     \lesssim \tau\ell(Q)\Lip\Bigl(g_{Q|_{L_Q\cap\UU_{\tau\ell(Q)}(J_Q)}}\Bigr)+\ell(Q)^{\alpha}.
  \end{split}
\end{equation}
To estimate $\Lip\Bigl(g_{Q|_{L_Q\cap\UU_{\tau\ell(Q)}(J_Q)}}\Bigr)$, observe that for $z\in L_Q\cap\UU_{\tau\ell(Q)}(J_Q)$, $\Pi_{L_Q}(z)=z$.  Then property (2) of $J_Q$ implies that
\begin{equation}
 \begin{split}
   \big|\nabla g_Q(z)\big|&=\big|\nabla(\Pi_{L_Q\sharp}\mu_{|Q_{(t)}}*\varphi_{2\tve\ell(Q)})(z)\big|
   =\big|(\nabla\varphi_{2\tve\ell(Q)}*\Pi_{L_Q\sharp}\psi_{3B_Q}\mu)(z)\big|\\&=\left|\int\nabla\varphi_{2\tve\ell(Q)}(z-y)\,d\Pi_{L_Q\sharp}\psi_{3B_Q}\mu(y)\right|\\&=
   \left|\int\nabla\varphi_{2\tve\ell(Q)}\big(z-\Pi_{L_Q}(y)\big)\psi_{3B_Q}(y)\,d\mu(y)\right|\\&=
   \left|\int\nabla\varphi_{2\tve\ell(Q)}\big(\Pi_{L_Q}(z)-\Pi_{L_Q}(y)\big)\psi_{3B_Q}(y)\,d\mu(y)\right|\\&=
   \left|\int\big(\nabla\varphi_{2\tve\ell(Q)}\circ\Pi_{L_Q})(z-y)\psi_{3B_Q}(y)\,d(\mu-c_Q\HH^n|_{L_Q})(y)\right|\\&\lesssim
   \alpha_\mu^{L_Q}(MB_Q)M^{n+1}\ell(Q)^{n+1}\Lip\left(\big(\nabla\varphi_{2\tve\ell(Q)}\circ\Pi_{L_Q})(z-\cdot)\psi_{3B_Q}\right)
  \end{split}
\end{equation}
To estimate $\Lip\left(\big(\nabla\varphi_{2\tve\ell(Q)}\circ\Pi_{L_Q})(x-\cdot)\psi_{3B_Q}\right)$, we use the fact that
 $$\Lip(\nabla(\varphi_{2\tve\ell(Q)}\circ\Pi_{L_Q}))\leq\Lip(\nabla\varphi_{2\tve\ell(Q)})+\|\nabla\varphi_{2\tve\ell(Q)}\|_{\infty}\frac{C}{\ell(Q)}
\lesssim\frac 1{(\tve\ell(Q))^{n+2}},$$
and then it follows easily also that
$$\Lip\left(\big(\nabla\varphi_{2\tve\ell(Q)}\circ\Pi_{L_Q})(x-\cdot)\psi_{3B_Q}\right)\lesssim\frac 1{(\tve\ell(Q))^{n+2}}.$$
Therefore,
 $$\big|\nabla g_Q(z)\big|\lesssim\frac{\alpha_\mu^{L_Q}(MB_Q)M^{n+1}}{\tve^{n+2}\ell(Q)}.$$
Plugging this estimate in \eqref{s1h}, we get that for $x\in J_Q$,
$$|S^H_1\sigma(x)|\lesssim \alpha_\mu^{L_Q}(MB_Q)\frac{\tau M^{n+1}}{\tve^{n+2}}+\ell(Q)^\alpha.$$
Thus,
$$S_{11}=\|S_1^H\sigma\|^2_{L^2(\sigma_{Q|J_Q})}\lesssim \big(M^{2n+2}\,\tau^2\,\tve^{-2n-4}\,\ve^2 +\ell(Q)^{2\alpha}\big)\,\mu(Q).$$

Therefore, if $\ell(Q)$ and $\ve$ are small enough, we obtain 
$$S_{11}=\|S_1^H\sigma\|^2_{L^2(\sigma_{Q|J_Q})}\leq\frac{\ve_5}{2}\,\mu(Q),$$
and finally
$$\|S_1^H\sigma\|^2_{L^2(\sigma)}=\sum_{\tiny{Q\in{\Ch}_{\Stop}(R)}}\|S_1^H\sigma\|^2_{L^2(\sigma_Q)}\leq \big(\frac{\ve_5}{2}+C\tau^{\gamma_0/2}\big)\mu(R)\leq \ve_5\,\mu(R),$$ 
for $\tau$ small enough.

\vspace{2mm}
\noindent {\bf Estimate of $ \|S_2^H\sigma\|_{L^2(\sigma)}$.}
First we will estimate $\|S^H_2\sigma\|_{L^2(\sigma)}$ in terms of $\|S^H_2\sigma\|_{L^2(\tmu)}$. 
Recall that, by definition, $\tsigma_Q=\Pi_{L_Q\sharp}\mu_{|Q_{(t)}}$. By Fubini
\begin{equation}\label{sigmatotmu}
\begin{split}
\|S^H_2\sigma\|^2_{L^2(\sigma)}&=\int|S_2^H\sigma(x)|^2d\sigma(x)=\sum_{{\tiny{Q\in{\Ch}_{\Stop}(R)}}}\int|S^H_2\sigma(x)|^2d\sigma_Q(x) \\&
=\sum_{\tiny{Q\in{\Ch}_{\Stop}(R)}}\int|S^H_2\sigma(x)|^2\left(\tsigma_Q*\varphi_{2\tve\ell(Q)}\right)(x)\,d\HH^n|_{L_Q}(x)\\&
=\sum_{\tiny{Q\in{\Ch}_{\Stop}(R)}}\int\left(\varphi_{\tve\ell(Q)}*|S^H_2\sigma|^2\HH^n|_{L_Q}\right)(x)\,d\tsigma_Q(x)\\&\leq
\sum_{\tiny{Q\in{\Ch}_{\Stop}(R)}}\int\sup_{\tiny{\begin{array}{l}|y-x|\leq2\tve\ell(Q)\\\hspace{6mm}y\in L_Q\end{array}}}\!\!|S^H_2\sigma(y)|^2d\tsigma_Q(x)\\&
=\sum_{\tiny{Q\in{\Ch}_{\Stop}(R)}}\int\sup_{\tiny{\begin{array}{l}|y-\Pi_{L_Q}(z)|\leq2\tve\ell(Q)\\\hspace{9mm}y\in L_Q\end{array}}}\!\!|S^H_2\sigma(y)|^2d\mu_{|Q_{(t)}}(z)
\\&\leq\sum_{\tiny{Q\in{\Ch}_{\Stop}(R)}}\int\sup_{\tiny{\begin{array}{l}|y-z|\leq3\tve\ell(Q)\\\hspace{6mm}y\in L_Q\end{array}}}\!\!|S^H_2\sigma(y)|^2d\mu_{|Q_{(t)}}(z),
\end{split}
\end{equation}
since $|y-z|\leq|y-\Pi_{L_Q}(z)|+|\Pi_{L_Q}(z)-z|\leq 3\tve\ell(Q)$. 
For such $y,z$, we write
$$\big|S^H_2\sigma(y)\big|\leq\big|S^H_2\sigma(z)\big|+\int\big|K_{S_2^H}(y,x)-K_{S_2^H}(z,x)\big|d\sigma(x).$$
Taking into account that $\big|K_{S_2^H}(y,\cdot)-K_{S_2^H}(z,\cdot)\big|$ is supported
in 
$$A\Bigl(y,\tfrac12\tau\ell(Q),\tfrac12M\ell(Q)\Bigr) \cup A\Bigl(z,\tfrac12\tau\ell(Q),\tfrac12M\ell(Q)\Bigr)$$
and that $\tve\ll\tau$,
by Lemma \ref{lemcz}, we deduce
\begin{equation}
\begin{split}
\int\big|K_{S_2^H}(y,x)-K_{S_2^H}(z,x)\big|d\sigma(x)\lesssim
\int_{\frac14\tau\ell(Q)\leq|x-y|\leq M\ell(Q)}\frac{|y-z|^{\alpha/2}}{|x-y|^{n+\alpha/2}}d\sigma(x)\lesssim \tve^{\alpha/2} M^n\,\tau^{-n-{\alpha/2}}.
\end{split}
\end{equation}
Therefore, by \eqref{sigmatotmu}, $$\|S^H_2\sigma\|^2_{L^2(\sigma)}\lesssim \|S^H_2\sigma\|^2_{L^2(\tmu)}+\tve^{{\alpha}} M^{2n}\,\tau^{-2n-{\alpha}}\mu(R).$$

Notice that arguing as in \eqref{difS3}, for $x\in\UU_{10\tve\ell(Q)}(Q_{(t)})$, we get
\begin{equation}\begin{split}
|S^H_2\sigma(x)|&\leq |S^H_2\sigma(x)-S^H_2\tmu(x)|+|S^H_2\tmu(x)|\\&\lesssim M^{\alpha/2}\tve^{\alpha/2} C(t,\tau)\underset{\tiny{P\in{\Ch}_{\Stop}(R)}}
{\sum}\frac{\ell(P)^{\alpha/2}\mu(P)}{D(P,Q)^{n+{\alpha/2}}}+|S^H_2\tmu(x)|.\end{split}
\end{equation}
Define $g$ as in \eqref{g}. Arguing as in \eqref{norma2S3}, \eqref{dual} and \eqref{dualpart}, we get
\begin{equation}\label{norma2S2}
\begin{split}
\|S^H_2\sigma\|^2_{L^2(\tmu)}&\lesssim\tve^{\alpha}M^{\alpha}C(t,\tau)\|g\|^2_{L^2(\tmu)}+\|S^H_2\tmu\|^2_{L^2(\tmu)}.
\end{split}
\end{equation}
Therefore, estimating $\|g\|^2_{L^2(\tmu)}$ by duality as it was done in the estimate of $S^H_3$, we have
\begin{equation}\label{eqs28io}
\|S^H_2\sigma\|^2_{L^2(\sigma)}\lesssim\tve^{\alpha}C(M,t,\tau)\mu(R)+\|S^H_2\tmu\|^2_{L^2(\mu_{|R})}.
\end{equation}

\vspace{2mm}
\noindent {\bf Estimate of $\|S^H_2\tmu\|^2_{L^2(\mu_{|R})}$.}
We write
\begin{equation}\label{eqsp832}
\begin{split}
\|S^H_2\tmu\|^2_{L^2(\mu_{|R})}&\lesssim\int|S^H_2(\mu_{|R})(x)|^2d\tmu(x)\\
& \quad+
\int|S^H_2(\mu_{|R})(x)|^2d|\mu|_R-\tmu|(x) +
\int_R|S^H_2\tmu-S^H_2(\mu_{|R})|^2d\mu(x).
\end{split}
\end{equation} 
Concerning the last term on the right hand side, by \rf{eqdif79} and the fact that the maximal operator $S^H_{*,\mu}$ is
bounded in $L^2(\mu|_R)$, we derive
$$\int_R|S^H_2\tmu-S^H_2(\mu_{|R})|^2d\mu(x)\lesssim \|\tmu-\mu|_R\|\lesssim t^{\gamma_0}\,\mu(R).$$
To deal with the second term on the right hand side of \rf{eqsp832} we argue analogously, using Cauchy-Schwarz and the $L^4(\mu|_R)$ boundedness of $S^H_{*,\mu}$. Then we get
$$\int|S^H_2(\mu_{|R})(x)|^2d\bigl|\mu|_R-\tmu\bigr|(x)\lesssim  t^{\gamma_0/2}\,\mu(R).$$

Finally we turn our attention to the first term. For $x\in Q\in{\Ch}_{\Stop}(R)$, we write $T^H_{2,x}$ for the corresponding frozen operator related to the kernel $\wh K_{2,x}$. Taking into 
account that $M\ell(Q)\ll\Delta\ell(R)$, we write
\begin{equation}
 \begin{split}
 \big|S^H_2(\mu_{|R})(x)\big|&=\big|\wh T^H_2(\mu_{|R})(x)\big|\leq\big|\wh T^H_{2,x}(\mu_{|R})(x)\big|+ \big|\wh T^H_2(\mu_{|R})(x)-\wh T^H_{2,x}(\mu_{|R})(x)\big|\\&\leq\big|\wh T^H_{2,x}(\mu_{|R})(x)\big|+ C\ell(R)^{\alpha/2}\\&\leq
\big |\wh T^H_{2,x}(c_Q\HH^n|_{L_Q})(x)\big|+\Big|\int\wh K^H_{2,x}(x-y)d(\mu-c_Q\HH^n|_{L_Q})(y)\Big|+ C\ell(R)^{\alpha/2}\\&=S_{21}+S_{22}+C\ell(R)^{\alpha/2}.
 \end{split}
\end{equation}
Notice that $\wh T_{2,x}^H(c_Q\HH^n|_{L_Q})(x')=0$ for $x'=\Pi_{L_Q}(x)$, $x\in Q$. Therefore, using  the standard estimates in Lemma \ref{lemcz}  we get
\begin{equation}
 \begin{split}
 S_{21}&=\big|\wh T_{2,x}(c_Q\HH^n|_{L_Q}))(x)-\wh T_{2,x}(c_Q\HH^n|_{L_Q})(x')\big|\\&
 \lesssim\int_{100MB_Q}|\wh K_{2,x}(x-y)-\wh K_{2,x}(x'-y)|\,d\HH^n|_{L_Q}(y)\lesssim C(M,\tau)\tve^{{\alpha/2}}.
 \end{split}
\end{equation}
To estimate the term $S_{22}$, we will use the fact that $\alpha_{\mu}^{L_Q}(MB_Q)\leq\ve$, that is 
$$S_{22}\leq C(M)\Lip(\wh K_2^H)\,\alpha_{\mu}^{L_Q}(MB_Q)\ell(Q)^{n+1}\lesssim C(M,\tau)\varepsilon.$$
Hence,
$$\int|S^H_2(\mu_{|R})(x)|^2d\tmu(x)
\lesssim C(M,\tau)(\tve^{\alpha} + \ve^2 + \ell(R)^\alpha)\mu(R).$$

Gathering the estimates above we get
\begin{equation}
   \|S_2^H\tmu\|^2_{L^2(\mu_{|R})}\lesssim \left( t^{\gamma_0/2} +C(M,\tau)(\tve^{\alpha} + \ve^2 + \ell(R)^\alpha)\right)  \mu(R)\leq \frac{\ve_5}{10}\,\mu(R),
\end{equation}
by choosing  $t,\ell(R),\ve, \tilde\ve$ small enough.
Together with \rf{eqs28io}, this implies that
$$\|S^H_2\sigma\|^2_{L^2(\sigma)}\lesssim\tve^{\alpha}M^{\alpha}C(t)\mu(R) +\frac{\ve_5}{10}\,\mu(R) \mu(R)\leq \frac{\ve_5}{5}\,\mu(R),$$
by appropriate choices of $M$, $t$ and $\tve$ again.

\vspace{2mm}
\noindent {\bf End of the proof of Lemma \ref{lemftr2}}.}
Taking into account that
$$\biggl\|\underset{\tiny{Q\in{\Ch}_{\Stop}(R)}}{\sum}m_{\mu,Q}(S^H_2\tmu)\chi_Q\biggr\|^2_{L^2(\mu|_R)}\leq \|S^H_2\tmu\|^2_{L^2(\mu_{|R})},$$
from the splitting \rf{eqspl9341} and the estimates obtained for $\|S_1^H\sigma\|_{L^2(\sigma)}$, $\|S_2^H\sigma\|_{L^2(\sigma)}$, $\|S_3^H\sigma\|_{L^2(\sigma)}$, and $\|S^H_2\tmu\|^2_{L^2(\mu_{|R})},$
we derive
\begin{equation}\label{eqsu181}
\begin{split}
\big\|S^H\sigma\big\|^2_{L^2(\sigma)} 
& \lesssim\|\tilde f^H\|_{L^2(\mu)}^2 + \underset{\tiny{Q\in{\Ch}_{\Stop}(R)}}{\sum} |m_{\mu,Q}(S^H_1\tmu)|^2\,\mu(Q)+
\frac{\ve_5}{2}\,\mu(R).
\end{split}
\end{equation}
Hence to conclude the proof of the lemma it just remains to estimate the second term on the right hand side above.

For a fixed cube $Q\in{\Ch}_{\Stop}(R)$, we write
\begin{equation}\label{eqsu83}
\mu(Q)\,\big|m_{\mu,Q}(S^H_1\tmu)\big| \leq \left|\int_Q  S^H_1(\chi_Q\mu)\,d\mu\right| + \left|
\int_Q  S^H_1(\tmu-\chi_Q\mu)\,d\mu\right|.
\end{equation}
To estimate the first term on the right hand side, recall that by \rf{remantis}, the difference between the
kernel of $S_{1,Q}^H$ and its antisymmetric part satisfies
$$\Bigl|
K_{S_{1,Q}^H}(x,y) - K_{S_{1,Q}^{H,(a)}}(x,y)\Bigr| \lesssim \frac1{|x-y|^{n-\alpha/2}},$$
and so  
\begin{equation}\label{eqsu84}
\left|\int_Q  S^H_1(\chi_Q\mu)\,d\mu\right|\lesssim \int_Q  \frac1{|x-y|^{n-\alpha/2}}\,d\mu\lesssim
\ell(Q)^{n+\alpha/2}\lesssim \ell(R)^{\alpha/2}\,\mu(Q).
\end{equation}
Concerning the second term on the right hand side of \rf{eqsu83}, observe that
$$\tmu-\chi_Q\mu = \chi_{Q^c}\tmu - \chi_{Q\setminus Q_{(t)}}\mu.$$
Then, 
using the fact that $\supp K_{S_{1,Q}^H}(x,\cdot)\subset B(x,M^{-1}\ell(Q))$ and Cauchy-Schwarz we deduce
\begin{align}
\left|
\int_Q  S^H_1(\tmu-\chi_Q\mu)\,d\mu\right| & = 
\left|\int_Q  \big|S^H_1(\chi_{\mathcal U_{M^{-1}\ell(Q)}(Q)\setminus Q}\tmu
-\chi_{Q\setminus Q_{(t)}}\mu
)\,d\tmu\right|\\
& \leq 
\big\|S^H_1(\chi_{\mathcal U_{M^{-1}\ell(Q)}(Q)\setminus Q}\tmu
-\chi_{Q\setminus Q_{(t)}}\mu
)\big\|_{L^2(\mu|_Q)} 
\,\mu(Q)^{1/2}.
\end{align}
Notice that $S_{1,Q}^H(\cdot \mu)$ is bounded in $L^2(\mu|_R)$. Indeed, one can easily check that for all
$g\in L^2(\mu|_R)$ and all $x\in\R^{n+1}$,
$$|S_{1,Q}^H(g\,\mu)(x)|\le S_{*}^H(g\,\mu)(x),$$
where $S_{*}^H(\cdot\,\mu)$ is the maximal operator associated with $S^H(\cdot\,\mu)$ and then the claim follows from Cotlar's inequality. This fact, together with the thin boundary condition for $Q$ yields
\begin{align}
\big\|S^H_1(\chi_{\mathcal U_{M^{-1}\ell(Q)}(Q)\setminus Q}\tmu
-\chi_{Q\setminus Q_{(t)}}\mu
)\big\|_{L^2(\mu|_Q)} &  \lesssim
\mu\big(
\mathcal U_{\tau\ell(Q)}(Q)\setminus Q\big)^{1/2} + \mu\big(Q\setminus Q_{(t)}\big)^{1/2}\\
&\lesssim (\tau^{\gamma_0/2} + t^{\gamma_0/2})\,\mu(Q)^{1/2}.
\end{align}
Therefore,
$$\left|
\int_Q  S^H_1(\chi_{Q^c}\tmu)\,d\tmu\right|\lesssim (\tau^{\gamma_0/2}+ t^{\gamma_0/2})\,\mu(Q).$$

Together with \rf{eqsu83} and \rf{eqsu84}, the last estimate yields
$$\big|m_{\mu,Q}(S^H_1\tmu)\big|\lesssim \ell(R)^{\alpha/2} + \tau^{\gamma_0/2} + t^{\gamma_0/2}.$$
Plugging this into \rf{eqsu181}, we get
$$
\big\|S^H\sigma\big\|^2_{L^2(\sigma)} 
\lesssim\|\tilde f^H\|_{L^2(\mu)}^2 + \big(\ell(R)^{\alpha} + \tau^{\gamma_0} + t^{\gamma_0}\big)\,\mu(R)+
\frac{\ve_5}{2}\,\mu(R),$$
which proves the lemma by choosing $\tau$, $t$, and $\ell(R)$ small enough.


\section{The continuous measure $\nu$}\label{secnou9}

We consider $R\in\Nice$ and $\sigma$ as above.
Because of technical reasons, it is convenient to replace $\sigma$ by a continuous measure $\nu$ (i.e., a measure absolutely 
continuous with respect to Lebesgue measure). 
Let $\varphi$ be a radial non-negative $C^\infty$ function supported in $B(0,1)$ such that $\int\varphi\,d\LL^{n+1}=1$, and set
\begin{equation}\label{eqdefnus}
\nu = \sigma * \frac1{s^{n+1}}\,\varphi\left(\frac{\cdot}s\right),
\end{equation}
where $s$ is small enough and will be fixed below. For the moment, let us say that
$s\ll\min_{Q\in{\Ch}_{\Stop}(R)}\ell(Q)$.

Recall that, by Lemma \ref{lemgrowthsigma}, $\sigma$ has $n$-polynomial growth. It is immediate to check
that the same holds for $\nu$, that is
\begin{equation}
\label{lemgrowthnu}
\nu(B(x,r))\leq C\,r^n\quad\mbox{ for all $x\in\R^{n+1}$, $r>0$.}
\end{equation}

The estimate in the following lemma is the analogue of \rf{eqsh421} in Lemma \ref{lemapprox0} with $\nu$ replacing $\sigma$.

\begin{lemm}\label{lemaux49}
Assume $s>0$ small enough in the definition of $\nu$ and $\ell(R)\leq1$.
We have
$$\int |\wh T\nu|^2\,d\nu\leq C\,\ell(R)^n.$$
\end{lemm}

We remark that the smallness requirement on $s$ in this lemma may depend on the number of cubes in ${\Ch}_{\Stop}(R)$, thus the value of
the threshold is merely qualitative.

\begin{proof}
By Fubini, Lemma \ref{lemm_freezing} and the $n$-growth of $\sigma$, 
\begin{align}\label{alg93}
\int |\wh T\nu|^2\,d\nu & = \int |\wh T\nu|^2\,d(\varphi_s*\sigma) = \int (|\wh T\nu|^2)*\varphi_s\,d\sigma\\
&\leq \int \sup_{|x-y|\leq s}|\wh T\nu(y)|^2\,d\sigma(x)\leq \int \sup_{|x-y|\leq s}|\wh T_y\nu(y)|^2\,d\sigma(x) + C\,\ell(R)^{n+{\alpha}}.
\end{align}
For all $x\in\supp\sigma$ and $y$ such that $|y-x|\leq s$, we write
\begin{equation}\label{eqfi9432}
|\wh T_y\nu(y)| = |(\varphi_s*\wh T_y\sigma)(y)|
\leq |(\varphi_s*\wh T_x\sigma)(y)| + |\varphi_s*(\wh T_x\sigma - \wh T_y\sigma)(y)|,
\end{equation}
where $\wh T_y$ stands for the frozen operator with kernel $\nabla_1\Theta(\cdot,0;A(y))$.
To estimate the last term on the right hand side, observe that one can estimate the kernel of $\wh T_x -\wh T_y$  as in \eqref{estim_diff_kern_hold}.
Recall that $\sigma$ is supported in a finite union of hyperplanes, and that it has a smooth density  with respect to $\HH^n$ on each hyperplane. Then one easily gets
\begin{equation}\label{eqhold45}
\bigl|\wh T_x \sigma(y) -\wh T_y \sigma(y) \bigr|\leq C(\sigma)\, |x-y|^{\alpha/2},
\end{equation}
with  $C(\sigma)$ depending on the precise form of $\sigma$ (like the number of cubes in ${\Ch}_{\Stop}(R)$, for example). 

Concerning the first term on the right hand side of \eqref{eqfi9432}, we claim that, if $|x-y|\leq s$,
$$|(\varphi_s*\wh T_x\sigma)(y)|\leq |\wh T_{x,s}\sigma(x)| + C,$$
where $\wh T_{x,s}$ stands for the $s$-truncated version of $\wh T_{x}$. The arguments to show this are quite standard, but we
show the details for the reader's convenience. We write
$$|(\varphi_s*\wh T_x\sigma)(y)| \leq \big|\big(\varphi_s*\wh T_x(\chi_{B(y,2s)}\sigma)\big)(y)\big| +  \big|\big(\varphi_s*\wh T_x(\chi_{B(y,2s)^c}\sigma)\big)(y)\big|.$$
We have
\begin{align}
|\varphi_s*\wh T_x(\chi_{B(y,2s)}\sigma)(y)|& \lesssim \int \varphi_s(y-w) \int_{B(y,2s)} \frac1{|w-z|^n}\,d\sigma(z)\,d\LL^{n+1}(w)\\
&\lesssim \frac 1{s^{n+1}} \int_{B(y,2s)} \int_{|w-z|\leq 3s}\frac1{|w-z|^n}\,d\LL^{n+1}(w)\,d\sigma(z)\\
&\lesssim \frac 1{s^{n+1}} \int_{B(y,2s)} s\,d\sigma(z)\lesssim1.
\end{align}
Also, by standard estimates,
\begin{align}
|\varphi_s*\wh T_x(\chi_{B(y,2s)^c}\sigma)(y)| & \leq \sup_{|y-z|\leq s}|\wh T_x(\chi_{B(y,2s)^c}\sigma)(z)|\\
&\leq |\wh T_{x,s}\sigma(x)| + C\!\sup_{r>s}\frac{\sigma(B(x,r))}{r^n}\leq |\wh T_{x,s}\sigma(x)| + C,
\end{align}
which concludes the proof of our claim.

By \eqref{eqfi9432}, \eqref{eqhold45}, and the claim above, we deduce
$$|\wh T_y\nu(y)|\leq C(\sigma)\, |x-y|^{\alpha/2} + |\wh T_{x,s}\sigma(x)| + C\leq  C(\sigma)\, s^{\alpha/2} + |\wh T_{x,s}\sigma(x)| + C,
$$
since $|x-y|\leq s$.
Plugging this estimate into \eqref{alg93}, we get
$$\int |\wh T\nu|^2\,d\nu \lesssim \int |\wh T_{x,s}\sigma(x)| ^2\,d\sigma(x) + C(\sigma)\,s^{\alpha}\,\ell(R)^n +
C\ell(R)^n + C\,\ell(R)^{n+{\alpha}}.
$$
Taking into account that $\ell(R)\leq1$ and using the connection between the kernels of $\wh T_x$ and $\wh T$ stated in Lemma \ref{lemm_freezing}, we derive
$$\int |\wh T\nu|^2\,d\nu \lesssim \int |\wh T_{s}\sigma|^2\,d\sigma +  C(\sigma)\,s^{\alpha}\,\ell(R)^n +
C\ell(R)^n.
$$
Since $\wh T$ is bounded in $L^2(\sigma)$ (with a qualitative bound on the norm, at least), by standard Calder\'on-Zygmund theory we deduce that
$$\int |\wh T_{s}\sigma|^2\,d\sigma \to \int |\wh T\sigma|^2\,d\sigma \quad\mbox{ as $s\to 0$.}$$
Thus, using also \rf{eqsh421},
$$ \int |\wh T\nu|^2\,d\nu \lesssim \int |\wh T\sigma|^2\,d\sigma + \ell(R)^n\lesssim\ell(R)^n.$$
which proves the lemma.
\end{proof}

Our next objective is to show that $\int |S^H\nu|^2\,d\nu$ is very small if $\int |S^H\sigma|^2\,d\sigma$
is also small. That is, we have to transfer the estimate in Lemma \ref{lemapprox} to the measure $\nu$.
The fact that we are considering just the horizontal component $H$ will be essential in this case.
 We need the following auxiliary result, proven in \cite[Lemma 1]{NTV_acta}.

\begin{lemm}
\label{lemsmoothtolip}
Suppose that $f$ is a $C^2$-smooth compactly supported function on a hyperplane $L$ parallel to $H$. Then the function $\RR^H(f\,\HH^n|_L)$ is a Lipschitz function in $\R^{n+1}$, harmonic outside
$\supp(f\,\HH^n|_L)$, and it satisfies
$$
\sup |\RR^H(f\,\HH^n|_L)|\le CD^2\sup_{L}|\nabla_ H^2 f|
$$
and
$$
\|\RR^H(f\,\HH^n|_L)\|_{\Lip}\le CD\sup_{L}|\nabla_ H^2 f|,
$$
where $D$ is the diameter of $\supp (f\,\HH^n|_ L)$ and $\nabla_H$ is the partial gradient
involving only the derivatives in the directions parallel to $H$.
\end{lemm}

Note that the second differential $\nabla_ H^2 f$
and the corresponding supremum on the right hand side are considered on $L$ only (the function $f$ in
the lemma does not even need to be defined outside $L$) while the $H$-restricted Riesz
transform $\RR^H(f\,\HH^n|_L)$ on the left hand side is viewed as a function on the entire space
$\R^{n+1}$
and its supremum and the Lipschitz norm are also taken in $\R^{n+1}$.

\begin{rem}
Below, we will apply Lemma \ref{lemsmoothtolip} to the operator $\wh T_x$, by means of the change of variable $\phi(y) = \wh A(x)^{1/2}\,y$.
Note that then the matrix $A_\phi$ in Corollary \ref{cor:A(x0)=id} coincides with the identity, and
thus the operator $T_{\phi}$ in \eqref{eq:tphinu} equals the Riesz transform, modulo a universal
factor. Hence, by Lemma \ref{lemfac*1}, denoting $D_x = \wh A(x)^{1/2}$, for any measure $\eta$
we have
\begin{equation}\label{eqriesz*0}
c_n\,\RR\eta(y)=D_x\,\wh T_x((D_x)_{\sharp}\eta)(D_x y),
\end{equation}
for all $x,y$.
\end{rem}


\begin{lemm}\label{lem16}
Assume $s>0$ small enough in the definition of $\nu$ and let $\ve'>0$.
 If $\|T_R{\mu}\|^2_{L^2(\mu)}\leq\varepsilon_1\,\mu(R)$, then 
$$\int |S^H\nu|^2\,d\nu\lesssim \ve' \,\ell(R)^n,$$
assuming that $\ve$, $\ve_1$, $\ell(R)$, $t$, and $\Delta$ are small enough and $M$ is big enough (as in Lemma \ref{lemapprox}).
\end{lemm}

\begin{proof}
Recall  that
$$
S^H\nu(x) = \wh T^H\nu(x) - \wh T^H\nu(x^*).$$


Consider the matrix $D_x= \wh A(x)^{1/2}$ and the hyperplane
$H_x= D_x^{-1}(H)$. 
Then we write
\begin{equation}\label{eqfir33}
\int |S^H\nu|^2\,d\nu  \lesssim \int |\Pi_{H_x}\,D_x\,\wh T_x \nu(x) - \wh T^H\nu(x)|^2\,d\nu(x) +
\int |\Pi_{H_x}\,D_x\,\wh T_x \nu(x) - \wh T^H\nu(x^*)|^2\,d\nu(x).
\end{equation}
To estimate the first integral on the right hand side we claim that
\begin{equation}\label{eqclaim76}
|\Pi_{H_x}\,D_x\,\wh T_x \nu(x) - \Pi_H \wh T\nu(x)|\lesssim  \ell(R)^{\alpha/2}\,(1+|\wh T\nu(x)|)\quad \mbox{ for
all $x\in\supp\nu$,}
\end{equation}
and also that the same estimate holds replacing $\nu$ by $\sigma$. That is,
\begin{equation}\label{eqclaim77}
|\Pi_{H_x}\,D_x\,\wh T_x \sigma(x) - \Pi_H \wh T\sigma(x)|\lesssim\ell(R)^{\alpha/2}\,(1+|\wh T\sigma(x)|)
\quad \mbox{ for
all $x\in\supp\sigma$.}
\end{equation}

To prove \eqref{eqclaim76}, we fix $x\in\supp\nu$ and we set
\begin{equation}\label{eqcl8g}
|\Pi_{H_x}\,D_x\,\wh T_x \nu(x) - \Pi_H \wh T\nu(x)|  \leq 
|\Pi_{H_x}\,D_x\,(\wh T_x\nu(x) -\wh T\nu(x))| + |(\Pi_{H_x}\,D_x - \Pi_H)\,\wh T\nu(x)|.
\end{equation}
Now we estimate the first summand on the right hand side:
\begin{equation}\label{projDxThat}
|\Pi_{H_x}\,D_x\,(\wh T_x\nu(x) -\wh T\nu(x))| 
\lesssim|\wh T_x\nu(x) -\wh T\nu(x)|\lesssim\int \frac1{|x-y|^{n-{\alpha/2}}}\,d\nu(y) \lesssim\ell(R)^{\alpha/2},
\end{equation}
using \eqref{lemgrowthnu} in the last inequality.

Concerning the last summand on the right hand side of \eqref{eqcl8g}, we have
$$|(\Pi_{H_x}\,D_x - \Pi_H)\,\wh T\nu(x)|\leq
 \bigl(\|\Pi_{H_x} D_x- \Pi_{H_x}\| +  \|\Pi_{H_x} - \Pi_{H}\|\bigr)\,|\wh T\nu(x)|.$$
 By the H\"older continuity of $\wh A$, we have 
 $$\|\Pi_{H_x} D_x- \Pi_{H_x}\|\leq \|D_x- Id\|\lesssim|x-x_R|^{\alpha/2}\leq C\,\ell(R)^{\alpha/2}.$$
Also, taking into account that $H_x=D_x^{-1}(H)$, we get
$$\|\Pi_{H_x} - \Pi_{H}\| \lesssim\|D_x- Id\|\lesssim\ell(R)^{\alpha/2}.$$
Thus, 
$$|(\Pi_{H_x}\,D_x - \Pi_H)\,\wh T\nu(x))|\lesssim\ell(R)^{\alpha/2}|\wh T\nu(x)|,$$
which together with \eqref{projDxThat} concludes the proof of \eqref{eqclaim76}. The arguments for \eqref{eqclaim77} are analogous and are left for the reader.

From the claim \eqref{eqclaim76} and applying Lemma \ref{lemaux49}, we derive
$$\int |\Pi_{H_x}\,D_x\,\wh T_x \nu(x) - \wh T^H\nu(x)|^2\,d\nu(x)\lesssim \ell(R)^{\alpha} \left(\ell(R)^n + \int
|\wh T\nu|^2\,d\nu\right) \lesssim \ell(R)^{n+\alpha}.$$

To deal with the second integral on the right hand side of \eqref{eqfir33}, we write
\begin{align}
\int |\Pi_{H_x}\,D_x\,\wh T_x \nu(x) - \wh T^H\nu(x^*)|^2\,d\nu(x) & \lesssim 
\int |\Pi_{H_x}\,D_x\,\wh T_x \sigma(x) - \wh T^H\sigma(x^*)|^2\,d\sigma(x)\\
&\quad+
\left|\int |\Pi_{H_x}\,D_x\,\wh T_x \sigma(x) - \wh T^H\sigma(x^*)|^2\,d(\sigma-\nu)(x)\right|\\
& \quad +
\int |\Pi_{H_x}\,D_x\,\wh T_x \sigma(x) - \Pi_{H_x}\,D_x\,\wh T_x \nu(x)|^2\,d\nu(x)\\
&\quad +
\int |\wh T^H\sigma(x^*)- \wh T^H\nu(x^*)|^2\,d\nu(x)\\
& =: I_1+ I_2 + I_3 + I_4.
\end{align}

To deal with the term $I_1$ we apply \eqref{eqclaim77} and Lemmas \ref{lemapprox} and \ref{lemapprox0}, and then we get
\begin{align}
I_1 & \lesssim \int |S^H\sigma|^2 \,d\sigma +
\int |\Pi_{H_x}\,D_x\,\wh T_x \sigma(x) - \wh T^H\sigma(x)|^2\,d\sigma(x)\\
&\lesssim \int |S^H\sigma|^2 \,d\sigma + \ell(R)^{\alpha} \left(\ell(R)^n + \int
|\wh T\sigma|^2\,d\sigma\right) \\&\lesssim (\ve_2 + \ell(R)^{\alpha})\,\ell(R)^n.
\end{align}

Next we consider the integral $I_3$. To this end, observe that for any given $x$, since $\wh T_x$ is a convolution operator, 
$$\Pi_{H_x}\,D_x\,\wh T_x \nu(x) = \Pi_{H_x}\,D_x\,\wh T_x (\varphi_s*\sigma)(x) = \varphi_s*\bigl(\Pi_{H_x}\,D_x\,\wh T_x \sigma\bigr)(x).$$
Therefore,
\begin{equation}\label{eqsp83w}
\begin{split}
|\Pi_{H_x}\,D_x\,\wh T_x \sigma(x) - \Pi_{H_x}\,D_x\,\wh T_x \nu(x)| & =
\bigl|\Pi_{H_x}\,D_x\,\wh T_x \sigma(x) - \varphi_s*\bigl(\Pi_{H_x}\,D_x\,\wh T_x \sigma\bigr)(x)\bigr|\\
& \leq\sup_{|y-x|\leq s} |\Pi_{H_x}\,D_x\,\wh T_x \sigma(x) - \Pi_{H_x}\,D_x\,\wh T_x \sigma(y)|.
\end{split}
\end{equation}
Recall now that, by \eqref{eqriesz*0},
\begin{equation}\label{eqriesz*1}
D_x\,\wh T_x \sigma(x) = c_n\RR(D_{x^{-1}\sharp}\sigma)(D_x^{-1}x).
\end{equation}
Since $\sigma$ is supported on a finite union of planes parallel to $H$, it follows that the measure $D_{x^{-1}\sharp}\sigma$ is supported on a finite union of planes which are parallel to $H_x=D_x^{-1}H$. Then, by Lemma \ref{lemsmoothtolip} (applied with $H_x$ instead of $H$), it turns out that
$\Pi_{H_x}\,D_x\,\wh T_x \sigma(\cdot)$ is a Lipschitz function (with the Lipschitz norm depending on the precise construction of $\sigma$, and in particular on the number of cubes in ${\Ch}_{\Stop}(R)$). Hence, the right hand side of \rf{eqsp83w} tends to $0$ uniformly on $x$ as $s\to0$, so
$$|\Pi_{H_x}\,D_x\,\wh T_x \sigma(x) - \Pi_{H_x}\,D_x\,\wh T_x \nu(x)|\to0\quad\mbox{as $s\to0$,}$$
uniformly on $x$ too. This implies that
$$I_3= I_3(s) = \int |\Pi_{H_x}\,D_x\,\wh T_x \sigma(x) - \Pi_{H_x}\,D_x\,\wh T_x \nu(x)|^2\,d\nu(x)\to 0
\quad\mbox{as $s\to0$.}$$

To estimate $I_4$, note that
$$\wh T^H\nu(x^*) = \int \wh K^H(x^*,y)\,d\nu(y) = \int \bigl(\wh K^H(x^*,\cdot) * \varphi_s\bigr)(y)\,d\sigma(y).$$
By the H\"older continuity of $\wh K^H(x^*,\cdot)$ with $x\in\supp\sigma$, it follows easily that
$\wh T^H\nu(x^*)\to \wh T^H\sigma(x^*)$ as $s\to 0$ uniformly for $x\in\supp\sigma$, taking into account 
also that for $x\in\supp\sigma\cup\supp\nu$,
$$\dist(x^*,\supp\sigma\cup\supp\nu)\gtrsim \Delta\,\ell(R)\gg s,$$
for $s$ small enough. Then we deduce that 
$$I_4 = I_4(s) \to 0\qquad \mbox{as $s\to 0$.}$$

Finally we turn our attention to the term $I_2$. Observe that
\begin{align}
I_2 & = 
\left|\int |\Pi_{H_x}\,D_x\,\wh T_x \sigma(x) - \wh T^H\sigma(x^*)|^2\,d(\sigma-\varphi_s*\sigma)(x)\right|\\
& \leq 
\int \bigl||\Pi_{H_x}\,D_x\,\wh T_x \sigma(x) - \wh T^H\sigma(x^*)|^2 - \varphi_s*\bigl(|\Pi_{H_x}\,D_x\,\wh T_x \sigma(x) - \wh T^H\sigma(x^*)|^2\bigr)\bigr|\,d\sigma(x)\\
&\lesssim \ell(R)^n\,\sup_{\substack{x\in\supp\sigma\\|y-x|\leq s}}
\bigl||\Pi_{H_x}\,D_x\,\wh T_x \sigma(x) - \wh T^H\sigma(x^*)|^2 - |\Pi_{H_y}\,D_y\,\wh T_y \sigma(y) - \wh T^H\sigma(y^*)|^2\bigr|.
\end{align}
We claim now that $\Pi_{H_x}\,D_x\,\wh T_x \sigma(x) - \wh T^H\sigma(x^*)$ is a H\"older continuous function of $x$, for $x$ in a small neighborhood of $\supp\sigma$. Clearly, this implies that
$$\sup_{\substack{x\in\supp\sigma\\|y-x|\leq s}}
\bigl||\Pi_{H_x}\,D_x\,\wh T_x \sigma(x) - \wh T^H\sigma(x^*)|^2 - |\Pi_{H_y}\,D_y\,\wh T_y \sigma(y) - \wh T^H\sigma(y^*)|^2\bigr|\to 0\qquad \mbox{as $s\to 0$,}$$
and thus
$$I_2= I_2(s)\to 0\qquad \mbox{as $s\to 0$.}$$

By the same arguments used to estimate $I_4$, it is easy to check that $\wh T^H\sigma(x^*)$ is a H\"older
continuous function of $x$, for $x$ in a small neighborhood of $\supp\sigma$. Thus, to prove our claim it 
suffices to show that $\Pi_{H_x}\,D_x\,\wh T_x \sigma(x)$ is a H\"older
continuous function of $x$ in that neighborhood. To this end,  for $x,y$ in a small neighborhood of $\supp\sigma$ we write
\begin{align}
\bigl|\Pi_{H_x}\,D_x\,\wh T_x \sigma(x) - \Pi_{H_y}\,D_y\,\wh T_y \sigma(y)\bigr| & \leq
\bigl|\Pi_{H_x}\,D_x\,\wh T_x \sigma(x) - \Pi_{H_x}\,D_x\,\wh T_x \sigma(y)\bigr| \\
&\quad + \bigl| \Pi_{H_x}\,D_x\,(\wh T_x \sigma(y) -\wh T_y \sigma(y)) \bigr|\\
&\quad+\bigl| (\Pi_{H_x}\,D_x - \Pi_{H_y}\,D_y) \, \wh T_y \sigma(y)) \bigr| =: J_1+ J_2+J_3.
\end{align}
By \eqref{eqriesz*1} and Lemma \ref{lemsmoothtolip} (applied with $H_x$ replacing $H$) we have
$$J_1 = c_n\,\bigl|\Pi_{H_x}\RR(D_{x^{-1}\sharp}\sigma)(D_x^{-1}x)- \Pi_{H_x}\RR(D_{x^{-1}\sharp}\sigma)(D_x^{-1}y)\bigr| \leq C(\sigma)\,|D_x^{-1}x - D_x^{-1}y|\leq C(\sigma)|x-y|.$$
 Regarding $J_2$, we have
$$J_2\lesssim \bigl|\wh T_x \sigma(y) -\wh T_y \sigma(y) \bigr|.$$
Recall that $\wh T_x -\wh T_y$ is an odd convolution operator whose kernel $K=\wh K_x-\wh K_y$ is given as in \eqref{eqkern56}, and it satisfies
\begin{equation}\label{eqdifker5}
|K(z)|\leq C\,|x-y|^{\alpha/2}\,\frac1{|z|^n}\quad \text{ and }\quad |\nabla K(z)|\leq C\,|x-y|^{\alpha/2}\,\frac1{|z|^{n+1}}.
\end{equation}
From this fact and the smoothness of the density of $\sigma$ with respect to $\HH^n$ on a finite union of hyperplanes, one easily gets
$$\bigl|\wh T_x \sigma(y) -\wh T_y \sigma(y) \bigr|\leq C(\sigma) |x-y|^{\alpha/2}.$$
 Next we turn to $J_3$:
\begin{align}
J_3 &\leq \|\Pi_{H_x}\,D_x - \Pi_{H_y}\,D_y\| \, |\wh T_y \sigma(y)|\\
&\leq \bigl(\|(\Pi_{H_x} - \Pi_{H_y}) \,D_x\| + \|\Pi_{H_y}\,(D_x - D_y)\|\bigr)
\, |\wh T_y \sigma(y)|\\
& \lesssim \bigl(\|\Pi_{H_x} - \Pi_{H_y}\| + \|D_x - D_y\|\bigr)
\, |\wh T_y \sigma(y)|.
\end{align}
Recall that $D_x =\wh A(x)^{1/2}$ and $H_x=D_x^{-1}(H)$. Then, by the H\"older continuity of $\wh A$, we derive
$$\|\Pi_{H_x} - \Pi_{H_y}\| + \|D_x - D_y\|\lesssim_\sigma |x-y|^{\alpha/2}.$$
Taking into account that $|\wh T_y \sigma(y)|\leq C(\sigma)$, we deduce that
$$J_3\leq C(\sigma)\,|x-y|^{\alpha/2}.$$
Thus $\Pi_{H_x}\,D_x\,\wh T_x \sigma(x)$ is a H\"older
continuous function of $x$ with exponent $\alpha/2$, as claimed.

The lemma follows from the estimates obtained for $I_1$, $I_2$, $I_3$, and $I_4$.
\end{proof}


\section{The function $h$ and the vector field $\Psi$}\label{secpsi}

For each cube $Q$ from the intermediate non-BAUP layer $\NB(R)$ with non-BAUPness
parameter $\delta>0$, we define a function $h_Q$ as follows. First we consider a radial $C^\infty$ function $h_0$ supported
in $B(0,1)$ such that $h_0=1$ on $B(0,1/2)$ and $0\leq h_0\leq1$.
Then we set
$$h_Q(x) = h_0\left(\frac{x-z_Q^a}{\delta\,\ell(Q)}\right) - h_0\left(\frac{x-z_Q^b}{\delta\,\ell(Q)}\right),$$
where $z^a_Q$ and $z^b_Q$ are the points introduced in Definition \ref{definition_NB_cube} and such that
the vector $z^a_Q- z^b_Q$ is parallel to $H$. This can be achieved by taking the hyperplane $L$ in Definition \ref{definition_NB_cube} parallel to $H$.
Note that $\supp h_Q\subset 3B_Q$, and the support of the negative part of 
$h_Q$ does not intersect $\supp\mu$.  On the other hand, the support of the positive part of $h_Q$ includes a sufficiently big portion of the measure, so that $\int h_Q\,d\mu\gtrsim c(\delta)\,\mu(Q)$.

Next, by a Vitali type covering lemma, we extract a subfamily $\NB'(R)\subset\NB(R)$ such that the balls
$4B_Q$, $Q\in\NB'(R)$, are pairwise disjoint and so that
$$\sum_{Q\in\NB'(R)} \mu(Q)\geq c\,\mu(R),$$
where $c$ depends at most on the AD-regularity constant of $\mu$. Then we define
$$h= \sum_{Q\in\NB'(R)} h_Q.$$

\begin{lemm}\label{lemeta}
Assume $\varepsilon$ and the parameter $s$ in the definition of $\nu$ in \rf{eqdefnus} small enough.
Then the function $h$ satisfies: $\supp h\subset 3B_R$, $\dist(\supp h,H)\geq \Delta\,\ell(R)/2$, 
$h\geq0$ on $\supp\nu$
and
$$\int h\,d\nu \geq c_7(\delta)\,\,\nu(\R^{n+1}),$$
with $c_7(\delta)>0$.
\end{lemm}

The proof of this lemma is elementary and follows from the construction of $h$. 
\vspace{2mm}

Our next objective consists in constructing a vector field $\Psi$ satisfying the properties stated in the next lemma.

\begin{lemm}\label{lemapsi}
There exists a compactly supported Lipschitz vector field $\Psi:\R^{n+1}\to\R^{n+1}$ which satisfies the following:
\begin{itemize}
\item[(i)] $\Psi = \sum_{Q\in\NB'(R)}\Psi_Q$, $\supp\Psi\subset 3B_R\cap \R^{n+1}_+$, and $\dist(\supp\Psi,H)\geq \frac\Delta2\,\ell(R)$.

\item[(ii)] For each  $Q\in\NB'(R)$, $\supp\Psi_Q\subset 3 B_Q$ and 
$$\int\Psi_Q\,d\LL^{n+1}=0,\quad \|\Psi_Q\|_\infty\lesssim\frac 1{\delta\,\ell(Q)},\quad \text{and}\quad \|\Psi_Q\|_{\rm Lip}\lesssim 
\frac{1}{\delta^2\ell(Q)^2}.$$

\item[(iii)] $\displaystyle\int |\Psi| \,d\LL^{n+1} \lesssim	\delta^{-1}\,\ell(R)^n$.

\item[(iv)] For each  $Q\in\NB'(R)$, $$\wh T^{H,*}(\Psi_Q\,\LL^{n+1}) = h_Q + e_Q,$$ with the ``error term'' $e_Q$ satisfying
$$|e_Q(x)|\lesssim \frac{C(\delta)\,\ell(R)^{\tilde\gamma}\,\ell(Q)^{n+\tilde\beta}}{(|x-x_Q|+\ell(Q))^{n+\tilde\beta}} \quad\mbox{ for all $x\in 10B_R$,}
$$
where $\tilde\beta$ and $\tilde\gamma$ are some fixed positive constants depending on $n$ and $\alpha$.


\item[(v)] $\|S^H(|\Psi|\LL^{n+1})\|_{L^2(\nu)}\leq C(\delta)\,\mu(R)^{1/2}$, assuming the parameter $s$ in the
definition of $\nu$ small enough.

\end{itemize}
\end{lemm}

We remark that in the statement (iv) above,  $\wh T^{H,*}(\Psi_Q\,\LL^{n+1})$ stands for the adjoint of 
$\wh T^{H}$ applied to the vectorial measure $\Psi_Q\,\LL^{n+1}$. That is,
$$\wh T^{H,*}(\Psi_Q\,\LL^{n+1})(x) = \int \wh K^H(y,x)\cdot \Psi_Q(y)\,d\LL^{n+1}(y),$$
where `$\cdot$' is the scalar product. Sometimes, abusing notation, we will write $\wh T^{H,*}\Psi_Q$ instead of
$\wh T^{H,*}(\Psi_Q\,\LL^{n+1})$. We will use analogous notations for other operators.

\begin{proof}
To construct each function $\Psi_Q$ for $Q\in\NB'(R)$ we argue as in \cite[Section 24]{NTV_acta}.
Let $v_Q$ be the unit vector in the direction $z_Q^a-z_Q^b$. Consider the function
$$
g_Q(x)=\int_{-\infty}^0 h_Q(x+tv_Q)\,dt,
$$
so that $\nabla_{v_Q} g_Q=h_Q$.
Since the restriction of $h_Q$ to any line parallel to $v_Q$ consists of two 
opposite bumps, the support of $h_Q$ is contained in the convex hull of $B(z_Q^a,\delta\ell(Q))$
and $B(z_Q^b,\delta\ell(Q))$. Also, since $\|\nabla^j h_Q\|_{L^\infty}\le C(j)[\delta\ell(Q)]^{-j}$
and since $\supp h_Q$ intersects any line parallel to $v_Q$ over two intervals of total
length $4\delta\ell(Q)$ or less, we have 
\begin{equation}\label{eqgqj1}
|\nabla^j g_Q(x)|\leq \int_{-\infty}^0 |(\nabla^j h_Q)(x+tv_Q)|\,dt\leq 
\frac{C(j)}{[\delta\ell(Q)]^{j-1}}
\end{equation}
for all $j\ge 0$. 

We define the vector fields 
$$
\Psi_Q=-\Delta g_Q \,v_Q,\qquad \Psi=\sum_{Q\in\NB'(R)}\Psi_{Q},
$$
so that the properties (i) and (ii) in the lemma hold, because of \rf{eqgqj1}. 
Indeed, the mean zero property holds because the integral
of the Laplacian of a compactly supported $C^\infty$ function over the entire space
is $0$ and the support property holds because the balls $B(x_Q,3\ell(Q))$ 
lie deep inside $3B_R$. 
The property (iii) is also
immediate:
\begin{align}
\int|\Psi|\,d\LL^{n+1} &=\sum_{Q\in\NB'(R)}\int|\Psi_{Q}|\,d\LL^{n+1} \lesssim
\sum_{Q\in\NB'(R)}[\delta\ell(Q)]^{-1}\,\LL^{n+1}(B(x_{Q},3\ell(Q))) 
\\
&\lesssim\delta^{-1}\sum_{Q\in\NB'(R)}\ell(Q)^n
\lesssim\delta^{-1}\sum_{Q\in\NB'(R)}\mu(Q)\lesssim\delta^{-1}\mu(R)\,.
\end{align}

Next we turn our attention to the statement (iv). 
Since $\wh A(x_R)=Id$, the kernel of $\wh T_{x_R}$ is the gradient of the fundamental solution
of the Laplacian (i.e., the Riesz kernel times an absolute constant). Thus,
$\wh T_{x_R} (\Delta g_Q) =\nabla g_Q$ and so
$\wh T_{x_R}^H (\Delta g_Q)=\nabla_H g_Q$. Therefore, since $v_Q\in H$,
$$\wh T_{x_R}^{H,*} \Psi_Q = 
\wh T_{x_R}^{H,*} (-\Delta g_Q\,v_Q) = \wh T_{x_R}^H (\Delta g_Q)\cdot v_Q = 
\wh T_{x_R} (\Delta g_Q)\cdot v_Q = \nabla_{v_Q} g_Q = h_Q.
$$
Hence,
$$\wh T^{H,*} \Psi_Q = h_Q + \bigl(\wh T^{H,*} \Psi_Q  - \wh T_{x_R}^{H,*} \Psi_Q\bigr) =: h_Q + e_Q.$$

We estimate $e_Q$ as follows:
\begin{align}\label{eqali934}
|e_Q(x)| & \leq \bigl|\wh T^{H,*} \Psi_Q (x) - \wh T_{x}^{H,*}\Psi_Q(x)\bigr| + \bigl|\wh T^{H,*}_x \Psi_Q (x) - \wh T_{x_R}^{H,*}\Psi_Q(x)\bigr|.
\end{align}
For the first summand on the right hand side we write
\begin{align}
\bigl|\wh T^{H,*} \Psi_Q (x) - \wh T_{x}^{H,*}\Psi_Q(x)\bigr| &\leq
\int |\wh K^H(y,x) - \wh K^H_x(y,x)|\,|\Psi_Q(y)|\,d\LL^{n+1}(y) \\
& \lesssim\int \frac1{|x-y|^{n-{\alpha/2}}}\,|\Psi_Q(y)|\,d\LL^{n+1}(y)\\
& \lesssim  \frac 1{\delta\,\ell(Q)} \int_{B(x_Q,3\ell(Q))} \frac1{|x-y|^{n-{\alpha/2}}}\,d\LL^{n+1}(y)\\
& \lesssim  \frac 1{\delta\,\ell(Q)}\, \ell(Q)^{1+{\alpha/2}} \lesssim\delta^{-1} \ell(R)^{\alpha/2}.
\end{align} 
Concerning the last summand in \rf{eqali934}, we write
\begin{align}
\bigl|\wh T^{H,*}_x \Psi_Q (x) - \wh T_{x_R}^{H,*}\Psi_Q(x)\bigr|\leq
\int |\wh K^H_x(y-x) - \wh K^H_{x_R}(y-x)|\,|\Psi_Q(y)|\,d\LL^{n+1}(y).
\end{align}
As in \rf{eqdifker5} we have
$$
|\wh K^H_x(y-x) - \wh K^H_{x_R}(y-x)|\leq |\wh K_x(y-x) - \wh K_{x_R}(y-x)|
\lesssim \frac{|x-x_R|^{\alpha/2}}{|x-y|^n}\lesssim\frac{\ell(R)^{\alpha/2}}{|x-y|^n}
$$
for all $x\in 10B_R$. Hence, for such points $x$,
\begin{align}
\bigl|\wh T^{H,*}_x \Psi_Q (x) - \wh T_{x_R}^{H,*}\Psi_Q(x)\bigr| & \lesssim\ell(R)^{\alpha/2}
\int \frac1{|x-y|^n}\,|\Psi_Q(y)|\,d\LL^{n+1}(y)\\
& \lesssim  \frac{\ell(R)^{\alpha/2}}{\delta\,\ell(Q)}\int_{B(x_Q,3\ell(Q))} \frac1{|x-y|^n}\,d\LL^{n+1}(y)\lesssim\delta^{-1} \ell(R)^{\alpha/2}.
\end{align} 
Therefore,
\begin{equation}\label{eqeq11}
|e_Q(x)|\lesssim\delta^{-1}\ell(R)^{\alpha/2}\quad \mbox{ for all $x\in 10B_R$.}
\end{equation}

On the other hand, we also have
$$|e_Q(x)|\leq \bigl|\wh T^{H,*} \Psi_Q (x)\bigr| + \bigl|\wh T_{x_R}^{H,*} \Psi_Q(x)\bigr|.$$
For $x\in 6B_Q$, we have
$$\bigl|\wh T^{H,*} \Psi_Q (x)\bigr| \lesssim 
\frac1{\delta\,\ell(Q)}\int_{B(x_Q,3\ell(Q))} \frac1{|x-y|^n}\,d\LL^{n+1}(y) \lesssim\delta^{-1}.$$
Using that $\Psi_Q$ has zero mean and standard estimates, for $x\in (6B_Q)^c$ we get
\begin{align}
\bigl|\wh T^{H,*} \Psi_Q (x)\bigr| &\leq \int |\wh K^H(y-x) - \wh K^H(x_Q-x)|\,|\Psi_Q(y)|\,d\LL^{n+1}(y)\\
&\lesssim \frac{1}{\delta\,\ell(Q)}\int_{B(x_Q,3\ell(Q))} \frac{\ell(Q)^{\alpha/2}}{|x-x_Q|^{n+{\alpha/2}}}\,d\LL^{n+1}(y)\\
&\lesssim
 \frac{\ell(Q)^{n+{\alpha/2}}}{\delta\,|x-x_Q|^{n+{\alpha/2}}}.
\end{align}
So we infer that for all $x\in\R^{n+1}$,
$$\bigl|\wh T^{H,*} \Psi_Q (x)\bigr| \lesssim  \frac{\delta^{-1}\ell(Q)^{n+{\alpha/2}}}{(\ell(Q)+|x-x_Q|)^{n+{\alpha/2}}}.$$
The same estimate holds for $\bigl|\wh T^{H,*}_x \Psi_Q (x)\bigr|$, and thus
\begin{equation}\label{eqeq23}
|e_Q(x)| \lesssim  \frac{\delta^{-1}\ell(Q)^{n+{\alpha/2}}}{(\ell(Q)+|x-x_Q|)^{n+{\alpha/2}}}\quad
\mbox{ for all $x\in\R^{n+1}$.}
\end{equation}

Denote $\overline\gamma=\alpha/(2(2n+\alpha))$. Notice that $\overline\gamma\alpha/2=\alpha^2/(4(2n+\alpha))<1/4$ and $(1-\overline\gamma)(n+\gamma)=n+\alpha/4$. So, by taking a suitable weighted geometric mean of \rf{eqeq11} and \rf{eqeq23}, we obtain
$$|e_Q(x)|=|e_Q(x)|^{\overline\gamma}|e_Q(x)|^{1-\overline\gamma}\lesssim \frac{\delta^{-1}\ell(R)^{\alpha^2/(4(2n+\alpha))}\ell(Q)^{n+\alpha/4}}{\big(|x-x_Q|+\ell(Q)\big)^{n+\alpha/4}}$$
for all $x\in 10B_R$, which completes the proof of (iv) by choosing 
$\tilde\gamma=\alpha^2/(4(2n+\alpha))$ and $\tilde\beta=\alpha/4$.

\vspace{2mm}
Finally we turn our attention to the estimate (v).
First we will show that
\begin{equation}\label{eqfifi58}
\|S^H(|\Psi|\LL^{n+1})\|_{L^2(\mu|_R)}^2\leq C(\delta)\,\mu(R).
\end{equation}
We consider the auxiliary measure 
$$\xi = \sum_{Q\in\NB'(R)} \frac1{\ell(Q)}\,\LL^{n+1}|_{3B_Q}.$$
We claim that $\xi$ has $n$-polynomial growth. That is,
\begin{equation}\label{eqgrowthxi}
\xi(B(x,r))\lesssim r^n\quad \mbox{ for all $x\in\R^{n+1}$, $r>0$.}
\end{equation}
The arguments to prove this are standard, but we show the details for the reader's convenience.
It suffices to prove the preceding inequality for $x\in\supp\xi\subset \bigcup_{Q\in\NB'(R)} 3{B_Q}$. 
So fix a point $x\in 3{B_Q}$, for some $Q\in\NB'(R)$. Since the balls $4B_P$, $P\in\NB'(R)$, are pairwise disjoint, it is clear that the condition \rf{eqgrowthxi} 
holds for $r< \ell(Q)$.
In the case $r\geq \ell(Q)$, let $I(x,r)$ denote the family of cubes $P\in\NB'(R)$ such that $3{B_P}\cap B(x,r)\neq
\varnothing$. Taking into account again that the balls $4B_S$, $S\in\NB'(R)$, are pairwise disjoint,
it follows that, for any $P\in I(x,r)$, $r\geq \ell(P)$ and then $B_P\subset B(x,7r)$. Therefore,
$$\xi(B(x,r))\leq \sum_{P\in I(x,r)} \xi(3B_P) \approx \sum_{P\in I(x,r)} \ell(P)^n\leq \sum_{P\in I(x,r)} \mu(P)
\leq \mu(B(x,7r))\lesssim r^n.
$$

Recall now that $\mu|_R$ is $n$-AD-regular and $\wh T_{\mu|_R}$ is bounded in $L^2(\mu|_R)$.  As a consequence, the maximal operator
$$\wh T_{\xi,*}f(x) = \sup_{\ve>0} |\wh T_{\xi,\ve}f(x)| = \sup_{\ve>0} \Big|\int_{|x-y|>\ve} \wh K(x,y)\,f(y)\,d\xi(y)\Big|$$
is bounded from $L^2(\xi)$ to $L^2(\mu|_R)$ (see Proposition 5 from \cite{David-corbes}).

Consider the vector field $\wt\Psi$ defined by
$$\wt\Psi = \sum_{Q\in\NB'(R)}\ell(Q)\,\Psi_Q,$$
so that $|\Psi|\,\LL^{n+1} = |\wt \Psi|\,\xi$. Observe that, by (ii),
$$\|\wt\Psi\|_{L^\infty(\xi)}\lesssim \delta^{-1},$$
and thus
\begin{align}
\|\wt\Psi\|_{L^2(\xi)}^2 
\lesssim\delta^{-2}\!\!\sum_{Q\in\NB'(R)}\ell(Q)^n
\lesssim\delta^{-2}\!\!\sum_{Q\in\NB'(R)}\mu(Q)\lesssim\delta^{-2}\,\mu(R).
\end{align}
For each $x\in R$, we split
$$|S^H(|\Psi|\,\LL^{n+1})(x)| = |S^H(|\wt\Psi|\,\xi)(x)| 
\leq |\wh T(|\wt\Psi|\xi)(x)| + |\wh T(|\wt\Psi|\xi)(x^*)|.$$
By standard estimates, it is also immediate to check that
$$|\wh T(|\wt\Psi|\xi)(x^*)| \leq |\wh T_{*}(|\wt\Psi|\xi)(x)| + M_n (|\wt\Psi|\xi)(x),$$
where $M_n$ is the maximal radial operator
\begin{equation}\label{maximalradial}
M_n \tau(x) = \sup_{r>0}\frac{|\tau|(B(x,r))}{r^n},
\end{equation}
for any signed measure $\tau$.
So we deduce that
$$\|S^H(|\Psi|\,\LL^{n+1})\|_{L^2(\mu|_R)}^2 \lesssim
\|\wh T_{\xi,*}(|\wt\Psi|)\|_{L^2(\mu|_R)}^2 + \|M_n (|\wt\Psi|\xi)\|_{L^2(\mu|_R)}^2.$$
Analogously to $\wh T_{\xi,*}$, the operator $M_n(\cdot\,\xi)$ is also bounded from $L^2(\xi)$ to $L^2(\mu|_R)$ (see
\cite{David-corbes} again). Hence,
\begin{align}\label{equl2**}
\|S^H(|\Psi|\,\LL^{n+1})\|_{L^2(\mu|_R)}^2\lesssim \|\wt\Psi\|_{L^2(\xi)}^2\lesssim\delta^{-2}\,\mu(R).
\end{align}

Our next objective is to prove the analogous estimate in $L^2(\sigma)$, that is,
$$\|S^H(|\Psi|\LL^{n+1})\|_{L^2(\sigma)}^2\leq C(\delta)\,\mu(R).$$
Recall that $\sigma = \sum_{P\in{\Ch}_{\Stop}(R)}\sigma_P$, where $\sigma_P =g_P\,\HH^n|_{L_P}$, with $g_P\lesssim \chi_{2B_P}$. So we have
$$\|S^H(|\Psi|\LL^{n+1})\|_{L^2(\sigma)}^2=  \sum_{P\in{\Ch}_{\Stop}(R)}\|S^H(|\Psi|\LL^{n+1})\|_{L^2(\sigma_P)}^2.$$

For each $P\in{\Ch}_{\Stop}(R)$ we split
\begin{equation}\label{eqspli84}
\|S^H(|\Psi|\LL^{n+1})\|_{L^2(\sigma_P)}^2 \leq 2 \int |S^H(\chi_{3B_P}|\Psi|\LL^{n+1})|^2\,d\sigma_P
+ 2 \int |S^H(\chi_{(3B_P)^c}|\Psi|\LL^{n+1})|^2\,d\sigma_P.
\end{equation}
Concerning the first summand on the right hand side, we have
\begin{equation}\label{eqspli84'}
\int |S^H(\chi_{3B_P}|\Psi|\LL^{n+1})|^2\,d\sigma_P\lesssim 
\int |S^H(\chi_{3B_P}|\Psi|\LL^{n+1})|^2\,d\HH^n|_{L_P}.
\end{equation}
Since $\wh T_{\HH^n|_{L_P}}$ is bounded in $L^2(\HH^n|_{L_P})$, the same argument as in \rf{equl2**} shows
that 
\begin{equation}\label{eqprim584}
\|S^H(\chi_{3B_P}|\Psi|\,\LL^{n+1})\|_{L^2(\HH^n|_{L_P})}^2\lesssim \|\chi_{3B_P}\wt\Psi\|_{L^2(\xi)}^2\lesssim \delta^{-2}\,\ell(P)^n,
\end{equation}
taking into account that $\|\wt\Psi\|_{L^\infty(\xi)}\lesssim \delta^{-1}$ and the polynomial growth of $\xi$ for the last inequality.

To estimate the last integral on the right hand side of \rf{eqspli84} we will show first that
\begin{equation}\label{eqdif832}
\bigl|S^H(\chi_{(3B_P)^c}|\Psi|\LL^{n+1})(x) - S^H(\chi_{(3B_P)^c}|\Psi|\LL^{n+1})(y)\bigr|\lesssim \delta^{-1}\quad
\mbox{ for all $x,y\in 2B_P$.}
\end{equation}
To this end, note that the left hand side above equals
\begin{align}
\bigl|S^H(\chi_{(3B_P)^c}|\wt\Psi|\xi)(x) - S^H(\chi_{(3B_P)^c}|\wt\Psi|\xi)(y)\bigr|
& \leq
\bigl|\wh T_\xi^H(\chi_{(3B_P)^c}|\wt\Psi|)(x) - \wh T_\xi^H(\chi_{(3B_P)^c}|\wt\Psi|)(y)\bigr|\\
& \quad + \bigl|\wh T_\xi^H(\chi_{(3B_P)^c}|\wt\Psi|)(x^*) - \wh T_\xi^H(\chi_{(3B_P)^c}|\wt\Psi|)(y^*)\bigr|.
\end{align}
Taking into account that both $x$ and $y$ are far from the $\supp(\chi_{(3B_P)^c}|\wt\Psi|)$, more precisely,
$|x-y|\lesssim \ell(P)\lesssim \min(\dist(x,(3B_P)^c),\dist(y,(3B_P)^c))$, by standard estimates from Calder\'on-Zygmund theory it follows that
$$\bigl|\wh T_\xi^H(\chi_{(3B_P)^c}|\wt\Psi|)(x) - \wh T_\xi^H(\chi_{(3B_P)^c}|\wt\Psi|)(y)\bigr|\lesssim 
\big\|\chi_{(3B_P)^c}|\wt\Psi|\big\|_{L^\infty(\xi)} \lesssim \delta^{-1}.$$
By analogous reasons, the same estimate holds replacing $x$ by $x^*$ and $y$ by $y^*$. Hence, \rf{eqdif832}
is proven.

From \rf{eqdif832} we infer that
\begin{align}
\big\|S^H(\chi_{(3B_P)^c}|\Psi|\LL^{n+1})\big\|_{\infty,2B_Q} &\leq \bigl|m_{\mu,P}(S^H(\chi_{(3B_P)^c}|\Psi|\LL^{n+1}))\bigr| + C\delta^{-1}\\
& \leq \bigl|m_{\mu,P}(S^H(|\Psi|\LL^{n+1}))\bigr| + \bigl|m_{\mu,P}(S^H(\chi_{3B_P}|\Psi|\LL^{n+1}))\bigr| + C\delta^{-1}.
\end{align}
Arguing again as in \rf{equl2**}, we obtain
$$\bigl|m_{\mu,P}(S^H(\chi_{3B_P}|\Psi|\LL^{n+1}))\bigr|^2
\leq m_{\mu,P}\bigl(|S^H(\chi_{3B_P}|\Psi|\LL^{n+1})|^2\bigr) \lesssim \frac1{\mu(P)} \,\bigl\|\chi_{3B_P}|\wt\Psi|\bigr\|_{L^2(\xi)}^2\lesssim\delta^{-2}.$$
Therefore,
$$\|S^H(\chi_{(3B_P)^c}|\Psi|\LL^{n+1})\|_{\infty,2B_P} \leq \bigl|m_{\mu,P}(S^H(|\Psi|\LL^{n+1}))\bigr|  + C\delta^{-1}.$$
As a consequence,
\begin{align}
\int |S^H(\chi_{(3B_P)^c}|\Psi|\LL^{n+1})|^2\,d\sigma_P \lesssim \bigl|m_{\mu,P}(S^H(|\Psi|\LL^{n+1}))\bigr|^2\,\ell(P)^n + \delta^{-2}\ell(P)^n.
\end{align}
Together with \rf{eqspli84'} and \rf{eqprim584}, this yields
\begin{align}
\int |S^H(|\Psi|\LL^{n+1})|^2\,d\sigma_P & \lesssim \bigl|m_{\mu,P}(S^H(|\Psi|\LL^{n+1}))\bigr|^2\,\ell(P)^n + \delta^{-2}\ell(P)^n\\
&\lesssim \int_P |S^H(|\Psi|\LL^{n+1})|^2\,d\mu + \delta^{-2}\ell(P)^n.
\end{align}
Summing on $P\in{\Ch}_{\Stop}(R)$ and using \rf{eqfifi58}, we obtain
$$\|S^H(|\Psi|\LL^{n+1})\|_{L^2(\sigma)}^2 \lesssim \|S^H(|\Psi|\LL^{n+1})\|_{L^2(\mu|_R)}^2+ \delta^{-2}\ell(R)^n\leq C(\delta)\,\ell(R)^n.$$

To prove the final estimate in (v) we just use the preceding inequality and take into account that
\begin{align}
\int |S^H(|\Psi|\LL^{n+1})|^2\,d\nu & = \int |S^H(|\Psi|\LL^{n+1})|^2\,d(\varphi_s*\sigma) \\
&=
\int (|S^H(|\Psi|\LL^{n+1})|^2)*\varphi_s \,d\sigma\to \int |S^H(|\Psi|\LL^{n+1})|^2\,d\sigma
\end{align}
as $s\to0$, since $|S^H(|\Psi|\LL^{n+1})|^2$ is a continuous function.
\end{proof}


\section{The variational argument} \label{sec12}

In this section we will prove the following:

\begin{prop} \label{lemavar}
Let $R\in\Nice$ and $\nu$ be as in Section \ref{secnou9}. 
Suppose that $\varepsilon$ and $\ell(R)$ are small enough, depending on the non-BAUPness parameter $\delta$.
Then we have
$$\| S^H\nu\|^2_{L^2(\sigma)}\geq c_8(\delta)\,\mu(R).$$
\end{prop}

Together with Lemma \ref{lem16} this shows that, for each $R\in\Nice$, $\|T_R{\mu}\|^2_{L^2(\mu)}\geq\varepsilon_1\,\mu(R)$,
assuming that $\ve$, $\ell(R)$, $t$, and $\Delta$ are small enough and $M$ is big enough. This proves Proposition \ref{prop3} and Theorem \ref{teo1}.

\subsection{A pointwise inequality}

The first step to prove Proposition \ref{lemavar} is the next one.

\begin{lemm}
\label{Le:DesigualdadAERiNu}
Suppose that for some $0<\lambda\leq1$ the inequality
\begin{equation}
\int |S^H\nu|^2 d\nu \leq \lambda \,\nu(\R^{n+1})
\end{equation}
holds. Let $h$ be the function in Lemma \ref{lemeta} and $c_7(\delta)$ the constant in the same lemma.
Then, there is some 
function $b\in L^{\infty}(\nu)$ such that
\begin{enumerate}
\item[(i)]\label{le:3-1} $0\leq b\leq 2$,
\item[(ii)]\label{le:3-3} $\displaystyle\int b\,h\,d\nu \geq c_7(\delta) \,\nu(\R^{n+1})$,
\end{enumerate}
and such that the measure $\eta = b\nu$ satisfies
\begin{equation}\label{eqa331}
\int |S^H\eta|^2 d\eta \leq 2\lambda \,\nu(\R^{n+1})
\end{equation}
and
\begin{equation}
\label{eq:DesCTPNu}
|S^H\eta(x)|^2 + 2S^{H,*}((S^H\eta)\eta)(x) \leq 6c_7(\delta)^{-1}\lambda \quad\mbox{ for $\eta$-a.e. $x\in\R^{n+1}$.}
\end{equation}
\end{lemm}

\begin{proof}
In order to find such a function $b$, we consider the following class of admissible functions
\begin{equation}
\mathcal{A} = \Bigl\{ a\in L^{\infty}(\nu): \,a\geq0,\, \textstyle\int a\,h\,d\nu\geq c_7(\delta)\,\nu(\R^{n+1}) \Bigr\}
\end{equation}
and we define a functional $J$ on $\mathcal{A}$ by
\begin{equation}
J(a) = \lambda \|a\|_{L^\infty(\nu)}\,\nu(\R^{n+1}) + \int |S^H(a\nu)|^2 a\, d\nu.
\end{equation}

Observe that $1\in \mathcal{A}$ and 
$$
 J(1) = \lambda \,\nu(\R^{n+1}) + \int |S^H\nu|^2\, d\nu \leq 2\lambda \,\nu(\R^{n+1}).
$$
Thus 
$$\inf_{a\in\cA} J(a)\leq2\lambda \,\nu(\R^{n+1}).$$
Since $J(a)\geq \lambda \|a\|_{L^\infty(\nu)}\,\nu(\R^{n+1})$,
it is clear that
$$\inf_{a\in\cA} J(a) = \inf_{a\in\cA:\|a\|_{L^\infty(\nu)}\leq 2} J(a).$$
We claim that $J$ attains a global minimum on $\mathcal{A}$, i.e. there is a function $b\in \mathcal{A}$ such that $J(b)\leq J(a)$ for all $a\in \mathcal{A}$. 
Indeed, by the Banach-Alaoglu theorem there exists a sequence $\{a_k\}_k\subset \cA$, with 
$J(a_k)\to \inf_{a\in\cA} J(a)$, $\|a_k\|_{L^\infty(\nu)}\leq 2$, 
so that $a_k$ converges weakly $*$ in $L^\infty(\nu)$ to some function
$b\in \cA$. 
It is clear that $b$ satisfies (i) and (ii). Recall that we denoted by $K_S^H$ the kernel of $S^H$.
Since $y\mapsto K_S^H(x,y)$ belongs to $L^1(\nu)$
(recall that $\nu$ has bounded density with respect to Lebesgue measure), it follows that
for all $x\in \R^{n+1}$ $S^H(a_k\nu)(x)\to S^H(b\nu)(x)$ as
$k\to\infty$. 
Taking into account that, for every $k$,
$$|S^H(a_k\nu)(x)| \lesssim
\int\frac1{|x-y|^n}\,d\nu(y) <\infty$$
by the dominated convergence theorem we infer that
$$\int |S^H(a_k\nu)|^2 d\nu \to \int |S^H(b\nu)|^2 d\nu\quad\mbox{ as\;
$k\to\infty$.}$$  
Using also that $\|b\|_{L^\infty(\nu)}\leq \limsup_k \|a_k\|_{L^\infty(\nu)}$, it follows 
that $J(b)\leq \limsup_k J(a_k)$, which proves the claim that $J(\cdot)$ attains a minimum at $b$.

\vspace{2mm}

The estimate \rf{eqa331} for $\eta=b\,\nu$ follows from the fact that $J(b)\leq J(1)\leq 2\lambda\,\nu(\R^{n+1})$.

\vspace{2mm}
To prove \rf{eq:DesCTPNu} we will apply a variational argument taking advantage of the fact that $b$ is a minimizer for $J$. Let $B$ be any ball centered in $\supp\eta$. Now, for every $0\leq t<1$, define
\begin{equation}
b_t = (1-t\chi_{B})b + t\,\frac{(h\eta)(B)}{(h\eta)(\R^{n+1})}\,b,
\end{equation}
where we used  the notation $(h\eta)(A)=\int_Ah\,d\eta$. To make the writing easier, we will also write below just
$(h\eta)(A)$.
It is clear that $b_t\in\mathcal{A}$ for all $0\leq t<1$ and $b_0=b$. Therefore,
\begin{equation}
\begin{aligned}
J(b) & \leq J(b_t) = \lambda \|b_t\|_{\infty} \nu(\R^{n+1}) + \int  | S^H(b_t\nu)|^2 b_t\,d\nu 
\\
	&\leq \lambda \left(1+t\frac{(h\eta)(B)}{(h\eta)(\R^{n+1})}\right)\|b\|_{\infty}\nu(\R^{n+1}) + \int  | S^H(b_t\nu)|^2 b_t\, d\nu := H(t).
\end{aligned}
\end{equation}
\par\medskip
Since $H(0)=J(b)$, we have that $H(0)\leq H(t)$ for $0\leq t<1$, thus $H'(0+)\geq 0$
(assuming that $H'(0+)$ exists). Notice that
\begin{equation}
\frac{db_t}{dt}\Big|_{t=0}  = -\chi_{B}b + \frac{(h\eta)(B)}{(h\eta)(\R^{n+1})}\,b,
\end{equation} 
Therefore,
\begin{align}
0 \leq H'(&0+) = \lambda\frac{(h\eta)(B)}{(h\eta)(\R^{n+1})}\,\|b\|_{\infty}\nu(\R^{n+1}) + \frac{d}{dt}\Big|_{t=0}\int  | S^H(b_t\nu)|^2 b_t d\nu \\
	&= \lambda\frac{(h\eta)(B)}{(h\eta)(\R^{n+1})}\|b\|_{\infty}\nu(\R^{n+1}) + 2\int  \!  S^H \left(\frac{db_t}{dt}\Big|_{t=0}\nu\right)\cdot  S^H\eta \,\, b\,d\nu + \!\int  | S^H\eta|^2 \frac{db_t}{dt}\Big|_{t=0}d\nu\\
	&=  \lambda\frac{(h\eta)(B)}{(h\eta)(\R^{n+1})}\,\|b\|_{\infty}\nu(\R^{n+1}) + 2\int   S^H\left(\left(-\chi_{B}b + \frac{(h\eta)(B)}{(h\eta)(\R^{n+1})}b\right)\nu\right)\cdot  S^H\eta \,\, b\,d\nu \\
	& \qquad+ \int  | S^H\eta|^2 \left(-\chi_{B}b + \frac{(h\eta)(B)}{(h\eta)(\R^{n+1})}b\right)d\nu\\
	&=  \lambda\frac{(h\eta)(B)}{(h\eta)(\R^{n+1})}\|b\|_{\infty}\nu(\R^{n+1}) - 2\int    S^H(\chi_{B}\eta)\cdot S^H\eta\,d\eta + 2\frac{(h\eta)(B)}{(h\eta)(\R^{n+1})}\int  | S^H\eta|^2\,d\eta\\
	& \qquad -\int_{B}| S^H\eta|^2 \,d\eta + \frac{(h\eta)(B)}{(h\eta)(\R^{n+1})}\int  | S^H\eta|^2\,d\eta.
\end{align}
The fact that the derivatives above commute with the integral sign and with the
operator $ S^H$ is guaranteed by the fact that $b_t$
is an affine function of $t$ and then one can expand the integrand
$| S^H(b_t\nu)|^2 b_t$ and obtain a polynomial expression on $t$.
Rearranging terms and using also that
$\lambda\leq1$ and that $J(b)\leq2\lambda\,(h\eta)(\R^{n+1})$, we get
\begin{equation}
\begin{split}
\int_B | S^H\eta|^2 \,d\eta + 2\int    S^H(\chi_{B}\eta)\cdot S^H\eta\,d\eta & \leq \frac{(h\eta)(B)}{(h\eta)(\R^{n+1})}\left[\lambda \|b\|_{\infty}\nu(\R^{n+1}) + 3\int  | S^H\eta|^2\,d\eta \right]\\
&\leq 3\,c_7(\delta)^{-1}J(b)\,(h\eta)(B)\leq 6\,c_7(\delta)^{-1}\lambda\,(h\eta)(B).
\end{split}
\end{equation}
Dividing by $\eta(B)$, recalling that $h\leq1$ and taking into account that
$$
\int    S^H(\chi_{B}\eta)\cdot S^H\eta\,d\eta = \int_B  S^{H,*}(( S^H\eta)\eta)\,d\eta,
$$
  we obtain 
\begin{equation}
\frac{1}{\eta(B)}\int_{B} | S^H\eta|^2d\eta + \frac{2}{\eta(B)}\int_B S^{H,*}(( S^H\eta)\eta)\,d\eta \leq 6\,c_7(\delta)^{-1}\lambda.
\end{equation}
Then, letting $\eta(B)\rightarrow 0$ and applying Lebesgue's differentiation theorem, we deduce that
\begin{equation}
| S^H\eta(x)|^2 + 2S^{H,*}(( S^H\eta)\eta)(x) \leq  6\,c_7(\delta)^{-1}\lambda\quad \mbox{ for $\eta$-a.e. $x\in\R^{n+1}$,}
\end{equation}
 as desired.
\end{proof}

\begin{lemm}
\label{Le:PrMax}
Assume that 
$
\displaystyle\int  |S^H\nu|^2\,d\nu \leq \lambda \nu(\R^{n+1})
$
 for some $0<\lambda\leq1$,
and let $b$ and $\eta$ be as in Lemma \ref{Le:DesigualdadAERiNu}. Then we have
\begin{equation}\label{eqd*10}
|S^H\eta(x)|^2 + 4S^{H,*}((S^H\eta)\eta)(x) \leq 12\,c_7(\delta)^{-1}\lambda + C\ell(R)^{\alpha/2}\quad\mbox{ for all $x\in\R^{n+1}_+$.}
\end{equation}
\end{lemm}

\begin{proof}
Since $\eta$ has a bounded density with respect to Lebesgue measure which is also uniformly bounded, it is immediate to
check that the expression on the left hand side of \rf{eqd*10} is a continuous function of $x$. Thus,
by Lemma \ref{Le:DesigualdadAERiNu} and by continuity, the inequality \rf{eqd*10} holds
for all $x\in\supp\eta$.

For  any $x\in\partial\R^{n+1}_+=H$, using 
\rf{eqref56} and that $x=x^*$,
we get 
$$\wh K^H(y^*,x) = \wh K^H(y,x^*) = \wh K^H(y,x),$$
and thus, for  any vectorial measure $\vec{\omega}$,
\begin{align}\label{eqal999}
S^{H,*}\vec{\omega}(x) & = \int K_S^H(y,x) \cdot d\vec{\omega}(y)\\
& = \int \wh K^H(y,x) \cdot d\vec{\omega}(y) - \int \wh K^H(y^*,x) \cdot d\vec{\omega}(y)=0.
\end{align}

\vspace{2mm}

Now we claim that the definition of $S^H$ implies
\begin{equation}\label{ENV.simplemaxprincestimate}
\sup_{x \in \mathbb{R}^{n+1}_+} |S^{H,*}\vec \omega(x)| \leq \sup_{x \in \mathrm{supp}(\vec \omega)} |S^{H,*}\vec \omega(x)|,
\end{equation}
for each vector valued measure $\vec\omega$ which is compactly supported in $\R^{n+1}$ and absolutely continuous with respect to Lebesgue measure with a bounded density function. 
To show this, by the maximum principle, it is enough to show that 
$S^{H,*}\vec \omega$ is $\wh A$-harmonic in $\R^{n+1}_+\setminus \supp(\vec \omega)$. In turn, to this
end it suffices to show that for $1\leq k\leq n$ and for
any signed measure $d\omega = g\,dx$, with 
$g\in L^\infty$ and compactly supported in $\R^{n+1}_+$, 
the function
$$
f(x):= \int \bigl(\partial_{y_k} \EE_{\wh A} (y,x)- \partial_{y_k} \EE_{\wh A} (y^*,x)\bigr) \, d \omega(y)$$
is $\wh A$-harmonic in $\R^{n+1}_+\setminus \supp(\omega)$. 
Given $\varphi \in C^\infty_c(\R^{n+1}_+ \setminus \supp\omega)$, by Fubini's theorem we get
\begin{align}
\int \wh A \,\nabla f \, \nabla \varphi\, dx & =  
\int \wh A(x) \nabla_x \left(\int \partial_{y_k} (\EE_{\wh A}(y,x)
-\EE_{\wh A}(y^*,x))\, g(y) \,dy\right)
\cdot \nabla \varphi(x) \, dx  \\
&=  \int  \!\!\! \int \wh A(x) \nabla_x\partial_{y_k}(\EE_{\wh A}(y,x)
-\EE_{\wh A}(y^*,x)) \cdot \nabla \varphi(x) \, dx \, g(y) \,dy \\
& = \int  \partial_{y_k} \int \wh A(x) \nabla_x\EE_{\wh A}(y,x) \cdot \nabla \varphi(x) \, dx \,  g(y) dy \\
&\quad - 
\int  \partial_{y_k} \int \wh A(x) \nabla_x\EE_{\wh A}(y^*,x) \cdot \nabla \varphi(x) \, dx \,  g(y) dy\\
&= \int  (\partial_{y_k} \varphi(y)-\partial_{y_k}\varphi(y^*)) \,g(y)\,dy= 0.
\end{align} 
Therefore, $f$ is $\wh A$-harmonic $\R^{n+1}_+\setminus \supp(\vec \omega)$ and thus \rf{ENV.simplemaxprincestimate} holds.

To prove \rf{eqd*10} we use the elementary formula
$$
\frac12 |z|^2 = \sup_{\begin{subarray}{c} \beta > 0 \\ e \in \mathbb{R}^{n+1}, \|e\|=1 \end{subarray}} \beta \,\langle e, z \rangle - \frac12 \beta^2\quad \mbox{ for all $z\in\R^{n+1}$.}
$$ 
We apply it with  $z=S^H\eta(x)$ and we get
\begin{equation}
\label{ENV.Tnucuadrado}
\frac12 |S^H\eta(x)|^2 = \sup_{\begin{subarray}{c} \beta > 0 \\ e \in \mathbb{R}^{n+1}, \|e\|=1 \end{subarray}} \beta \,\langle e, S^H\eta(x) \rangle - \frac12 \beta^2.
\end{equation}
Now, if $e=(e_1, \ldots, e_{n+1})$ and we define the vector valued measure $\eta e = (\eta e_1, \ldots, \eta e_{n+1})$, for all 
$x\in\R^{n+1}_+$ we obtain
\begin{align}
\langle e, S^H\eta(x) \rangle & = \int K_S^H(x,y) \cdot e \; d\eta(y) 
 =  \int K_S^H(x,y) \cdot d(\eta e)(y) \\
& =  -S^{H,*}(\eta e)(x) + e\cdot\int \left[ K_S^H(x,y) + K_S^H(y,x) \right]\, d\eta(y).
\end{align}
Taking into account $\wh K^H(y^*,x)= \wh K^H(y,x^*)$ and \rf{remantis} applied to $\wh A$, we derive
\begin{align}
|K_S^H(x,y) + K_S^H(y,x)| &= |\wh K^H(x,y) - \wh K^H(x^*,y) + \wh K^H(y,x) - \wh K^H(y^*,x)|\\
& \leq |\wh K^H(x,y)  + \wh K^H(y,x)| + |\wh K^H(y,x^*) + \wh K^H(x^*,y)| \\
& \lesssim \frac{1}{|x-y|^{n-{\alpha/2}}} + \frac{1}{|x^*-y|^{n-{\alpha/2}}}\lesssim \frac{1}{|x-y|^{n-{\alpha/2}}},
\end{align}
since $|x-y|\leq |x^*-y|$ for all $x,y\in\R^{n+1}_+$.
So the function $F(x) := \int \left[ K_S^H(x,y) + K_S^H(y,x) \right]\, d\eta(y)$ satisfies
$$|F(x)|\lesssim \int \frac{1}{|x-y|^{n-{\alpha/2}}}\,d\eta(y)\lesssim\ell(R)^{\alpha/2}$$
if $\dist(x,R)\leq 1$.
In the case that $\dist(x,Q)\geq1$, we use the fact that 
$|\wh K^H(x,y)|+ |\wh K^H(y,x)|\lesssim1$ by Lemma \ref{lemcz} (c),
and it also follows that
$$|F(x)|\leq \int \left| K_S^H(x,y) + K_S^H(y,x) \right| d\eta(y)
\lesssim \|\eta\|\lesssim \ell(R)^n \lesssim \ell(R)^{\alpha/2},$$
So in both cases we get
\begin{equation}\label{eqcas1*}
\langle e, S^H\eta(x) \rangle = -S^{H,*}(\eta e)(x) + F(x)\cdot e,
\end{equation}
with $|F(x)|\lesssim \ell(R)^{\alpha/2}$.

We insert the above calculation in \eqref{ENV.Tnucuadrado} and by \rf{ENV.simplemaxprincestimate} we get,
for $x\in\R^{n+1}_+$,
\begin{align}
|S^H\eta(x)|^2 &+ 4 S^{H,*}\left( [S^H\eta]\eta \right)(x) \\
& =  \sup_{\begin{subarray}{c} \beta > 0 \\ e \in \mathbb{R}^{n+1}, \|e\|=1 \end{subarray}} \left\{ -2\beta S^{H,*}(\eta e)(x) +  2\beta F(x)\cdot e  -\beta^2 + 4 S^{H,*}\left( [S^H\eta]\eta \right)(x) \right\} \\
& =   \sup_{\begin{subarray}{c} \beta > 0 \\ e \in \mathbb{R}^{n+1}, \|e\|=1 \end{subarray}} \left\{ S^{H,*}\left(-2\beta \eta e + 4[S^H\eta]\eta \right)(x) + 2\beta F(x) \cdot e  -\beta^2 \right\} \\
& \leq  \sup_{\begin{subarray}{c} \beta > 0 \\ e \in \mathbb{R}^{n+1}, \|e\|=1 \end{subarray}} \sup_{z \in \mathrm{supp}(\eta)} \left\{ S^{H,*}\left(-2\beta \eta e + 4[S^H\eta]\eta \right)(z) + 2\beta F(x) \cdot e -\beta^2 \right\} \\
& =   \sup_{z \in \mathrm{supp}(\eta)} \!\sup_{\begin{subarray}{c} \beta > 0 \\ e \in \mathbb{R}^{n+1}, \|e\|=1 \end{subarray}} \left\{ S^{H,*}\left(-2\beta \eta e + 4[S^H\eta]\eta \right)(z) + 2\beta F(x) \cdot e -\beta^2 \right\}. 
\end{align}
 
Now we reverse the process using again \rf{eqcas1*} to obtain
\begin{align}
|&S^H\eta(x)|^2 + 4 S^{H,*}\left( [S^H\eta]\eta \right)(x)  \\
&\leq   \sup_{z \in \mathrm{supp}(\eta)} \sup_{\begin{subarray}{c} \beta > 0 \\ e \in \mathbb{R}^{n+1}, \|e\|=1 \end{subarray}} \!\!\left\{ -2\beta S^{H,*}( \eta e)(z) + 4S^{H,*}\left( [S^H\eta]\eta \right)(z) + 2\beta F(x) \cdot e -\beta^2 \right\} \\
& = \sup_{z \in \mathrm{supp}(\eta)}\!\! \sup_{\begin{subarray}{c} \beta > 0 \\ e \in \mathbb{R}^{n+1}, \|e\|=1 \end{subarray}}  \! \!\!\left\{-2\beta \langle S^H\eta(z),e\rangle -2 \beta F(z)\cdot e + 4S^{H,*}\left( [S^H\eta]\eta \right)(z) + 2\beta F(x) \cdot e -\beta^2  \right\} \\
& =  \sup_{z \in \mathrm{supp}(\eta)} \sup_{\begin{subarray}{c} \beta > 0 \\ e \in \mathbb{R}^{n+1}, \|e\|=1 \end{subarray}}\! \!\!\left\{-2\beta \,\big\langle S^H\eta(z) + (F(x)-F(z)) ,\,e\big\rangle + 4S^{H,*}\left( [T\eta]\eta \right)(z) -\beta^2  \right\} \\
& =  \sup_{z \in \mathrm{supp}(\eta)} \left\{ \left| S^H\eta(z) + (F(x)+G(z))  \right|^2 +4S^{H,*}\left( [S^H\eta]\eta \right)(z)\right\} \\
& \leq  \sup_{z \in \mathrm{supp}(\eta)} \left\{ 2|S^H\eta(z)|^2 + 4 S^{H,*}\left( [S^H\eta]\eta \right)(z)\right\} + C\,\ell(R)^{\alpha/2}.  
\end{align}
Finally, we apply \eqref{eq:DesCTPNu} to get
\begin{equation}
\label{ENV.goodPointwise}
|S^H\eta(x)|^2 + 4S^{H,*}((S^H\eta)\eta)(x) \leq 12\,c_7(\delta)^{-1}\lambda + C\ell(R)^{\alpha/2}\quad\mbox{ for all $x\in\R^{n+1}_+$,}
\end{equation}
as wished.
\end{proof}


\subsection{Proof of Proposition \ref{lemavar}}

Let $R\in\Nice$ and $\nu$ be as in Section \ref{secnou9}. We have to show that
$$\| S^H\nu\|^2_{L^2(\sigma)}\geq c_8(\delta)\,\mu(R),$$
with $c_8(\delta)>0$.
We assume that this does not hold and we argue by contradiction. So we suppose that 
$\int |S^H\nu|^2 d\nu \leq \lambda \,\nu(\R^{n+1})$ for some small $\lambda\in (0,1)$ to be fixed below  and then we will get a contradiction if $\lambda$ is chosen small enough (depending on $\delta$). 
 By Lemma \ref{Le:PrMax}, our assumption implies that the measure $\eta$ defined in Lemma \ref{Le:DesigualdadAERiNu}
satisfies
$$|S^H\eta(x)|^2 + 4S^{H,*}((S^H\eta)\eta)(x) \leq 12\,c_7(\delta)^{-1} \lambda+ C\ell(R)^{\alpha/2}\quad\mbox{ for all $x\in\R^{n+1}_+$.}$$

Consider the vector field $\Psi$ from Lemma \ref{lemeta} in Section \ref{secpsi}. Multiplying  the preceding inequality by $|\Psi|$ and integrating with respect to Lebesgue measure, we derive
\begin{equation}\label{eqfi77}
\int|S^H\eta|^2\,|\Psi|\,d\LL^{n+1} \leq 4\int S^{H,*}\big((S^H\eta)\eta\big)\,|\Psi|\,d\LL^{n+1} + \big(12c_7(\delta)^{-1}\lambda + C\ell(R)^{\alpha/2}\big)
\int |\Psi|\,d\LL^{n+1}.
\end{equation}
By Lemma \ref{lemapsi} we have
$$\big(12c_7(\delta)^{-1}\lambda + C\ell(R)^{\alpha/2}\big)
\int |\Psi|\,d\LL^{n+1}\leq C(\delta)\,\big(\lambda + \ell(R)^{\alpha/2}\big)\,\ell(R)^n.$$
Regarding the first integral on the right hand side of \rf{eqfi77}, we have
\begin{align}
\int S^{H,*}\big((S^H\eta)\eta\big)\,|\Psi|\,d\LL^{n+1} &= \int S^H\eta\cdot S^H(|\Psi|\,\LL^{n+1})\,d\eta\\
& \leq 
\left(\int |S^H \eta|^2\,d\eta\right)^{1/2} \left( 2\int |S^H(|\Psi|\,\LL^{n+1})|^2\,d\nu\right)^{1/2}\\
& \leq \lambda^{1/2}\,\eta(\R^{n+1})^{1/2}\,C(\delta)\,\mu(R) \leq C(\delta)\,\lambda^{1/2}\,\mu(R),
\end{align}
by \rf{eqa331} and (v) from Lemma \ref{lemapsi}. So we derive
\begin{equation}\label{eqfi770}
\int|S^H\eta|^2\,|\Psi|\,d\LL^{n+1} \!\leq C(\delta)\,\lambda^{1/2}\mu(R)+
C(\delta)\,(\lambda + \ell(R)^{\alpha/2})\,\mu(R)\!\leq C(\delta)\,(\lambda^{1/2} + \ell(R)^{\alpha/2})\,\mu(R).
\end{equation}

Next we will estimate from below the integral on the left hand side above. By Cauchy-Schwarz,
we have
\begin{equation}\label{eqsh100}
\begin{split}
\int|S^H\eta|^2\,|\Psi|\,d\LL^{n+1} & \geq \left(\int |S^H \eta|\,|\Psi|\,d\LL^{n+1} \right)^2\,\left(\int |\Psi|\,d\LL^{n+1}\right)^{-1}\\
& \geq \frac {c(\delta)}{\mu(R)}
\left(\int S^H \eta\cdot\Psi\,d\LL^{n+1} \right)^2 = \frac {c(\delta)}{\mu(R)}
\left(\int S^{H,*}(\Psi\,\LL^{n+1})\,d\eta \right)^2.
\end{split}
\end{equation}
By the definition of $S^H$ and the fact that $\wh K^H(y^*,x) = \wh K^H(y,x^*)$ (by \rf{eqref56}), we get
\begin{align}
S^{H,*}(\Psi\,\LL^{n+1})(x) & = \int \wh K^H(y,x) \cdot\Psi(y)\, d\LL^{n+1}(y) - \int \wh K^H(y^*,x) \cdot\Psi(y)\, d\LL^{n+1}(y)\\
& = \wh T^{H,*}(\Psi\,\LL^{n+1})(x) - \wh T^{H,*}(\Psi\,\LL^{n+1})(x^*).
\end{align}
Thus, by Lemma \ref{lemapsi} (iv),
\begin{align}\label{eqsh10}
\int S^{H,*}(\Psi\,\LL^{n+1})\,d\eta &= \int\wh T^{H,*}(\Psi\,\LL^{n+1})(x)\,d\eta(x) - \int\wh T^{H,*}(\Psi\,\LL^{n+1})(x^*)\,d\eta(x) \\ 
&=\sum_{Q\in\NB'(R)} \int (h_Q + e_Q)\,d\eta - \int\wh T^{H,*}(\Psi\,\LL^{n+1})(x^*)\,d\eta(x).
\end{align}

By Lemma \ref{lemeta} and Lemma \ref{lemapsi} (iv),
$$\sum_{Q\in\NB'(R)} \int (h_Q + e_Q)\,d\eta \geq c(\delta)\,\mu(R)- C(\delta)\sum_{Q\in\NB'(R)}\int
\frac{\ell(R)^{\tilde\gamma}\,\ell(Q)^{n+\tilde\beta}}{(|x-x_Q|+\ell(Q))^{n+\tilde\beta}} \,d\eta(x).$$
Using the polynomial growth of $\nu$ (recall \eqref{lemgrowthnu}) and standard estimates, for each 
$Q\in\NB'(R)$ we get 
\begin{equation}\label{eqeq*4}
\int
\frac{\ell(R)^{\tilde\gamma}\,\ell(Q)^{n+\tilde\beta}}{(|x-x_Q|+\ell(Q))^{n+\tilde\beta}} \,d\eta(x)\lesssim 
\ell(R)^{\tilde\gamma}\,\ell(Q)^n.
\end{equation}
Thus
\begin{equation}\label{eqnb56}
\sum_{Q\in\NB'(R)} \int (h_Q + e_Q)\,d\eta \geq c(\delta)\,\mu(R)- C(\delta)\,\ell(R)^{\tilde\gamma}\!\!\sum_{Q\in\NB'(R)}\!
\mu(Q) \geq \bigl(c(\delta)- C'(\delta)\,\ell(R)^{\tilde\gamma}\bigr)\,\mu(R).
\end{equation}

To estimate the last integral on the right hand side of \rf{eqsh10} we take into account that,
if $x\in\supp\eta$, then $x^*\in \R^{n+1}_-$, and thus
$$h_Q(x^*)=0\quad \mbox{ for all $Q\in\NB'(R)$,}$$ 
since $\supp h_Q\subset 3B_Q\subset \R^{n+1}_+$ because, recalling \eqref{Delta} and choosing $\Delta$ as in Section \ref{choosingdelta}, $\ell(Q)\ll \Delta\,\ell(R)$. 

Therefore, for $x\in\supp\eta$, using again Lemma \ref{lemapsi} (iv),
\begin{align}
\bigl|\wh T^{H,*}(\Psi\,\LL^{n+1})(x^*)\bigr|  = \Biggl|\sum_{Q\in\NB'(R)} e_Q(x^*)\Biggr|
 &\leq
 C(\delta)\sum_{Q\in\NB'(R)}
\frac{\ell(R)^{\tilde\gamma}\,\ell(Q)^{n+\tilde\beta}}{(|x^*-x_Q|+\ell(Q))^{n+\tilde\beta}}\\
&\leq 
 C(\delta)\sum_{Q\in\NB'(R)}
\frac{\ell(R)^{\tilde\gamma}\,\ell(Q)^{n+\tilde\beta}}{(|x-x_Q|+\ell(Q))^{n+\tilde\beta}}.
\end{align}
Hence, from \rf{eqeq*4} we derive
$$\left|\int\wh T^{H,*}(\Psi\,\LL^{n+1})(x^*)\,d\eta(x)\right|\leq  C(\delta)\,\ell(R)^{\tilde\gamma}\,\sum_{Q\in\NB'(R)}
\ell(Q)^n\leq C(\delta)\,\ell(R)^{\tilde\gamma}\mu(R).$$
Plugging this estimate and \rf{eqnb56} into \rf{eqsh10}, we obtain
$$\int S^{H,*}(\Psi\,\LL^{n+1})\,d\eta\geq \bigl(c(\delta)- C''(\delta)\,\ell(R)^{\tilde\gamma}\bigr)\,\mu(R).$$
Then, by \rf{eqsh100},
$$\int|S^H\eta|^2\,|\Psi|\,d\LL^{n+1}  \geq \bigl(c(\delta)- C''(\delta)\,\ell(R)^{\tilde\gamma}\bigr)^2\,
\mu(R).
$$
Together with \rf{eqfi770}, this implies that
$$\bigl(c(\delta)- C''(\delta)\,\ell(R)^{\tilde\gamma}\bigr)^2
\mu(R)\leq C(\delta)\,\big(\lambda^{1/2} + \ell(R)^{\alpha/2}\big)\,\mu(R).
$$
So we get a contradiction if $\ell(R)$ and $\lambda$ are small enough, depending on $\delta$.
This concludes the proof of Proposition \ref{lemavar}, and thus of Theorem \ref{teo1}.


\section{Proof of Theorem \ref{teo3}} \label{sec12*}

The arguments are very similar to the ones in \cite{AHM3TV} and thus we only sketch them.

To simplify notation, we will write $\omega^p$ instead of $\omega_{L_A}^p$.
Recall that the Green function for the operator $L_A u =-\Div A(\cdot)\nabla u$ satisfies, for 
every $\varphi\in C_c^\infty(\R^{n+1})$,
$$\int_{\partial\Omega} \varphi\,d\omega^x - \varphi(x) = - \int_\Omega A^T(y)\nabla_yG(x,y)\cdot \nabla\varphi(y)\,dy,
\quad\mbox{ for a.e. $x\in\Omega$.}$$
See (2.6) in \cite{AGMT}, for example. 
From this equation it easily follows that
\begin{equation}\label{eqgreen1}
G(p,x) =\EE (p,x) - \int \EE(z,x)\,d\omega^p(z)\quad\mbox{for all $p,x\in\Omega$.}
\end{equation}	
We assume that $G(p,x)=0$ if $x\not\in\Omega$, so that the preceding
identity also holds in this case.
The identity \rf{eqgreen1} provides the key connection between the gradient of the single layer potential and elliptic measure.
Indeed, differentiating with respect to $x$, we derive
$$\nabla_2 G(p,x) = \nabla_2\EE (p,x) - \int \nabla_2\EE(z,x)\,d\omega^p(z).$$

By almost the same arguments as in \cite[Lemma 3.3]{AHM3TV} one can prove the following:

\begin{lemm}\label{l:w>G}
Let $n\ge 2$ and $\Omega\subset\R^{n+1}$ be a bounded open connected Wiener regular set.
Let $B=\bar B(x_0,r)$ be a closed ball with $x_0\in\partial\Omega$ and $0<r<\diam(\partial \Omega)$. Then, for all $a>0$,
\begin{equation}\label{eq:Green-lowerbound}
 \omega^{x}(aB)\gtrsim \inf_{z\in 2B\cap \Omega} \omega^{z}(aB)\, r^{n-1}\, G(x,y)\quad\mbox{
 for all $x\in \Omega\setminus  2B$ and $y\in B\cap\Omega$,}
 \end{equation}
 with the implicit constant independent of $a$.
\end{lemm}

Analogously, as in \cite[Lemma 3.4]{AHM3TV}, we have:

\begin{lemm}
\label{lembourgain}
There is $\delta_{0}\in(0,1)$ depending only on $n\geq 1$ so that the following holds for $\delta\in (0,\delta_{0}]$. Let $\Omega\subsetneq \R^{n+1}$ be a  bounded Wiener regular domain, $n-1<s\le n+1$,  $\xi \in \partial \Omega$, $r>0$, and $B=B(\xi,r)$. Then 
\[ \omega^{x}(B)\gtrsim_{n,s} \frac{\mathcal \HH_\infty^{s}(\partial\Omega\cap \delta B)}{(\delta r)^{s}}\quad \mbox{  for all }x\in \delta B\cap \Omega .\]
\end{lemm}

In the statement above, $\HH_\infty^s$ stands for the $s$-dimensional Hausdorff content.

 The following can be proved as in  \cite[Lemma 3.1]{AHM3TV}:

\begin{lemm}\label{lemgreen*}
Let $\Omega$ be as above and let $p\in\Omega$. For $\LL^{n+1}$-almost all $x\in\Omega^c$ we have
\begin{equation}\label{eqdf12}
\mathcal{E}(p,x) - \int_{\partial\Omega} \mathcal{E}(z,x)\,d\omega^p(z)=0.
\end{equation}
\end{lemm}

Then we get:

\begin{lemm}\label{lemaxcor}
Let $L_A$, $\Omega$ and $E$ be as in Theorem \ref{teo3}. Then we have
\begin{equation}\label{eqdk10}
M_n\omega^p (x) + T_{*}\omega^p(x) <\infty \quad\mbox{ for $\omega^p$-a.e.\ $x\in E$.}
\end{equation}
\end{lemm}

Above, $M_n$ is the maximal radial operator defined in \rf{maximalradial}.

This result can be deduced from the
preceding lemmas arguing as in \cite{AHM3TV}. For  the convenience of the reader we show the detailed proof below. Remark that, instead of the stopping time arguments
from \cite{AHM3TV}, we use a simpler approach relying on the Lebesgue differentiation theorem.

\begin{proof}
For $\omega^p$-a.e.\ $x\in E$, we write
$$\limsup_{r\to 0} \frac{\omega^p(B(x,r))}{r^n} \leq \limsup_{r\to 0} \frac{\omega^p(B(x,r))}{\HH^n(B(x,r)\cap E)}\,\,
\limsup_{r\to 0} \frac{\HH^n(B(x,r)\cap E)}{r^n}.$$
The first $\limsup$ on the right hand side is finite $\omega^p$-a.e.\ in $E$ because of the absolute continuity of  
$\omega^p$ with respect to $\HH^n$ in $E$, while the last one is also finite by the classical density bounds for Hausdorff measure. Hence the left hand side is also finite  $\omega^p$-a.e.\ in $E$, or equivalently,
$$M_n\omega^p (x)  <\infty \quad\mbox{ for $\omega^p$-a.e.\ $x\in E$.}$$

It remains to show that $T_{*}\omega^p(x) <\infty$ for $\omega^p$-a.e.\ $x\in E$.
To this end, for $k\geq 1$ we define
$$E_k = \{x\in E: M_n\omega^p(x)\leq k\},$$
so that $E=\bigcup_{k\geq 1} E_k$, up to a set of $\omega^p$-measure zero. For a fixed $k\geq1$, let $x\in E_k$ be a density point of $E_k$, and let $r_0$ be small enough so that
$$\frac{\omega^p(B(x,r)\cap E_k)}{\omega^p(B(x,r))} \geq \frac12\quad\mbox{ for $0<r\leq r_0$.}$$
Observe that, since $\omega^p(B(z,\rho)\cap E_k)\leq k \rho^n$ for all $z\in E_k$ and all $\rho>0$,
by Frostman's Lemma we have
\begin{equation}\label{eqfrostman}
\HH_\infty^n(B(x,r)\cap\partial\Omega )\geq \HH_\infty^n(B(x,r)\cap E_k)\geq C(k)\, \omega^p(B(x,r)\cap E_k)\geq \frac{C(k)}2\, \omega^p(B(x,r)),
\end{equation}
for $0<r\leq r_0$.

Next we consider  a radial $C^\infty$ function
 $\varphi:\R^{n+1}\to[0,1]$  which vanishes in $B(0,1)$ and equals $1$ on $\R^{n+1}\setminus B(0,2)$,
and for $r>0$ and $z\in \R^{n+1}$ we denote
$\varphi_r(z) = \varphi\left(\frac{z}r\right) $ and $\psi_r = 1-\varphi_r$.
We set
$$
\wt T_r\omega^p(z) =\int \nabla_2\EE(y,z)\,\varphi_r(z-y)\,d\omega^p(y).$$
Note that, by Lemma \ref{lemm_freezing},
\begin{equation}\label{eqkey**}
\begin{split}
|T_r\omega^p(x)| & \leq \left|\int \varphi(x-y) \nabla_2\EE(y,x)\,d\omega^p(y)\right|  
+ \int \big|\chi_{|x-y|>r}-\varphi(x-y)\big| \,\big|\nabla_2\EE(y,x)\big|\,d\omega^p(y)\\
& \quad +
\int_{|x-y|>r} \big|\nabla_1\EE(x,y) - \nabla_2\EE(y,x)\big|\,d\omega^p(y)\\
& \leq \wt T_r\omega^p(x) + C\,M_n\omega^p(x) + \int \frac{C}{|x-y|^{n-\alpha}}\,d\omega^p(y)\\
& \leq \wt T_r\omega^p(x)+ C\,M_n\omega^p(x).
\end{split}
\end{equation}

To estimate $\wt T_r\omega^p(x)$, first we assume that
\begin{equation}\label{lemdoub**}
\omega^{p}(B(x,2\delta_0^{-1}r))\leq 2\delta_0^{-(n+1)}\omega^p(B(x,2r)),
\end{equation}
with $\delta_0$ as in Lemma \ref{lembourgain}.
For a fixed $x \in E_k$ and $z\in \R^{n+1}\setminus \bigl[\supp(\varphi_r(x-\cdot)\,\omega^p)\cup \{p\}\bigr]$, consider the function
\begin{equation}\label{lemufundsol*}
u_r(z) = \EE(p,z) - \int \EE(y,z)\,\varphi_r(x-y)\,d\omega^p(y),
\end{equation}
so that, by \rf{eqgreen1} and Lemma \ref{lemgreen*},
\begin{equation}\label{eqfj33}
G(p,z) = u_r(z) - \int \EE(y,z)\,\psi_r(x-y)\,d\omega^p(y)\quad \mbox{ for $\LL^{n+1}$-a.e.\   $z\in\R^{n+1}$.}
\end{equation}
Differentiating \eqref{lemufundsol*} with respect to $z$, we obtain
$$\nabla u_r(z) = \nabla_2 \EE(p,z) - \int \nabla_2\EE(y,z)\,\varphi_r(x-y)\,d\omega^p(y)
.$$
In the particular case $z=x$ we get (using also the H\"older continuity of $u_r$)
$$\nabla u_r(x) = \nabla_2 \EE(p,x)  - \wt T_r\omega^p(x),$$
and thus
\begin{equation}\label{eqcv1}
|\wt T_r\omega^p(x)|\lesssim \frac1{\dist(p,\partial\Omega)^n} + |\nabla u_r(x)|.
\end{equation}

Since $u_r$ is $L_{A^T}$-harmonic in $\R^{n+1}\setminus \bigl[\supp(\varphi_r(x-\cdot)\,\omega^p)\cup \{p\}\bigr]$ (and so in $B(x,r)$) and $A$ is H\"older continuous, using Moser's Harnack inequality,
we have
\begin{equation}\label{eqcv2}
|\nabla u_r(x)| \lesssim \frac1r\,\left(\avint_{B(x,r/2)}|u_r(z)|^2\,dz\right)^{1/2}\lesssim
\frac1r\,\avint_{B(x,r)}|u_r(z)|\,dz.
\end{equation}
From the identity \rf{eqfj33} we deduce that
\begin{align}\label{eqcv3}
|\nabla u_r(x)| &\lesssim \frac1r\,\avint_{B(x,r)}G(p,z)\,dz + 
\frac1r\,\avint_{B(x,r)}
\int \EE(y,z)\,\psi_r(x-y)\,d\omega^p(y)\,dz \nonumber\\
& =:I + II.
\end{align}
To estimate the term $II$ we use Fubini and the fact that $\supp\psi_r\subset B(x,2r)$:
\begin{align}
II & \lesssim \frac1{r^{n+2}}\, \int_{y\in B(x,2r)}\int_{z\in B(x,r)} \frac 1{|z-y|^{n-1}} \,dz\,d\omega^p(y)\\
& \lesssim \frac{\omega^p(B(x,2r))}{r^{n}} \lesssim M_n\omega^p(x).
\end{align}

We want to show now that $I\lesssim_k 1$.
Clearly it is enough to show that
\begin{equation}\label{eqsuf1}
\frac1r\,| G(p,y)|\lesssim_k 1\qquad\mbox{for all $y\in  B(x,r)\cap\Omega$}
\end{equation}
(still under the assumptions $x\in E_k$, $0<r\leq r_0/2$, and \rf{lemdoub**}).
To prove this, observe that by Lemma \ref{l:w>G} (with $B= B(x,r)$, $a=2\delta_0^{-1}$), for all $y\in B(x,r)\cap\Omega$,
we have
$$\omega^{p}(B(x,2\delta_0^{-1}r))\gtrsim \inf_{z\in B(x,2r)\cap \Omega} \omega^{z}(B(x,2\delta_0^{-1}r))\, r^{n-1}\,|G(p,y)|.$$
On the other hand, by Lemma \ref{lembourgain} and \rf{eqfrostman}, for any $z\in B(x,2r)\cap\Omega$ and 
$0<r\leq r_0/2$,
$$\omega^{z}(B(x,2\delta_0^{-1}r))\gtrsim \frac{\HH^n_\infty(B(x,2r)\cap\partial\Omega)}{r^n}\gtrsim C(k)\frac{\omega^p(B(x,2r))}{r^n}.$$
Therefore we have
$$
\omega^{p}(B(x,2\delta_0^{-1}r))
\gtrsim C(k)
\frac{\omega^p(B(x,2r))}{r^n}\, r^{n-1}\,|G(p,y)|,
$$
and thus, by \rf{lemdoub**},
$$
\frac1r\,| G(p,y)|\lesssim_k \frac{\omega^{p}(B(x,2\delta_0^{-1}r))}{\omega^p(B(x,2r))}\lesssim_k 1,
$$
which proves \rf{eqsuf1}. So we deduce that
\begin{equation}\label{eqry33}
|\wt T_r\omega^p(x)|\lesssim_k \frac1{\dist(p,\partial\Omega)^n} + 1
\end{equation}
for $x\in E_k$ and $0<r\leq r_0/2$ satisfying \rf{lemdoub**}.

In the case where \rf{lemdoub**} does not hold, we consider the largest $s>0$ of the form
$s= 2\delta_0^{j} r$, $j>0$, such that \rf{lemdoub**} holds with $s$ replacing $r$. By standard methods 
from non-doubling Calder\'on-Zygmund theory, it follows that such $s$ exists for $\omega^p$-a.e.\ $x\in E_k$ and moreover
$$|\wt T_r\omega^p(x)| \leq |\wt T_s\omega^p(x)| + C\,M_n \omega^p(x).$$
See, for example, Lemmas 2.8 and 2.20 from \cite{Tolsa-llibre}.
Then, applying \rf{eqry33} with $r=s$, we infer that
$$|\wt T_r\omega^p(x)| \lesssim_k \frac1{\dist(p,\partial\Omega)^n} + 1 + M_n \omega^p(x) 
\lesssim _k \frac1{\dist(p,\partial\Omega)^n} + 1.$$
So in any case we deduce that  $|\wt T_r\omega^p(x)|$ is bounded uniformly for $\omega^p$-a.e.\ $x\in E_k$
and $r$ small enough. By \rf{eqkey**}, this implies that the same holds for $|T_r\omega^p(x)|$, and thus it follows that $T_*\omega^p(x)<\infty$ for $\omega^p$-a.e.\ $x\in E_k$, and so for $\omega^p$-a.e.\ $x\in E$, as wished.
\end{proof}

From the preceding lemma and \rf{remantis} we deduce that the antisymmetric operator $T^{(a)}$ satisfies
$$T_{*}^{(a)}\omega^p(x) \leq M_n\omega^p (x) + T_{*}\omega^p(x)<\infty.$$
Next we apply the following $Tb$ type theorem due to Nazarov, Treil and Volberg \cite{NTV}, \cite{Volberg}
 in combination with the methods in \cite{Tolsa-pams}. For the detailed proof in the case of the Cauchy
transform, see \cite[Theorem 8.13]{Tolsa-llibre}. The same arguments with very minor modifications work for antisymmetric operators.

\begin{theor}  \label{teo**}
Let $\sigma$ be a Radon measure with compact support on $\R^{n+1}$   and consider a $\sigma$-measurable set
$G$ with $\sigma(G)>0$ such that
$$G\subset\{x\in \R^{n+1}: 
M_n\sigma(x) < \infty \mbox{ and } \,T^{(a)}_* \sigma(x) <\infty\}.$$
Then there exists a Borel subset $G_0\subset G$ with $\sigma(G_0)>0$  such
that $\sup_{x\in G_0}M_n\sigma|_{G_0}(x)<\infty$ and  $T^{(a)}_{\sigma|_{G_0}}$ is bounded in $L^2(\sigma|_{G_0})$.
\end{theor}

Applying this theorem to the measure $\sigma=\omega^p$ and the set $G=E$, we infer that there exists 
a subset $G_0\subset E$ with $\omega^p(E)>0$ such that $T^{(a)}_{\omega^p|_{G_0}}$ is bounded
in $L^2(\omega^p|_{G_0})$. Then, by Lemma \ref{lemantisym} it turns out that $T_{\omega^p|_{G_0}}$ is also bounded
in $L^2(\omega^p|_{G_0})$.
Since $\omega^p$ is absolutely continuous with respect to $\HH^n$ on $G_0$, by applying
Theorem \ref{teo2} we deduce that $G_0$ is $n$-rectifiable. Now, by a standard exhausting argument we deduce that $\omega^p$ is concentrated in an $n$-rectifiable set and thus $\omega^p$ is $n$-rectifiable.


\label{Bibliography}

\end{document}